\documentclass{memo-l}

\usepackage{amssymb}
\usepackage{stackengine}[2013-09-11]
\usepackage{graphicx}
\usepackage{xcolor}
\usepackage{subcaption}
\usepackage{cancel}
\usepackage{multirow}
\usepackage{mathrsfs}
\usepackage{tikz-cd}
\usepackage{enumitem}
\usepackage{float}
\usepackage{setspace}
\usepackage{appendix}
\usepackage{hyperref}

\newcommand{\Mm}{\mathcal{M}}

\newcommand{\Oo}{\mathcal{O}}

\newcommand{\jj}{\mathcal{J}}
\newcommand{\J}{\mathcal{J}}

\newcommand{\T}{\mathbb{T}^2}

\newcommand{\N}{\mathbb{N}}
\newcommand{\Z}{\mathbb{Z}}
\newcommand{\Q}{\mathbb{Q}}
\newcommand{\R}{\mathbb{R}}
\newcommand{\C}{\mathbb{C}}

\newcommand{\om}{\omega}

\newcommand{\joml}{\J_{\om_\lambda}}
\newcommand{\oml}{\omega_\lambda}

\newcommand{\del}{\partial}
\newcommand{\delbar}{\bar\partial}
\newcommand{\dd}[1]{\frac{\partial}{\partial #1}}

\newcommand{\from}{\leftarrow}
\newcommand{\into}{\hookrightarrow}
\newcommand{\lto}{\longrightarrow}
\newcommand{\lmapsto}{\longmapsto}

\newcommand{\lie}[1]{\mathfrak{#1}}
\newcommand{\near}{\mathrm{Near}}
\DeclareRobustCommand\longtwoheadrightarrow
     {\relbar\joinrel\twoheadrightarrow}

\newcommand{\SSS}{S^2 \times S^2}

\newcommand{\CP}{\mathbb{C}P}
\newcommand{\CCC}{\CP^2\# \overline{\CP^2}}
\newcommand{\jsoml}{\mathcal{J}^{S^1}_{\om_\lambda, l}}
\newcommand{\jsom}{\mathcal{J}^{S^1}_{\om_\lambda}}

\DeclareMathOperator{\Diff}{Diff}

\DeclareMathOperator{\Symp}{Symp}
\DeclareMathOperator{\Sym}{Sym}

\DeclareMathOperator{\Ham}{Ham}
\DeclareMathOperator{\Stab}{Stab}
\DeclareMathOperator{\Fix}{Fix}
\DeclareMathOperator{\Hom}{Hom}
\DeclareMathOperator{\Aut}{Aut}

\DeclareMathOperator{\End}{End}
\DeclareMathOperator{\Maps}{Maps}

\DeclareMathOperator{\AGL}{AGL}
\DeclareMathOperator{\Sp}{Sp}
\DeclareMathOperator{\lin}{lin}
\DeclareMathOperator{\U}{U}
\DeclareMathOperator{\SO}{SO}
\DeclareMathOperator{\PU}{PU}
\DeclareMathOperator{\PSL}{PSL}
\DeclareMathOperator{\SU}{SU}

\DeclareMathOperator{\hocolim}{hocolim}
\DeclareMathOperator{\Ker}{Ker}
\DeclareMathOperator{\Coker}{Coker}
\DeclareMathOperator{\pushout}{pushout}
\DeclareMathOperator{\Met}{Met}

\DeclareMathOperator{\image}{Im}

\DeclareMathOperator{\id}{id}
\DeclareMathOperator{\rk}{rk}
\DeclareMathOperator{\ev}{ev}

\DeclareMathOperator{\Det}{Det}

\DeclareMathOperator{\im}{im}

\newtheorem{thm}{Theorem}[chapter]
\newtheorem{prop}[thm]{Proposition}
\newtheorem{lemma}[thm]{Lemma}
\newtheorem{cor}[thm]{Corollary}

\theoremstyle{definition}
\newtheorem{defn}[thm]{Definition}

\theoremstyle{remark}
\newtheorem{remark}[thm]{Remark}

\numberwithin{section}{chapter}
\numberwithin{equation}{chapter}

\begin{document}

%
%

\title[Centralizers of Hamiltonian circle actions]{Centralizers of Hamiltonian circle actions\\ on rational ruled surfaces}

\author[P. Chakravarthy]{Pranav.V.Chakravarthy}
\address{Department of Mathematics \\ Universit\'e Libre de Bruxelles\\ Belgium}
\email{pranav.vijay.chakravarthy@ulb.be}
\thanks{The paper contains results from the first author's PhD thesis at the University of Western Ontario under the supervision of the second author. The first author would like to thank Yael Karshon and Mohan Swaminathan for helpful discussions during the completion of the project, and Weimin Chen for sharing the manuscript~\cite{ChenUnpub}. The first author would also like to thank the Hebrew University of Jerusalem where this project was completed.}

\author[M. Pinsonnault]{Martin Pinsonnault}
\address{Department of Mathematics\\  The University of Western Ontario\\ Canada}
\email{mpinson@uwo.ca}
\thanks{The second author is supported by the Natural Science and Engineering Research Council of Canada.}


\date{\today}

\begin{abstract}
We compute the homotopy type of the group of equivariant symplectomorphisms of $\SSS$ and $\CCC$ under the presence of a Hamiltonian action of the circle $S^1$. We prove that the group of equivariant symplectomorphisms is homotopy equivalent to either a torus, or to the homotopy pushout of two tori depending on whether the circle action extends to a single toric action or to exactly two non-equivalent toric actions. This follows from the analysis of the action of equivariant symplectomorphisms on the space of compatible and invariant almost complex structures $\J^{S^1}_{\om}$. In particular, we show that this action preserves a decomposition of $\J^{S^1}_{\om}$ into strata which are in bijection with toric extensions of the circle action. Our results rely on $J$-holomorphic techniques, on Delzant's classification of toric actions, and on Karshon's classification of Hamiltonian circle actions on $4$-manifolds.
\end{abstract}

\maketitle

\tableofcontents

\chapter{Introduction} 

{\it Background.} Let $(M, \omega)$ be a compact symplectic manifold without boundary. The group $\Symp(M, \omega)$ of symplectomorphisms, equipped with the standard $C^\infty$-topology, is an infinite dimensional Fréchet Lie group. Although the isomorphism type of $\Symp(M,\om)$ as a discrete group determines the conformal symplectomorphism type of $(M,\om)$ (see Banyaga~\cite{Banyaga-Isomorphism}), the ways algebraic and topological properties of $\Symp(M,\om)$ relate to the geometry of $(M,\om)$ remain largely unknown. An interesting situation to consider is when $(M,\om)$ admits symplectic or Hamiltonian actions of some compact Lie group $G$ (possibly finite). Viewing effective actions as injective homomorphisms $G\into\Symp(M,\om)$, two actions are considered equivalent if the corresponding two subgroups are in the same conjugacy class. By analogy with the theory of finite dimensional Lie groups, it is then natural to consider the following questions:
\begin{enumerate}[label=\textbf{Q\arabic*.}]
\item Can we list the conjugacy classes of $G$ subgroups in $\Symp(M,\om)$\,?
\item Can we characterize the centralizer $C(G)$ and normalizer $N(G)$\,? 
\item Can we understand the inclusions 
\[G\into \Symp(M,\om)\quad \text{~and~}\quad C(G)\into N(G)\into \Symp(M,\om)\]
from a homotopy theoretic point of view\,?
\end{enumerate}

The first question amounts to classifying symplectic or Hamiltonian $G$-actions on a given symplectic manifold $(M,\om)$ up to conjugation and reparametrization. It is important to distinguish this problem from \emph{equivariant} classifications whose aim is to classify \emph{pairs} \{symplectic manifold, $G$-action\} up to equivariant symplectomorphisms, and which usually do not directly produce the list of $G$-actions on a given manifold $(M,\om)$. In the symplectic category, examples of equivariant classifications include Delzant's classification of Hamiltonian toric actions~\cite{Delzant}, Karshon-Tolman's classification of tall complexity one Hamiltonian torus actions~\cite{KT-TallComplexityOne}, Iglesias' classification of Hamiltonian $\SO(3)$ actions~\cite{Iglesias1991}, and Karshon's classification of Hamiltonian circle actions on $4$-manifolds. On the other hand, satisfying answers to Question 1 have only been given in dimension $2$ and $4$. The $2$-dimensional case is classical: the only surface that admits a Hamiltonian action of a continuous Lie group $G$ is $S^2$, and all finite subgroups of $\Symp(\Sigma^2,\sigma)$ are K\"ahler isometries of a compatible complex structure. In dimension $4$, the diffeomorphism types of Hamiltonian $S^1$, $\T$, or $\SO(3)$ manifolds are known, and one can devise algorithms that produce the list of inequivalent $G$-actions on a given manifold $(M^4,\om)$ from the equivariant classifications mentioned above, see~\cite{P-MaxTori, KKP, KK-Counting}.\\

The second and third questions on the properties of centralizers and normalizers subgroups can be fully answered for all Hamiltonian toric actions. Let $\mathbb{T}^n\subset\Ham(M^{2n},\om)$ acts effectively on $(M^{2n},\om)$ with moment map $\mu:M^{2n}\to\R^n$. Using standard moment map techniques, it was shown in~\cite{P-MaxTori} that 
\begin{enumerate}
\item the centralizer $C(\mathbb{T}^n)$ is equal to the group of all symplectomorphisms $\phi$ that preserve the moment map, that is, such that $\mu \circ \phi=\mu$.
\item $C(\mathbb{T}^n)$ is a maximal torus in $\Symp(M^{2n},\om)$, that is, a maximal, connected, and abelian subgroup of $\Symp(M^{2n},\om)$. In particular, since toric manifolds are simply-connected, $C(\mathbb{T}^n) \subset\Ham(M^{2n},\omega)$. 
\item $C(\mathbb{T}^n)$ deformation retracts onto $\mathbb{T}^{n}$. In particular, the homotopy type of the centraliser of a toric action only depends on the dimension of the torus and is independent of the symplectic manifold $(M,\om)$ and of the specific toric action considered.
\item The Weyl group $W(\mathbb{T}^n):=N(\mathbb{T}^n)/C(\mathbb{T}^n)$ is always finite\footnote{Furthermore, as for maximal tori in compact Lie groups, it can be shown that the number of conjugacy classes of toric centralizers is finite, and that each $C(\mathbb{T}^n)$ is flat and totally geodesic in $\Symp(M^{2n},\om)$ for the $L^2$  metric.}.
\end{enumerate}
Moreover, it was shown by McDuff and Tolman in~\cite{McDuff-Tolman-MassLinear1} that the map $\pi_1(T^n)\into\pi_1(\Ham(M^{2n},\om))$ induced by a toric action fails to be injective only in the rare cases the moment polytope $\Delta=\mu(M^{2n})$ satisfies one of two special conditions.  Consequently, for a generic toric action, the inclusion maps
\[T^n\into C(T^n)=(N(T^n))_0\into\Ham(M^{2n},\om)=\Symp_0(M^{2n},\om)\]
are injective in homotopy. To our knowledge, beside the toric case, no other general results on $C(G)$ and $N(G)$ are known.\\

Regarding the homotopy type of symplectomorphism groups, most efforts as been devoted to rational $4$-manifolds. Following the seminal work of M. Gromov~\cite{Gr} who showed that the group of compactly supported symplectomorphisms of $\R^4$ is contractible, the homotopical properties of the group of symplectomorphisms of $\CP^2$, $\SSS$ and of the $k$-fold symplectic blow-ups $\CP^2\#k\overline{\CP}^2$, $k\leq 5$, were studied in several papers such as \cite{abreu}, \cite{AGK}, \cite{MR1775741}, \cite{P-com}, \cite{AG}, \cite{AP}, \cite{AE}, and~\cite{LiLiWu2022}. In particular, for $\CP^2$, $\SSS$, and $\CP^2\#k\overline{\CP^2}$, $k\leq 3$, the rational homotopy type of $\Symp(M,\om)$ can be described precisely in terms of the cohomology class $[\om]$. For $k\geq 4$, partial results are known, including various bounds on the ranks of $\pi_0(\Symp(M,\om))$ and $\pi_1(\Symp(M,\om))$. As a byproduct, combining this with our understanding of Hamiltonian actions on rational $4$ manifolds of small Euler number, this allows us to understand the subrings of $\pi_*(\Symp(M,\om))$ or $H_*(\Symp(M,\om))$ that the various Hamiltonian actions on $(M^4,\om)$ generate.\\

{\it Main results.} In the present paper, we combine pseudo-holomorphic curve techniques with moment map techniques to determine the homotopy type of centralizers of Hamiltonian $S^1$ actions on $\SSS$ and $\CCC$. Our main results are Theorem~\ref{full_homo} and Theorem~\ref{full_homo_CCC} that give the full homotopy type of the centralizer $C(S^1)\subset\Symp(M,\om)$ for all choices of symplectic forms and of Hamiltonian circle actions on these two $4$-manifolds. Contrary to the toric case, the homotopy type of the stabilizer $C(S^1)$ is not constant and depends, essentially, on whether the circle action extends to a single toric action or to two distinct toric actions. In the former case, apart from a few exceptional circle actions that must be treated separately, $C(S^1)$ retracts onto the unique torus $\mathbb{T}^2$ the circle extends to. In the latter case, the two toric actions $\mathbb{T}^2_1$, $\mathbb{T}^2_2$ that extend the circle action to do not commute, even up to homotopy. We show that the Pontryagin products of the generators of $H_1(\mathbb{T}^2_1)$ and $H_1(\mathbb{T}^2_2)$ generate a subalgebra $P^{alg}\subset H_*(C(S^1))$ that contains classes of arbitrary large degrees. In particular, $C(S^1)$ does not have the homotopy type of a finite dimensional $H$-space. Moreover, looking at the action of the centralizer $C(S^1)$ on the space of invariant almost complex structures, we prove that there is an equality of Pontryagin rings $P^{alg}= H_*(C(S^1))$. This homology equivalence implies that $C(S^1)$ has the homotopy type of a pushout $\mathbb{T}^2_1 \from S^1\to \mathbb{T}^2_2$ in the category of topological groups. Finally, when viewed as a bare topological space, we show that $C(S^1)$ is homotopy equivalent to the product $\Omega S^3 \times S^1 \times S^1 \times S^1$.\\

It seems likely that our results on centralizers of Hamiltonian circle actions could be proven using moment map techniques alone. The main advantage of introducing pseudo-holomorphic curve techniques is that our setting may apply to actions of other compact, possibly finite, abelian group $A\subset\Ham(M,\om)$. For instance, in the companion paper~\cite{Zn_symp}, we use the same framework to determine the homotopy type of the centralizers of most finite cyclic subgroups $\Z_n$ acting on $(\SSS,\oml)$ and $(\CCC,\oml)$ through Hamiltonian diffeomorphisms.\\

Finally, it is worth pointing out that all the above results on the topology of symplectomorphism groups of rational $4$-manifolds depend crucially on our ability to understand the action of $\Symp$ on the space of compatible almost complex structures. Since this analysis relies on properties of $J$-holomorphic curves that are specific to $4$-manifolds, it does not generalize to higher dimensions.\\ 

{\it Organization of the paper.} In an effort to make the paper understandable to both equivariant geometers and symplectic topologists, we include more details than would be needed in a document aimed at either audience alone. The paper is structured as follows:\\

In Chapter 2, we recall T. Delzant's equivariant classification of toric actions in terms of moment polytopes, and Y. Karshon's equivariant classification of Hamiltonian circle actions on $4$-manifolds in terms of labelled graphs. We then determine all possible toric extensions, up to equivalences, of an arbitrary Hamiltonian circle actions on $\SSS$ and $\CCC$.\\ 

The crux of the paper lies in Chapters 3, 4, and 5 in which we adapt the framework of \cite{MR1775741} to study groups of equivariant symplectomorphisms in the presence of an $S^1$ action. In particular we show that the space of invariant almost complex structures $\mathcal{J}^{S^1}_\om$ decomposes into disjoint strata, each of them being homotopy equivalent to an orbit of the centralizer, with stabilizer homotopy equivalent to $S^1$ equivariant K\"ahler isometries (see Theorems~\ref{homogenous} and \ref{homog}).  In Chapter 3, we show that the number of invariant strata in the decomposition corresponds to the number of toric extensions of the given circle action (Proposition~\ref{prop:ToricExtensionCorrespondance}). In particular we prove that $\mathcal{J}^{S^1}_\om$ decomposes into either one or two strata, and that the later case occurs only for an exceptional family of circle actions on $\SSS$ and $\CCC$ (Corollaries~\ref{cor:lambda=1_IntersectingOnlyOneStratum},\ref{cor:IntersectingTwoStrata} and \ref{cor:IntersectingOnlyOneStratum}). In Chapter 4, using techniques similar to the ones developed in~\cite{AG} we compute the homotopy type of $\Symp^{S^1}(\SSS)$ for all Hamiltonian circle actions on $\SSS$. In Chapter 5, we prove $S^1$ equivariant analogues of some technical lemmas of~\cite{AGK} involving deformation theory. We use these results to justify the claim made in Chapter~4 that in the case the circle admits two distinct toric extensions, the stratum with positive codimension in $\mathcal{J}^{S^1}_\om$ is always of codimension two.\\

In Chapter 6, we carry out a similar analysis on the manifold $\CCC$ and obtain the homotopy type of $\Symp^{S^1}(\CCC)$ for all Hamiltonian circle actions on $\CCC$.\\


\chapter{Hamiltonian torus actions on ruled \texorpdfstring{$4$}{4}-manifolds}

\section{Preliminaries on Hamiltonian actions}
Let $G$ be a compact Lie group acting effectively and symplectically on a symplectic manifold $(M,\om)$. Every such action $\rho:G\times M\to M$ induces an injective homomorphism $G\into\Symp(M,\om)$ and, in particular, defines a subgroup $G\subset\Symp(M,\om)$. Let $\mathfrak{g}$ denote the Lie algebra of $G$ and $\mathfrak{g}^*$ be it's dual. Given $Y\in \mathfrak{g}$, we denote by $\overline{Y}$ the corresponding fundamental vector field on~$M$.  

\begin{defn}\label{def:HamiltonianActions} The action of $G$ is called Hamiltonian if there exists a moment map, that is, a smooth map $\mu:M \to \mathfrak{g}^*$ such that
\begin{enumerate}
\item $d\mu_p(X_p)(Y) = \om(X,\overline{Y})$ for all $X \in T_pM$ and $Y \in \mathfrak{g}$, and
\item the map $\mu$ is equivariant with respect to the $G$ action on $M$ and the coadjoint action on~$\mathfrak{g}^*$.
\end{enumerate}
In particular, $G\subset\Ham(M,\om)\subset\Symp(M,\om)$. 
\end{defn}

\begin{remark}\label{rmk:UniquenessMomentMap}
We note that in this definition, the moment map $\mu:M\to\mathfrak{g}^{*}$ is an auxiliary structure whose sole existence is required. There is no claim about its unicity. For instance, the moment map of a Hamiltonian torus action is only defined up to adding a constant $c\in\mathfrak{t}^*$. However, there are situations where the choice of a specific moment map is needed. For this reason, we will always distinguish between the structure $(M,\om,\rho)$ given by a Hamiltonian action alone, from the structure $(M,\om,\rho,\mu)$ in which a moment map $\mu$ is chosen.
\end{remark}

\begin{defn}\label{defn:EquivalenceActions}
We will say that two actions $\rho_{i}:G\to\Symp(M,\om)$ are symplectically equivalent if the corresponding subgroups $\rho_{i}(G)$ belong to the same conjugacy class with respect to the action of $\Symp(M,\om)$.
\end{defn}

Given a symplectic action $\rho:G\times M\to M$, let $C(G)$ be the centralizer of the corresponding subgroup $G\subset\Symp(M,\om)$, that is,
\[C(G)=\{\phi\in\Symp(M,\om)~|~\phi\circ g\circ \phi^{-1} = g,~\forall g\in G\}\]
and let $N(G)$ be its normalizer, namely,
\[N(G)=\{\phi\in\Symp(M,\om)~|~\phi\circ G\circ \phi^{-1} = G\}\]
Clearly, equivalent actions have isomorphic centralizers and normalizers. 

In this paper, we will investigate the homotopy type of the centralizers $C(T)$ of Hamiltonian torus actions on some $4$-manifolds. We start by giving a simple characterization of symplectomorphisms commuting with a Hamiltonian torus action.
\begin{lemma}\label{lemma:CharacterizationCentralizer}
Let $(M, \om)$ be a compact symplectic manifold equipped with a Hamiltonian torus action $\rho:T\times M\to M$ with moment map $\mu$. Then $\phi \in C(T)$ if, and only if, $\mu \circ \phi = \mu$ and $\phi \in \Symp(M,\omega)$.
\end{lemma}
\begin{proof}
$(\Leftarrow)$ Suppose $\phi \in \Symp(M,\omega)$ satisfies $\mu \circ \phi = \mu$. Let $X \in\mathfrak{t}$(where $\mathfrak{t}$ denotes the Lie algebra of the torus) and let $\overline{X}$ denote the associated fundamental vector field. We then have
\[\omega(d\phi^{-1}(\overline{X}),Y) = \phi^*\omega(\overline{X}, d\phi(Y)) = \omega(\overline{X}, d\phi(Y) =  d\mu(d\phi(Y))=d\mu(Y) = \omega(\overline{X}, Y)\]
for any vector field $Y$, which implies $d\phi(\overline{X}) = \overline{X}$ for all $X \in \mathfrak{t}$. Consequently, $\phi$ commutes with the action.
\\

$(\Rightarrow)$ If the action $\rho$ has moment map $\mu$, then $\mu \circ \phi^{-1}$ is a moment map for the conjugate action $\phi^{-1} \circ \rho \circ \phi$. If $\phi^{-1} \circ \rho \circ \phi = \rho$, this implies $\mu \circ \phi^{-1} = \mu + C$ for some constant $C\in\mathfrak{t}$. Since the moment images $\mu(M)$ and $\mu\circ \phi(M)$ are compact and coincide, this constant must be $0$. 
\end{proof}

The Atiyah-Guillemin-Sternberg theorem (see~\cite{Audin}) states that the the moment map image of a torus action $\mathbb{T}^d$ is a convex polytope of $\mathfrak{t}^*$ whose vertices are images of the fixed points of the torus action. 

We call an effective Hamiltonian torus action \emph{toric} if the torus acting is half the dimension of the manifold $M$. By a theorem of Delzant, the moment map image $\Delta:=\mu(M)$ of a toric action determines the symplectic manifold $(M,\om)$, the action $\rho$, and the moment map $\mu$ up to equivariant symplectomorphisms. 
\begin{thm}(Delzant \cite{Delzant}) 
Let $(M_{i},\om_{i},\rho_{i},\mu_{i})$, $i=1,2$, be two toric manifolds of dimension $2n$ with moment maps $\mu_{i}:M_{i}\to\mathfrak{t}^{*}$. If the two moment polytopes $\mu_{i}(M_{i})$ coincide, then there exists a $\mathbb{T}^{n}$-equivariant symplectomorphism $\phi:(M_{1},\om_{1})\to (M_{2},\om_{2})$ such that $\phi^{*}\mu_{2}=\mu_{1}$. Conversely, the moment map images of two equivariantly symplectomorphic toric structures $(M_{i},\om_{i},\rho_{i},\mu_{i})$, $i=1,2$, coincide.
\end{thm}
Choose an identification $\mathfrak{t}\simeq\R^{n}$ such that $\ker(\exp:\mathfrak{t}\to \mathbb{T}^{n})\simeq \Z^{n}$. The moment polytopes of toric actions on manifolds of dimension $2n$ -- called Delzant polytopes of dimension $n$ -- are completely characterized by the following three properties, see~\cite{Delzant}:  
\begin{itemize}
\item (simplicity) there are exactly $n$ edges meeting at each vertex; 
\item (rationality) the edges meeting at any vertex $p$ are of the form $ p + t u_{i} $, $ t \in [0,\ell_i] $, where $ \ell_i \in \R$ and $ u_{i} \in \mathbb{Z}^{n} $;
\item (smoothness) for each vertex $p$, the corresponding vectors $u_{1},\ldots,u_{n}$ can be chosen to be a basis of $\mathbb{Z}^{n}$ over $\Z$.
\end{itemize}
This provides a purely combinatorial description of toric structures. Finally, if we are only interested in the subgroup $\mathbb{T}^n\subset\Ham(M^{2n},\om)$ associated to a toric action, that is, if we disregard both the pa\-ra\-me\-tri\-za\-tion $\mathbb{T}^n\into\Ham(M^{2n},\om)$ and the moment map, Delzant's classification theorem yields a bijection
\begin{gather*}
\{\text{Inequivalent toric actions on~} 2n\text{-manifolds}\}\\
\updownarrow\\
\{\text{Delzant polytopes in~}\R^n \text{~up to~} \AGL(n;\Z) \text{~action}\}
\end{gather*}
As the next proposition shows, the centralizer and normalizer of symplectic toric actions are easy to describe in terms of moment maps. In particular, the homotopy type of $C(T)$ does not depend on the toric structure.
\begin{prop}\label{prop:CentralizersToricActions}
Consider an effective toric action $\rho:\mathbb{T}^{n}\to\Symp(M^{2n},\om)$
with moment map $\mu:M\to\mathfrak{t}^*$. Denote by $T$ the corresponding torus
in $\Symp(M,\om)$, by $C(T)$ its centralizer, by $N(T)$ its normalizer, and by $W(T) =N(T)/C(T)$ its Weyl group.
\begin{enumerate}
\item The centralizer $C(T)$ deformation retracts onto $\mathbb{T}^{n}$. In particular, $C(T)\subset\Ham(M,\om)$.
\item The normalizer $N(T)$ is equal to the group of all symplectomorphisms $\psi$ such that $\mu\circ\psi = \Lambda\circ\mu$ for some $\Lambda\in\AGL(n;\Z)$. In particular, the Weyl group $W(T)$ is finite.
\end{enumerate}
\end{prop}
\begin{proof}
This is a slightly stronger version of Proposition 3.21 in~\cite{P-MaxTori}, and only the first statement requires some additional justification. By Lemma~\ref{lemma:CharacterizationCentralizer}, $\phi\in C(T)$ iff $\phi\in\Symp(M,\om)$ and $\mu=\mu\circ\phi$. Since the action is toric, each fiber of the moment map consists in a single orbit. It follows that there is a unique map $g_{\phi}:\mu(M)=\Delta\to\mathbb{T}^{n}$ such that $\phi$ can be written as
\[\phi(m) = g_{\phi}(\mu(m))\cdot m\]
Pick a base point $b\in\Delta$ and consider the evaluation fibration
\begin{align*}
C_{b}(T)\to C(T)&\to \mathbb{T}^{n}\\
\phantom{C_{b}(T)\to} \phi&\mapsto g_{\phi}(b) 
\end{align*} 
whose fiber is the subgroup
\[C_{b}(T)=\{\phi\in C(T)~|~\phi=\id\text{~on~}\mu^{-1}(b)\}\]
Consider $\phi\in C_{b}(T)$. Since $\Delta$ is contractible, the function $g_{\phi}$ has a unique lift $\tilde{g}_{\phi}:\Delta\to\mathfrak{t}$ such that $\tilde{g}_{\phi}(b)=0$, and $g_{\phi}=\exp\circ\tilde{g}_{\phi}$. Setting $\tilde{g}_{\phi,s}=s\tilde{g}_{\phi}$ for $s\in[0,1]$, we obtain a retraction of $C_{b}(T)$ onto $\{\id\}$ through diffeomorphisms $\phi_{s}$ leaving $\mu$ invariant and such that $\phi_{s}=\id$ on $\mu^{-1}(b)$. Applying an equivariant Moser isotopy to the family $\phi_{s}$, this retraction is seen to be homotopic to a retraction in $C_{b}(T)$. This shows that $C_{b}(T)$ is contractible, which in turns implies the result.
\end{proof}
For Hamiltonian torus actions $T^{d}\to\Ham(M^{2n},\om)$ for which the torus $T^{d}$ has dimension $d<n$, there is no simple description of the homotopy type of the symplectic centralizer $C(T)$. As we will see, this already happens in the simplest case, namely, for Hamiltonian circle actions on $4$-manifolds. 

In \cite{Karshon}, Y. Karshon established an equivariant classification of Hamiltonian circle actions on 4-dimensional symplectic manifolds, similar to Delzant classification, in which the moment map image is replaced by labelled graphs. More precisely, given any Hamiltonian $S^1$ action on a 4-manifold $(M,\om)$ with moment map $\mu$, one can construct a labelled graph as follows:
\begin{itemize}
\item Each component of the fixed point set corresponds to a unique vertex of the graph.
\item Each vertex is labeled by the value of the moment map on the corresponding
fixed point component. 
\item An extremal vertex that corresponds to a fixed symplectic
surface S is said to be "fat". Two additional labels are attached to fat vertices: the genus of that surface,
and its normalized symplectic area. 
\item Two vertices are connected by an edge if and only if the corresponding isolated
fixed points are connected by a $\Z_k$-sphere i.e by a $S^1$ invariant sphere on which  $S^1$ acts with a global stabilizer $\Z_k$, $k \geq 2$.
\item  Each edge is labelled by the isotropy weight $k \geq 2$ of the corresponding $\Z_k$ sphere.
\end{itemize}
Labeled graphs associated to effective Hamiltonian circle actions are characterized by the following properties. If we order the vertices according to their moment map labels, then
\begin{itemize}
\item there are exactly two extremal vertices;
\item fat vertices are extremal, and if the graph contains two fat vertices, then their genus label must coincide;
\item the area label of any fat vertex must be strictly positive;
\item a vertex is connected to no more than two edges, and no edge is connected to a fat vertex;
\item the moment map labels must be strictly monotone along each chain of edges;
\item if $e_{1},\ldots,e_{\ell}$ is a chain of edges, and if $k_{1},\ldots,k_{\ell}$ are the orders of their stabilizers, then $\gcd(k_{i}, k_{i+1}) = 1$ for $i = 1,\ldots\ell-1$, and $(k_{i-1} + k_{i+1})/k_{i}$ is an integer for $i= 2,\ldots,\ell-1$.
\end{itemize}
We call such graphs \emph{admissible}. Just as Delzant polytopes classify toric structures up to symplectomorphisms, admissible labelled graphs classify Hamiltonian $S^1$ actions with moment maps.

\begin{thm}(Karshon \cite{Karshon})\label{graphiso}
Each admissible labelled graph determines a Hamiltonian circle action $\rho$ with moment map $\mu$ on a symplectic $4$ manifold $(M,\om)$. The tuple $(M,\om,\rho,\mu)$ is unique up to $S^1$-equivariant symplectomorphisms preserving the moment map. Conversely, the labelled graphs of two equivariantly symplectomorphic Hamiltonian circle actions with moment maps are isomorphic. 
\end{thm}
Again, if one is only interested in conjugacy classes of circles in Hamiltonian groups, Karshon's classification implies that there is a bijection
\begin{gather*}
\{\text{Inequivalent Hamiltonian circle actions on~} 4\text{-manifolds}\}\\
\updownarrow\\
\{\text{Admissible labelled graphs up to~} \AGL(1;\Z) \text{~action on moment map labels}\}
\end{gather*}

\begin{remark}\label{rmk:NormalizationMomentMap}~
\begin{itemize}
\item Note that when $M$ is compact, it is always possible to canonically assign a moment map to an Hamiltonian circle action by imposing one of the normalizations $\min \mu = 0$ or $\int_{M}\mu = 0$. 
\item Note also that in Karshon's classification, it is not necessary to keep track of invariant spheres on which the restricted action is effective. In particular, these spheres do not appear in the labelled graphs. This point will be important in our analysis.
\end{itemize}
\end{remark}

The equivariant Darboux-Weinstein theorem together with the classification of symplectic representations imply that a Hamiltonian action by a circle on a $4$-manifold has a simple local description near any fixed point. The following proposition follows from \cite{Montaldi} Theorem~2.4, see also~\cite{Karshon} Corollary~A.7 or \cite{Lerman-Tolman} Lemma~A.1.

\begin{prop}\label{prop:DefinitionOfWeights}
Let $p$ be a fixed point of a Hamiltonian action by either a circle $S^1$  on a $4$-manifold $(M,\om)$. Then there exist complex coordinates $(z,w)$ on a neighborhood of $p$ in $M$, and unique integers $\alpha$ and $\beta$, called the isotropy weights at $p$, such that
\begin{enumerate}
\item the action is $t\cdot (z, w) = (t^{\alpha}z, t^{\beta}w)$,
\item the symplectic form is $\om = \frac{i}{2}(dz\wedge d\bar z + dw \wedge d\bar w)$,
\item if the group is a circle, the moment map is $\mu(z,w) = \mu(p) + \frac{\alpha}{2}|z|^{2} + \frac{\beta}{2}|w|^{2}$.
\end{enumerate}
\end{prop}
Although complex coordinates are used in the definition of weights, the unordered set $\{\alpha,\beta\}$ is a symplectic invariant of the action at the fixed point $p$. This already imposes some restrictions on equivariant symplectomorphisms.

\begin{cor}\label{cor:ActionPreservesWeights}
Consider a Hamiltonian action by a circle $S^1$ on a symplectic $4$-manifold $(M,\om)$. Let $\phi$ be an equivariant symplectomorphism. Then $\phi$ acts on the fixed point set preserving their weights as unordered sets. In particular, equivalent actions have the same set of weights.\qed
\end{cor}

As explained in \cite{Karshon} pages 6-7, the isotropy weights of a Hamiltonian circle action can be read from its labelled graph as follows:
\begin{itemize}
    \item for $k \geq 2$, a fixed point has an isotropy weight $-k$ if and only if it is the north pole of a $\Z_{k}$ -sphere, and it has an isotropy weight $k$ if and only if it is the south pole of a $\Z_{k}$ -sphere;
    \item an interior fixed point has one positive weight and one negative weight, a maximal fixed point has both weights non-positive, a minimal fixed point has both weights non-negative;
    \item a fixed point has a weight $0$ if and only if it lies on a fixed surface.
\end{itemize}

The next lemma will be important in our analysis of $S^1$  actions.
\begin{lemma}[Lemma 5.4 in~\cite{Karshon}]\label{weight}
Let $S^1$  act on $S^2$ by rotating it $k$ times while fixing the
north and south poles. Suppose that the action lifts to a complex line
bundle $E$ over $S^2$. Then $S^1$ acts linearly on the fibers over the north and south
poles; let $\alpha$ and $\beta$ be the weights for these actions. Then
$$\alpha - \beta = -ek$$
where $e$ is the self intersection of the zero section.\qed
\end{lemma}

We end this section with a simple lemma that will be useful in computing homology classes of invariant spheres directly from labelled graphs.

\begin{lemma}\label{Lemma:SymplecticArea}
Let $S^1$  act on $S^2$ by rotating it $k$ times while fixing the
north and south poles. Let $\mu(n)$ denote the value of the momentum map at the north pole and $\mu(s)$ the value of the momentum map at the south pole. Then the symplectic area of the $S^2$ is given by $\frac{\mu(n) - \mu(s)}{k}$.\qed
\end{lemma}

\section{Torus actions on \texorpdfstring{$\SSS$ and $\CCC$}{S2xS2 and CP2\#CP2}}
We now combine the \emph{equivariant} classification theorems of T. Delzant and Y. Karshon, together with the Lalonde-McDuff classification of symplectic forms on ruled surfaces, to describe all inequivalent Hamiltonian circle actions on $S^2 \times S^2$ and $\CCC$ and their relations to toric actions. 

\subsection{Symplectic forms on \texorpdfstring{$\SSS$ and $\CCC$}{S2xS2 and CP2\#CP2}}
We equip $\SSS$ with the product symplectic form $\om_\lambda=\lambda \sigma \oplus \sigma$, where $\sigma$ is the standard area form on $S^2$ normalized such that $\sigma(S^2)=1$. If we think of $\SSS$ as a trivial fiber bundle, $\oml$ gives area $1$ to the fibers, while the area of horizontal sections is $\lambda$. Similarly, if we view $\CCC$ as the non-trivial $S^2$ bundle over $S^2$, we define an analogous form $\oml$ which gives area $1$ to the fibers and area $\lambda-1$ to symplectic sections of self-intersection $-1$, that is, to sections homologous to the exceptional divisor. From an homological point of view, if $F$ denotes the homology class of a fiber in either $\SSS$ or $\CCC$, if $B$ denotes the class of a section of self-intersection $0$ in $\SSS$, if $E$ denotes the class of the exceptional divisor in $\CCC$, and if $L$ denotes the class of a line in $\CCC$, then $[\oml]F=1$, $[\oml]B=\lambda$, $[\oml]L=\lambda$ and $[\oml]E=\lambda-1$. Note that with our conventions on $\oml$ for the non-trivial bundle $\CCC$, $\lambda$ must be strictly greater than 1 for the symplectic curve in class $E$ to have positive area. We can now state the Lalonde-McDuff classification theorem.

\begin{thm}[Lalonde-McDuff~\cite{MR1426534}]\label{L-McD-Classification} Any symplectic form on $\SSS$ or $\CCC$ is conformally diffeomorphic to a form $\oml$ with $\lambda\geq 1$. Moreover, any two cohomologous forms are diffeomorphic.
\end{thm}

\subsection{Toric actions on \texorpdfstring{$\SSS$ and $\CCC$}{S2xS2 and CP2\#CP2}}\label{Section:ToricActionsOnSSSandCCC} It is well known that, up to symplectic equivalence, the only toric actions on $\SSS$ and $\CCC$ are given by the standard torus actions on the Hirzebruch surfaces. We now review the main points of the argument. 
 
Recall that the $m^{\text{th}}$ Hirzebruch surface $W_m$ is the complex submanifold of $\mathbb{C}P^1 \times \mathbb{C}P^2$ defined by the equation
\[
W_m:=  \left\{ \left(\left[x_1,x_2\right],\left[y_1,y_2, y_3\right]\right) \in \mathbb{C}P^1 \times \mathbb{C}P^2 ~|~  x^m_1y_2 - x_2^my_1 = 0 \right\}
\]
The projection map $\mathbb{C}P^1 \times \mathbb{C}P^2 \longrightarrow \mathbb{C}P^1$ gives $W_m$ the structure of a $\mathbb{C}P^1$ bundle over $\mathbb{C}P^1$ which is diffeomorphic (but not complex isomorphic) to $S^2 \times S^2$ if $m$ is even and diffeomorphic to the non-trivial $S^2$ bundle over $S^2$ i.e $\CP^2 \# \overline{\CP^2}$ if $m$ is odd. Thus, for each $m$ we have an integrable complex structure $J_m$ induced on $\SSS$ or $\CCC$.  When $m=2k$ even, choose $\lambda > k$, then we can endow  $\mathbb{C}P^1 \times \mathbb{C}P^2$ with the symplectic form $(\lambda -k) \sigma_1 \oplus \sigma_2$, where $\sigma_1$ is   the standard symplectic form on  $\mathbb{C}P^1$ with area 1 and $\sigma_2$ is the standard symplectic form on $\mathbb{C}P^2$ such that $\sigma_2(L) =1$, where $L$ is the class of the line in $\CP^2$.  Restricting this symplectic form to $W_m$ makes it a symplectic manifold. Similarly when $m=2k+1$,  choose $\lambda > k+1$, then we can analogously define the form $(\lambda - (k+1)) \sigma_1 \oplus \sigma_2$ when $m$ is odd. With these choices of symplectic forms, the Lalonde-McDuff classification theorem~\ref{L-McD-Classification} implies that $W_m$ is symplectomorphic  to $(\SSS,\oml)$ if $m$ is even and to $(\CP^2 \# \overline{\CP^2},\oml)$ when $m$ is odd.\\

Given an integer $m\geq 0$, the torus $\T$ acts on $\mathbb{C}P^1 \times \mathbb{C}P^2$ by setting
\[
 \left(u,v\right) \cdot  \left(\left[x_1,x_2\right],\left[y_1,y_2, y_3\right]\right) = \left(\left[ux_1,x_2\right],\left[u^my_1,y_2,vy_3\right]\right)
\]
This action leaves $W_m$ invariant and preserves both the complex and the symplectic structures. Its restriction to $W_m$ defines a toric action that we denote $\T_m$. 

When $m$ is even, the image of the moment map is the polytope of Figure~\ref{hirz}
\begin{figure}[H]
\centering       
\scalebox{0.9}{
\begin{tikzpicture}
\node[left] at (0,2) {$Q=(0,1)$};
\node[left] at (0,0) {$P=(0,0)$};
\node[right] at (4,2) {$R= (\lambda - \frac{m}{2} ,1)$};
\node[right] at (6,0) {$S=(\lambda + \frac{m}{2} ,0)$};
\node[above] at (2,2) {$D_m=B-\frac{m}{2}F$};
\node[right] at (5.15,1) {$F$};
\node[left] at (0,1) {$F$};
\node[below] at (3,0) {$B+ \frac{m}{2}F$};
\draw (0,2) -- (4,2) ;
\draw (0,0) -- (0,2) ;
\draw (0,0) -- (6,0) ;
\draw (4,2) -- (6,0) ;
\end{tikzpicture}}
\caption{Even Hirzebruch polygon}
\label{hirz}
\end{figure}

\noindent where the labels along the edges refer to the homology classes of the  $\T$ invariant spheres in $\SSS$, and where the vertices $P$,$Q$,$R$,$S$ are the fixed points of the torus action. 
Similarly, when $m$ is odd, the moment map image is given in Figure~\ref{fig:OddHirzebruch}
\begin{figure}[H]
\centering   
\label{hirz0}
\scalebox{0.9}{
\begin{tikzpicture}
\node[left] at (0,2) {$Q=(0,1)$};
\node[left] at (0,0) {$P=(0,0)$};
\node[right] at (4,2) {$R= (\lambda - \frac{m+1}{2} ,1)$};
\node[right] at (6,0) {$S=(\lambda + \frac{m-1}{2} ,0)$};
\node[above] at (2,2) {$D_m=B-\frac{m+1}{2}F$};
\node[right] at (5.15,1) {$F$};
\node[left] at (0,1) {$F$};
\node[below] at (3,0) {$B+ \frac{m-1}{2}F$};
\draw (0,2) -- (4,2) ;
\draw (0,0) -- (0,2) ;
\draw (0,0) -- (6,0) ;
\draw (4,2) -- (6,0) ;
\end{tikzpicture}}
\caption{Odd Hirzebruch polygon}\label{fig:OddHirzebruch}
\end{figure}
\noindent where $B$ now refers to the homology class of a line $L$ in $\CP^2 \# \overline{\CP^2}$, $E$ is the class of the exceptional divisor, and $F$ is the fiber class $L - E$.\\

We define the zero-section $s_0$ to be
\begin{align*}
   s_0: \CP^1 &\to W_m \\
    [x_1;x_2]  &\mapsto \{[x_1,x_2], [0;0;1]\}
\end{align*}
and the section at infinity $s_\infty$ to be 
\begin{align*}
     s_\infty: \CP^1 &\to W_m \\
    [x_1;x_2]  &\mapsto \{[x_1,x_2], [x^m_1;x^m_2;0]\}
\end{align*}
The image of $s_0$ is an invariant and holomorphic sphere homologous to either $D_m:=B-\frac{m}{2}F$ in $\SSS$ or to $D_m:=B- \frac{m+1}{2}F$ in $\CCC$, depending on the parity of $m$. Similarly, $s_\infty$ defines an invariant and holomorphic sphere that represents either $B + \frac{m}{2}F$ in $\SSS$ or $B+\frac{m+1}{2}F$ in $\CCC$. Finally, the homology class $F$ can be represented by an invariant fibre such as $\{[1,0], [y_1,0,y_3]\}$.\\

It is well known (see~\cite{Audin}) that a Delzant polygon $\Delta\subset\R^{2}$ with $e\geq 3$ edges defines a toric $4$-manifold $M_{\Delta}$ whose second Betti number is $b_{2}(M_{\Delta})=e-2$. It follows that the moment polytope of any toric action on either $\SSS$ or $\CCC$ is a quadrilateral. It is easy to see that any Delzant quadrilateral is equivalent to one of the above Hirzebruch trapezoids up to $\AGL(2,\Z)$ action and up to rescaling. It follows from Delzant's classification that any toric action on $\SSS$ and $\CCC$ is equivalent to an action of the above form. In particular, the equivalence classes of toric actions on $\SSS$ and $\CCC$ are completely characterized by the existence of invariant spheres in specific homology classes.

\begin{lemma}\label{lemma_torus_action_-vecurve}
Up to equivalence, the toric action $\T_m$, $m\geq 1$, is characterised by the existence of an invariant, embedded, symplectic sphere $C_m$ with self intersection $-m$. The toric action $\T_{0}$ is characterized by the existence of invariant, embedded, symplectic spheres representing the classes $B$ and~$F$.\qed
\end{lemma}

\begin{lemma}\label{lemma_number_of_torus_actions}
Write $\lambda=\ell+\delta$ with $\ell$ an integer and $0<\delta\leq 1$. Then, up to symplectomorphisms and reparametrizations, 
\begin{itemize}
\item if $\lambda \geq 1$, there are exactly $\ell+1$ inequivalent toric actions on $(\SSS,\oml)$ given by the even Hirzebruch actions $\T_{2k}$ with $0\leq k\leq \ell$, and
\item if $\lambda >1$, there are exactly $\ell$ inequivalent toric actions on $(\CCC,\oml)$ given by the odd Hirzebruch actions $\T_{2k+1}$ with $0\leq k\leq \ell-1$.
\end{itemize}
\end{lemma}

\begin{proof}
Write $m=2k$ or $m=2k+1$ with $k\geq 0$. As seen above, any toric action on $\SSS$ and $\CCC$ is $\T$-equivariantly symplectomorphic to one of the actions $\T_m$. The invariant symplectic sphere $D_k$ in class $B-kF$ on $\SSS$ or $L - (k+1)F$ in $\CCC$ must have positive area, that is, 
\begin{multline*}
0<\oml(B-kF) = \lambda - k = \ell+\delta-k,\\
\quad\text{~or~}\quad 0<\oml(L-(k+1)F) =\lambda-(k+1)F = \ell+\delta-(k+1)
\end{multline*}
The result follows.
\end{proof}

\begin{prop}\label{prop:NormalizersHirzebruch}
The Weyl group $W(\T_{m})=N(\T_{m})/C(\T_{m})$ is isomorphic to
\[W(\T_{m})\simeq
\begin{cases}
\text{the dihedral group~} D_{4}, & \text{~when~} m=0 \text{~and~} \lambda=1;\\
\text{the dihedral group~} D_{2}\simeq\Z_{2} \times \Z_2, & \text{~when~} m=0 \text{~and~} \lambda>1;\\
\text{the dihedral group~} D_{1}\simeq\Z_{2},& \text{~when~} m\geq 1.
\end{cases}\]
\end{prop}
\begin{proof}
This follows directly from Proposition~\ref{prop:CentralizersToricActions} and from the description of Hirzebruch trapezoids. Indeed, for $m=0$ and $\lambda=1$, the moment map polygon is a unit square whose symmetries can be realized by elements of $\AGL(2,\Z)$. The same holds for $m=0$ and $\lambda>1$, with the only difference that the moment polygone is now a rectangle with sides of lengths $1$ and $\lambda$. Finaly, for $m\geq 1$, if we write $m=2k$ or $m=2k+1$, the only non-trivial element of  $\AGL(2,\Z)$ that leaves the standard Hirzebruch trapezoid invariant is
\[\Lambda 
=\begin{pmatrix}-1 & -m\\ 0 & 1\end{pmatrix}
+\begin{pmatrix}\lambda+k,0\end{pmatrix}\]
which is an element of order 2.
\end{proof}

\subsection{Hamiltonian circle actions on \texorpdfstring{$\SSS$ and $\CCC$}{S2xS2 and CP2\#CP2}}
A list of all possible Hamiltonian circle action on $\SSS$ or $\CCC$ comes easily from an extension theorem of Y. Karshon.
\begin{thm}[Karshon~\cite{finTori}, Theorem 1]
Any symplectic $S^1$ action on $(\SSS,\oml)$ and $(\CCC,\oml)$ extends to an Hamiltonian toric action. 
\end{thm} 

Consequently, every symplectic $S^1$ action on $\SSS$ and $\CP^2 \# \overline{\CP^2}$ is given by an embedding 
\begin{align*}
    S^1 &\hookrightarrow \T_m \\
    t &\mapsto (t^a, t^b)
\end{align*}
which is itself characterized by a unique triple of numbers $(a,b;m)\in \Z \times \Z \times \Z_{\geq 0}$. Since we are only interested in effective actions, this translates numerically into the condition  $\gcd(a,b) = 1$. We shall always assume this unless otherwise stated.

In order to construct the graphs of the circle action $S^1(a,b;m)$, we claim that it is enough to compute the isotropy weights at every fixed points. Indeed, given an Hirzebruch trapezoid, the pre-image of a vertex under the moment map is a toric fixed point, and the pre-image of an edge is an invariant embedded two-sphere connecting two fixed points. When we view the space as a Hamiltonian $S^1$-space, such a two-sphere is either fixed by the action, or is a $\Z_k$ sphere, or has trivial global isotropy. These three possibilities are completely determined by the weights of the $S^{1}(a,b;m)$ action at its two fixed points. We can thus construct the graphs starting from the fixed points and adding edges according to the weights. If there is a fixed surface, its area label can be read from the Hirzebruch trapezoid. Finally, the moment map labels can be added starting with the minimal vertex (characterised by having two positive weights) that we label $\mu=0$, and then using Lemma~\ref{Lemma:SymplecticArea} to compute the moment labels for the remaining interior fixed points, and for the maximal fixed point (characterised by having two negative weights).

Now, for the Hirzebruch actions $\T_{m}=(S^{1}\times S^{1})_{m}$, each $S^{1}$ factor defines two weights at each of the four fixed points $P,Q,R,S$. These weights are listed in table~\ref{table_weights}. Further, if the the $\T_m$ action has weights $\{\alpha_1,\beta_{1}\}$ and $\{\alpha_{2},\beta_2\}$ at a fixed point $x$, then the restricted $S^1(a,b;m)$ action has weights
\[\left\{a\alpha_1 + b\alpha_{2}, a\beta_{1} + b\beta_2 \right\} \]
at $x$. This gives the weights of the $S^{1}(a,b;m)$ actions at the fixed points $P,Q,R,S$.
\begin{table}[h]
\begin{center}
\begin{tabular}{ |p{2cm}||p{4cm}|p{55mm}|  }
 \hline
 Vertex & Weights for $\T_m$ action & Weights for the $S^1(a,b;m)$ action \\
 \hline
 P & $\big\{\{1,0\},\{0,1\}\big\}$   & $\{a,b\}$ \rule{0pt}{10pt}\\
 Q & $\big\{\{1,0\},\{0,-1\}\big\}$  & $\{a,-b\}$ \rule{0pt}{10pt}\\
 R & $\big\{\{-1,m\},\{0,-1\}\big\}$ & $\{-a,am-b\}$ \rule{0pt}{10pt}\\
 S & $\big\{\{-1,-m\},\{0,1\}\big\}$ & $\{-a,-am+b\}$ \rule{0pt}{10pt}\\
 \hline
\end{tabular}
\end{center}
\caption{Weights of $\T_{m}$ and $S^{1}(a,b;m)$ actions}
\label{table_weights}
\end{table}


In Figures~\ref{fig:GraphsWithFixedSurfaces} and~\ref{fig:GraphsIsolatedFixedPoints}, we present the graphs of circles actions $S^{1}(a,b;m)$ on $(\SSS,\oml)$ and $(\CCC,\oml)$. As before, we write $m=2k$ or $m=2k+1$, and in order to deal with even and odd Hirzebruch surfaces simultaneously, we introduce the symbol $\epsilon_{m} := m\mod 2$. Each label $\mu$ represents the value of the moment map normalized by setting $\min \mu=0$, while $A$ is the area label of a fixed surface. All fixed surfaces have genus 0. Note that when the isotropy label on an edge is 1, then the circle action on the corresponding invariant sphere has no isotropy, and we erase that edge from the graph. Note also that the identification of the fixed points with the vertices $P,Q,R,S$ of the Hirzebruch trapezoid is there for convenience only and is not part of the data.  

For actions with fixed surfaces, one of the weights $a$, $b$, or $am-b$ must be zero. Since $\gcd(a,b)=1$, the possible triples are $(\pm1,0;m)$, $(0,\pm1;m)$, and $(\pm1,\pm m;m)$. 
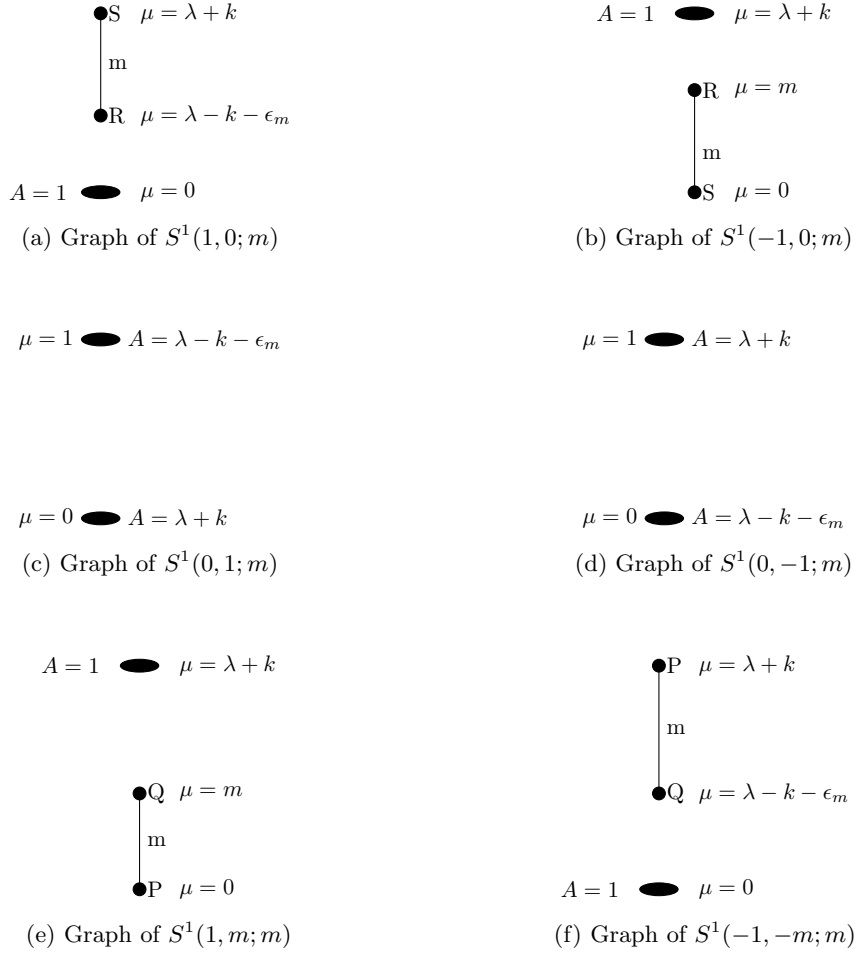
\begin{figure}[H]
\centering
\subcaptionbox{Graph of $S^{1}(1,0;m)$}
[.45\linewidth]
{\begin{tikzpicture}[scale=0.85, every node/.style={scale=0.85}]
\draw [fill] (0,0) ellipse (0.3cm and 0.1cm);
\draw [fill] (0,1.2) ellipse (0.1cm and 0.1cm);
\draw [fill] (0,2.8) ellipse (0.1cm and 0.1cm);
\draw (0,1.2) -- (0,2.8);
\node[right] at (0,2) {m};
\node[right] at (0,1.2) {R};
\node[right] at (0,2.8) {S};
\node[right] at (0.5,2.8){$\mu= \lambda +k$};
\node[right] at (0.5,0){$\mu=0$};
\node[left] at (-0.4,0) {$A= 1$};
\node[right] at (0.5,1.2) {$\mu=\lambda - k - \epsilon_{m}$};
\end{tikzpicture}}
~
\subcaptionbox{Graph of $S^{1}(-1,0;m)$}
[.45\linewidth]
{\begin{tikzpicture}[scale=0.85, every node/.style={scale=0.85}]
\draw [fill] (0,0) ellipse (0.3cm and 0.1cm);
\draw [fill] (0,-1.2) ellipse (0.1cm and 0.1cm);
\draw [fill] (0,-2.8) ellipse (0.1cm and 0.1cm);
\draw (0,-1.2) -- (0,-2.8);
\node[right] at (0,-2.2) {m};
\node[right] at (0,-2.8) {S};
\node[right] at (0,-1.2) {R};
\node[right] at (0.5,0){$\mu= \lambda +k$};
\node[right] at (0.5,-2.8){$\mu=0$};
\node[left] at (-0.5,0) {$A= 1$};
\node[right] at (0.5,-1.2) {$\mu=m$};
\end{tikzpicture}}

\subcaptionbox{Graph of $S^{1}(0,1;m)$}
[.45\linewidth]
{\begin{tikzpicture}[scale=0.85, every node/.style={scale=0.85}]
\draw [fill] (0,0) ellipse (0.3cm and 0.1cm); 
\draw [fill] (0,2.8) ellipse (0.3cm and 0.1cm); 
\node[left] at (-0.3,2.8) {$\mu = 1$};
\node[left] at (-0.3,0){$\mu=0$};
\node[right] at (0.3,2.8){$A= \lambda-k-\epsilon_{m}$};
\node[right] at (0.3,0){$A= \lambda+k$};
\node[above] at (0,2.8) {\rule{0em}{3em}}; 
\end{tikzpicture}}
~
\subcaptionbox{Graph of $S^{1}(0,-1;m)$}
[.45\linewidth]
{\begin{tikzpicture}[scale=0.85, every node/.style={scale=0.85}]
\draw [fill] (0,0) ellipse (0.3cm and 0.1cm); 
\draw [fill] (0,2.8) ellipse (0.3cm and 0.1cm); 
\node[left] at (-0.3,2.8) {$\mu = 1$};
\node[left] at (-0.3,0){$\mu=0$};
\node[right] at (0.3,2.8){$A= \lambda+k$};
\node[right] at (0.3,0){$A= \lambda-k-\epsilon_{m}$};
\end{tikzpicture}}
\subcaptionbox{Graph of $S^{1}(1,m;m)$}
[.45\linewidth]
{\begin{tikzpicture}[scale=0.85, every node/.style={scale=0.85}]
\draw [fill] (0,0) ellipse (0.3cm and 0.1cm);
\draw [fill] (0,-2) ellipse (0.1cm and 0.1cm);
\draw [fill] (0,-3.5) ellipse (0.1cm and 0.1cm);
\draw (0,-2) -- (0,-3.5);
\node[above] at (0,0) {\rule{0em}{3em}}; 
\node[right] at (0,-2.75) {m};
\node[right] at (0,-3.5) {P};
\node[right] at (0,-2) {Q};
\node[right] at (0.5,0){$\mu= \lambda +k$};
\node[right] at (0.5,-3.5){$\mu=0$};
\node[left] at (-0.5,0) {$A= 1$};
\node[right] at (0.5,-2) {$\mu=m$};
\end{tikzpicture}}
\subcaptionbox{Graph of $S^{1}(-1,-m;m)$}
[.45\linewidth]
{\begin{tikzpicture}[scale=0.85, every node/.style={scale=0.85}]
\draw [fill] (0,-3.5) ellipse (0.3cm and 0.1cm);
\draw [fill] (0,-2) ellipse (0.1cm and 0.1cm);
\draw [fill] (0,0) ellipse (0.1cm and 0.1cm);
\draw (0,0) -- (0,-2);
\node[right] at (0,-1) {m};
\node[right] at (0,0) {P};
\node[right] at (0,-2) {Q};
\node[right] at (0.5,0){$\mu= \lambda +k$};
\node[right] at (0.5,-3.5){$\mu=0$};
\node[left] at (-0.5,-3.5) {$A= 1$};
\node[right] at (0.5,-2) {$\mu=\lambda - k-\epsilon_{m}$};
\end{tikzpicture}}
\caption{Graphs of circle actions with non-isolated fixed points}\label{fig:GraphsWithFixedSurfaces} 
\end{figure}


For actions whose fixed points are isolated, none of the weights are zero. The graphs then depend on the signs of $a$, $b$, and $am-b$, as these signs determine which fixed point is minimal and which one is maximal.

\begin{figure}[H] 

\subcaptionbox*{}{} 

\setcounter{subfigure}{6} 
\subcaptionbox{When $a>0,~b>0$, and $am-b>0$\label{fig:a>0|b>0|am-b>0}}
[.53\linewidth]
{\begin{tikzpicture}[scale=0.9, every node/.style={scale=0.9}]
\draw [fill] (0,0) ellipse (0.1cm and 0.1cm);
\draw [fill] (1,-1.5) ellipse (0.1cm and 0.1cm);
\draw [fill] (0,-4) ellipse (0.1cm and 0.1cm);
\draw [fill] (-1,-2.5) ellipse (0.1cm and 0.1cm);
\draw (0,0) -- (1,-1.5);
\draw (1,-1.5) -- (-1,-2.5);
\draw (0,0) .. controls (-4,-1.5) and (-2,-3.5) .. (0,-4);
\draw (-1,-2.5) -- (0,-4);
\node[right] at (-0.35,-3.2) {$b$};
\node[right] at (- 0.4,-1.8) {$a$};
\node[left] at (-2.2,-1.5) {$a$};
\node[right] at (0.7,-0.5) {$am-b$};
\node[above] at (0,0.2) {$\mu = a(\lambda + k)$};
\node[right] at (-0.6,-2.5) {$\mu = b$};
\node[right] at (0.6,-1.9) {$\mu =b + a(\lambda - k-\epsilon_{m}) $};
\node[below] at (0, -4.2) {$\mu = 0$};
\node[left] at (-1,-2.5) {Q};
\node[left] at (0.9,-1.4) {R};
\node[right] at (0,0) {S};
\node[right] at (0,-4) {P};
\end{tikzpicture}}
\subcaptionbox{When $a>0,~b>0$ and $am-b<0$
}
[.45\linewidth]
{\begin{tikzpicture}[scale=0.9, every node/.style={scale=0.9}]
\draw [fill] (0,0) ellipse (0.1cm and 0.1cm); 
\draw [fill] (1,-2) ellipse (0.1cm and 0.1cm); 
\draw [fill] (0,-4) ellipse (0.1cm and 0.1cm); 
\draw [fill] (-1,-3) ellipse (0.1cm and 0.1cm); 
\draw (0,0) -- (1,-2);
\draw (-1,-3) -- (0,-4);
\draw (0,0) -- (-1,-3) ;
\draw (1,-2) -- (0,-4);
\node[right] at (-0.9,-3) {Q};
\node[left] at (0.9,-2) {S};
\node[right] at (0,0) {R};
\node[above] at (0,0.2) {$\mu = b + a (\lambda -k-\epsilon_{m})$};
\node[right] at (1.3,-2) {$\mu = a(\lambda + k) $};
\node[left] at (-1.3,-3) {$\mu = b$};
\node[right] at (0,-4) {P};
\node[right] at (-1,-3.7) {$b$};
\node[right] at (0.5, -3) {$a$};
\node[left] at (-0.5,-1.5) {$a$};
\node[right] at (0.7,-1) {$b-am$};
\node[below] at (0,-4.2) {$\mu = 0$};
\end{tikzpicture}}
\subcaptionbox{When $a>0,~b<0$}
[.45\linewidth]
{\begin{tikzpicture}[scale=0.9, every node/.style={scale=0.9}]
\draw [fill] (0,0) ellipse (0.1cm and 0.1cm);
\draw [fill] (1,-3) ellipse (0.1cm and 0.1cm);
\draw [fill] (0,-4) ellipse (0.1cm and 0.1cm);
\draw [fill] (-1,-2) ellipse (0.1cm and 0.1cm);
\draw (0,0) -- (1,-3);
\draw (-1,-2) -- (0,-4);
\draw (0,0) -- (-1,-2) ;
\draw (1,-3) -- (0,-4);
\node[above] at (0,1) {\rule{0em}{1em}}; 
\node[above] at (0,0.2) {$\mu = a(\lambda +k)-b$};
\node[right] at (1.4,-3) {$\mu = -b $};
\node[left] at (-1.4,-2) {$\mu = a(\lambda-k-\epsilon_{m})$};
\node[below] at (0,-4.2) {$\mu = 0$};
\node[right] at (-1,-3.2) {$a$};
\node[right] at (0.5, -3.6) {$-b$};
\node[left] at (-0.7,-1.0) {$am-b$};
\node[right] at (0.6,-1.3) {$a$};
\node[left] at (0.9,-3) {P};
\node[right] at (-1,-2) {R};
\node[right] at (0,0) {S};
\node[right] at (0,-4) {Q};
\end{tikzpicture}}
\subcaptionbox{When $a<0,~b>0$}
[.53\linewidth]
{\begin{tikzpicture}[scale=0.9, every node/.style={scale=0.9}]
\draw [fill] (0,0) ellipse (0.1cm and 0.1cm);
\draw [fill] (1,-3) ellipse (0.1cm and 0.1cm);
\draw [fill] (0,-4) ellipse (0.1cm and 0.1cm);
\draw [fill] (-1,-2) ellipse (0.1cm and 0.1cm);
\draw (0,0) -- (1,-3);
\draw (-1,-2) -- (0,-4);
\draw (0,0) -- (-1,-2) ;
\draw (1,-3) -- (0,-4);
\node[above] at (0,1) {\rule{0em}{1em}}; 
\node[above] at (0,0.2) {$\mu = b - a (\lambda +k)$};
\node[right] at (1.4,-3) {$\mu = b-am $};
\node[left] at (-1.4,-2) {$\mu = -a(\lambda+k )$};
\node[below] at (0,-4.2) {$\mu = 0$};
\node[right] at (-1.2,-3.2) {$-a$};
\node[right] at (0.5, -3.6) {$b-am$};
\node[left] at (-0.7,-1.0) {$b$};
\node[right] at (0.6,-1.3) {$-a$};
\node[left] at (0.9,-3) {R};
\node[right] at (-1,-2) {P};
\node[right] at (0,0) {Q};
\node[right] at (0,-4) {S};
\end{tikzpicture}}
\subcaptionbox{When $a<0,~b<0$, and $am-b>0$}
[.43\linewidth]
{\begin{tikzpicture}[scale=0.9, every node/.style={scale=0.9}]
\draw [fill] (0,0) ellipse (0.1cm and 0.1cm);
\draw [fill] (1,-3) ellipse (0.1cm and 0.1cm);
\draw [fill] (0,-4) ellipse (0.1cm and 0.1cm);
\draw [fill] (-1,-2) ellipse (0.1cm and 0.1cm);
\draw (0,0) -- (1,-3);
\draw (-1,-2) -- (0,-4);
\draw (0,0) -- (-1,-2) ;
\draw (1,-3) -- (0,-4);
\node[above] at (0,1) {\rule{0em}{1em}}; 
\node[above] at (0.0,0.2) {$\mu = -b - a (\lambda -k-\epsilon_{m})$};
\node[right] at (1.2,-3) {$\mu = am-b $};
\node[left] at (-1.2,-2) {$\mu = -a(\lambda-k-\epsilon_{m})$};
\node[below] at (0,-4.2) {$\mu = 0$};
\node[right] at (-1.2,-3.2) {$-a$};
\node[right] at (0.5, -3.6) {$am-b$};
\node[left] at (-0.7,-1.0) {$-b$};
\node[right] at (0.6,-1.3) {$-a$};
\node[left] at (0.9,-3) {S};
\node[right] at (-1,-2) {Q};
\node[right] at (0,0) {P};
\node[right] at (0,-4) {R};
\end{tikzpicture}}
\subcaptionbox{When $a<0,~b<0$, and $am-b<0$}
[.55\linewidth]
{\begin{tikzpicture}[scale=0.9, every node/.style={scale=0.9}]
\draw [fill] (0,0) ellipse (0.1cm and 0.1cm);
\draw [fill] (1,-1.5) ellipse (0.1cm and 0.1cm);
\draw [fill] (0,-4) ellipse (0.1cm and 0.1cm);
\draw [fill] (-1,-2.5) ellipse (0.1cm and 0.1cm);
\draw (0,0) -- (1,-1.5);
\draw (1,-1.5) -- (-1,-2.5);
\draw (0,0) .. controls (-4,-1.5) and (-2,-3.5) .. (0,-4);
\draw (-1,-2.5) -- (0,-4);
\node[right] at (-0.35,-3.2) {$b-am$};
\node[right] at (- 0.4,-1.8) {$-a$};
\node[left] at (-2.2,-1.5) {$-a$};
\node[right] at (0.7,-0.5) {$-b$};
\node[above] at (0,0.2) {$\mu = -a(\lambda + k )$};
\node[right] at (-0.6,-2.5) {$\mu = b-am$};
\node[right] at (0.7,-1.8) {$\mu = b-a(\lambda + k)$};
\node[below] at (0, -4.2) {$\mu = 0$};
\node[right] at (0,0) {P};
\node[left] at (0.9,-1.4) {Q};
\node[left] at (-1,-2.5) {R};
\node[right] at (0,-4) {S};
\end{tikzpicture}}
\caption{Graphs of circle actions whose fixed points are isolated}\label{fig:GraphsIsolatedFixedPoints}
\end{figure}

\section{Equivalent circle actions}

A single equivalence class of Hamiltonian circle actions on $\SSS$ or $\CCC$ may be represented by more than one triple $(a,b;m)$, and this happens whenever the associated labelled graphs are the same up to the action of $\AGL(1,\Z)$ on the moment map labels. 

\subsection{Equivalent circle actions in \texorpdfstring{$\T_{m}$}{Tm}}
The action $S^{1}(-a,-b;m)$ is equivalent to $S^{1}(a,b;m)$ since they only differ by a reparametrization of the circle. It follows that in 
Figure~\ref{fig:GraphsWithFixedSurfaces}, the graphs (A), (C), and (E) are $\AGL(1,\Z)$-equivalent, respectively, to the graphs (B), (D), and (F). Similarly, the graphs (G), (H), and (I) in Figure~\ref{fig:GraphsIsolatedFixedPoints} are $\AGL(1,\Z)$-equivalent, respectively, to the graphs (L), (K), and (J). Other examples of equivalent subcircles in $\T_{m}$ are obtained by letting the normalizer $N(\T_{m})$ act on the circle subgroups of $\T_{m}$. For instance, $S^{1}(a,b;m)$ is conjugated to $S^{1}(-a,b-am;m)$ by the adjoint of the element $\Lambda\in N(\T_{m})$ of Proposition~\ref{prop:NormalizersHirzebruch}.

\begin{prop}\label{prop:EquivalentSubcircles}
Up to reparametrizations, the only subcircles of $\T_{m}$ that are equivalent to $S^{1}(a,b;m)$ are the elements of the orbit of $S^{1}(a,b;m)$ under the action of the Weyl group $W(\T_{m})$. 
\end{prop}
\begin{proof}
We have to show that if $S^{1}(a,b;m)$ and $S^{1}(c,d;m)$ are in the same conjugacy class, then there is an element of $N(\T_{m})$ taking one circle to the other. As conjugation of actions corresponds to  uniform translation of the moment map labels, this will follow from a systematic inspection of the possible labelled graphs.

The first cases to consider are actions $S^{1}(a,b;m)$ whose fixed points are all isolated, which means that the weights $a$, $b$, and $am-b$ are all non-zero. Two equivalent graphs arising from normalized moment maps must have the same moment map values at the maximum, and their respective weights at the maximum and at the minimum must also coincide. 

\begin{table}[H]
\begin{tabular}{|l||c|c|c|}
\hline
Type of graph & Weights max & Weights min & Moment at max\\
\hline
\hline
G: $a>0$, $b>0$, $am-b>0$ & $-a$, $b-am$ & $a$, $b$ & $a(\lambda+k)$\\
\hline
H: $a>0$, $b>0$, $am-b<0$ & $-a$, $am-b$ & $a$, $b$ & $b+a(\lambda-k-\epsilon_{m})$ \\
\hline
I: $a>0$, $b<0$ & $-a$, $b-am$ & $a$, $-b$ & $a(\lambda+k)-b$\\
\hline
J: $a<0$, $b>0$ & $a$, $-b$ & $-a$, $b-am$ & $b-a(\lambda+k)$\\
\hline
K: $a<0$, $b<0$, $am-b>0$ & $a$, $b$ & $-a$, $am-b$ & $-b-a(\lambda-k-\epsilon_{m})$\\
\hline
L: $a<0$, $b<0$, $am-b<0$ & $a$, $b$ & $-a$, $b-am$ & $-a(\lambda+k)$\\
\hline
\end{tabular}
\caption{Weights and moment map values for graphs (G) -- (L)}\label{table:WeightsMomentMapValuesGL}
\end{table}
We start by assuming $a>0$, $b>0$, and $m>0$, which implies that the graph of $S^{1}(a,b;m)$ is of type (G) or (H).

Assume the graph of $S^{1}(a,b;m)$ is of type (G), that is, $am-b>0$, and suppose $S^{1}(c,d;m)$, $(c,d)\neq (a,b)$, is in the same conjugacy class.

\begin{itemize}
\item Suppose the graph of $S^{1}(c,d;m)$ is also of type (G). Looking at the weights at the minimum, we see that $(c,d)=(b,a)$ is the only non-trivial possibility. Looking at the moment map values at the maximum, we conclude that $a(\lambda+k)=b(\lambda+k)$, that is, $a=b$, contradicting the fact that $(a,b)\neq(c,d)$.
\item Suppose the graph of $S^{1}(c,d;m)$ is of type (H). Looking again at the weights at the minimum, we see that $c=b$, and $d=a$, and that $a\neq b$ as before. The sets of weights at the maximum must also be equal, that is, $\{-a,b-am\}=\{-c, cm-d\}=\{-b,bm-a\}$, which implies that $a=a-bm$, contradicting the fact that $bm\neq 0$.
\item Suppose the graph of $S^{1}(c,d;m)$ is of type (I), which implies $c>0$, $d<0$. Looking at the weights at the minimum, we must have either $(c,d)=(a,-b)$ or $(c,d)=(b,-a)$. In the former case, looking at the moment map values at the maximum, we conclude that $a(\lambda+k)=c(\lambda+k)-d=a(\lambda+k)+b$, that is, $b=0$, which is a contradiction. In the later case, the weights at the maximum become $\{-c, d-cm\} = \{-b, -a-bm\}$, and we must have $\{-b, -a-bm\}=\{-a, b-am\}$ as sets. Since $a+bm\neq a$, we must have $b=a$, which implies $a=b=1$, $c=1$, and $d=-1$. The moment map values at the maximum are then $a(\lambda+k)=(\lambda+k)$ and $c(\lambda+k)-d=(\lambda+k)+1$, which are not equal. 
\item Suppose the graph of $S^{1}(c,d;m)$ is of type (J), which implies $c<0$, $d>0$. Looking at the weights at the minimum, we must have $\{a,b\}=\{-c,d-cm\}$ as sets. So, either $(c,d)=(-a,b-am)$ or $(c,d)=(-b,a-bm)$. In the first case, $b-am<0$, which is impossible at the minimum. In the second case, looking at the moment map values at the maximum, we conclude that $a(\lambda+k)=d-c(\lambda+k)=a-bm+b(\lambda+k)=a+b(\lambda-k)$. This implies $a(\lambda+k-1)/(\lambda+k)=b$. As $a$ and $b$ are non-zero integers, this is impossible.

\item Suppose the graph of $S^{1}(c,d;m)$ is of type (K), which implies $c<0$, $d<0$, and $cm-d>0$. Comparing the weights at the minimum, we see that either $(c,d)=(-a, -am-b)$ or $(c,d)=(-b,-bm-a)$ with $a\neq b$. In the former case, comparing the weights at the maximum yields $\{-a,b-am\}=\{c,d\}=\{-a,-am-b\}$. This implies $b=0$, which is excluded. In the former case, we must have $a\neq b$ and $\{-a,b-am\}=\{c,d\}=\{-b,-bm-a\}$, which again implies $b=0$.
\item Finally, suppose the graph of $S^{1}(c,d;m)$ is of type (K), which implies $c<0$, $d<0$, and $cm-d<0$. Comparing the weights at the minimum, we see that either $(c,d)=(-a, b-am)$ or $(c,d)=(-b,a-bm)$. In the former case, we get the conjugate circle $S^{1}(-a,b-am;m)$. In the later case, the moment map values at the maximum gives $a(\lambda+k)=-c(\lambda+k)=b(\lambda+k)$, that is, $a=b$. Then $(c,d)=(-a, b-am)$ and $S^{1}(c,d;m)=S^{1}(-a,b-am;m)$ as before.
\end{itemize}
We conclude that the only action $S^{1}(c,d;m)$ conjugated to an action $S^{1}(a,b;m)$ of type (G) is $S^{1}(-a,b-am;m)$.

Assume now that the graph of $S^{1}(a,b;m)$ is of type (H), that is, $a>0$, $b>0$, and $am-b<0$. By transitivity, we already know that an action $S^{1}(c,d;m)$ in the same conjugacy class cannot be of type (G) or (L).

\begin{itemize}
\item Suppose the graph of $S^{1}(c,d;m)$ is also of type (H). Comparing the weights at the minimum, we see that we must have $(c,d)=(b,a)$. Looking at the weights at the maximum, we must have $\{-a,am-b\}=\{-c,cm-d\}$ as sets. Since $(c,d)\neq(a,b)$, the only possibility is $a=d-cm=a-bm$, which implies $b=0$.
\item Suppose the graph of $S^{1}(c,d;m)$ is of type (I). Comparing the weights at the minimum, we see that we must have either $(c,d)=(a,-b)$ or $(c,d)=(b,-a)$ with $a\neq b$. In the first case, comparing the weights at the maximum, we should have $\{-a,am-b\}=\{-c,d-cm\}=\{-a,-b-am\}$. This implies $a=0$, which is excluded. Similarly, in the second case, we get $\{-a,am-b\}=\{-c,d-cm\}=\{-b,-b-am\}$ with $a\neq b$. This again implies  $a=0$.
\item Suppose the graph of $S^{1}(c,d;m)$ is of type (J). Looking at the weights at the minimum, we must have $\{a,b\}=\{-c,d-cm\}$ as sets. So, either $(c,d)=(-a,b-am)$ or $(c,d)=(-b,a-bm)$. In the first case, we obtain the conjugate action $S^{1}(-a,b-am;m)$. In the second case, looking at the weights at the maximum, we must have $\{-a,am-b\}=\{c,-d\}=\{-b, bm-a\}$. If $a=b=1$, this gives
us the conjugate action $S^{1}(-1,1-m;m)$. If $a=a-bm$, then $b=0$, which is excluded.
\item Finally, suppose the graph of $S^{1}(c,d;m)$ is of type (K). Looking at the weights at the minimum, we must have $\{a,b\}=\{-c,cm-d\}$ as sets. So, either $(c,d)=(-a,-b-am)$ or $(c,d)=(-b,-a-bm)$. Looking at the weights at the maximum, we must also have $\{-a,am-b\}=\{c,d\}$. If $(c,d)=(-a,-b-am)$, then $am-b=-am-b$, which implies $a=0$. Instead, if $(c,d)=(-b,-a-bm)$, then either $a=b=1$, which is equivalent to the previous case, or $a=a+bm$ which forces $b=0$.
\end{itemize}
We conclude that $S^{1}(-a,b-am;m)$ is the only subcircle of $\T_{m}$ conjugated to an action $S^{1}(a,b;m)$ of type (H).

Assume now that the graph of $S^{1}(a,b;m)$ is of type (I), that is, $a>0$, and $b<0$. By transitivity, we already know that an action $S^{1}(c,d;m)$ in the same conjugacy class cannot be of type (G), (H), (J), or (L). By symmetry of Table~\ref{table:WeightsMomentMapValuesGL} under a change of sign of the pair $(a,b)$, and comparing with actions of types (H) and (J), we see that the only action $S^{1}(c,d;m)$ of type (K) conjugated to $S^{1}(a,b;m)$ is $S^{1}(-a,b-am;m)$. 

\begin{itemize}
\item Suppose the graph of $S^{1}(c,d;m)$ is also of type (I). Comparing the weights at the minimum, we see that we must have $(c,d)=(-b,-a)$. Looking at the weights at the maximum, we must also have $\{-a,b-am\}=\{-c, d-cm\}=\{b,-a+bm\}$. If $b=-a$, then $(a,b)=(1,-1)=(c,d)$. If $b=b-am$, then we must have $a=0$, which is impossible.

\end{itemize}
This shows that for an action $S^{1}(a,b;m)$ of type (I), the only other action $S^{1}(c,d;m)$ in the same conjugacy class is $S^{1}(-a,b-am;m)$.

This concludes the proof of the proposition for actions $S^{1}(a,b;m)$ whose fixed points are all isolated, in the case $m>0$. When $m=0$ and $\lambda>1$, the graphs (G) and (L) no longer exist, and the four remaining graphs (H)--(K) only differ by the action of $D_{2}$ on the pair $(a,b)$, that is, by a change of sign of $a$ or $b$. When $\lambda=1$, we can also interchange $a$ and $b$ to obtain an equivalent graph, which defines an action of $D_{4}$. Finally, the case of actions with non-isolated surfaces involves the graphs (A)--(F) and is simpler. The details are left to the reader. 
\end{proof}

\subsection{Toric extensions of circle actions}
\begin{defn}
We shall say a circle action $S^1(a,b; m)$ \emph{extends} to a toric action $\T_{n}$ if it is $S^1$-equivariantly symplectomorphic to a circle action of the form $S^1(c,d; n)$. 
\end{defn}

In this section, we determine all possible toric extensions of $S^{1}(a,b;m)$. We begin with the exceptional case $\lambda=1$.
\begin{prop}
Consider $(\SSS,\oml)$ with $\lambda=1$. Then the only Hamiltonian circle actions are of the form $S^1(a,b; 0)$. In particular, they can only extend to the torus $\T_0$. 
\end{prop}
\begin{proof}
This follows directly from Theorem~\ref{lemma_number_of_torus_actions}.
\end{proof}

By Lemma~\ref{lemma_torus_action_-vecurve}, a toric extension of $S^{1}(a,b;m)$ to $\T_{n}$, $n\geq 1$, implies the existence of an invariant sphere of self-intersection $-n$ and of positive symplectic area. These two numerical invariants can be read from labelled graphs using Lemma~\ref{weight} and Lemma~\ref{Lemma:SymplecticArea}. As we will see, this imposes enough conditions on the triple $(a,b;m)$ to determine all possible embeddings $S^{1}(a,b;m)\into \T_{n}$, $n\neq m$.
\begin{prop}\label{prop:AtMostTwoExtensions}
Consider a Hamiltonian circle action $S^1(a,b;m)$ with $\lambda >1$. Under the following numerical conditions on $a,b,m,\lambda$, the circle action only extends to the toric action~$\T_m$:
\begin{itemize}
    \item when $a\neq\pm 1$;
    \item when $b=0$ or $b=am$;
    \item when $a = \pm 1$ and $2 \lambda \leq |2b-am|+\epsilon_{m}$.
\end{itemize}
In all other cases, the circle action may extend to at most two inequivalent toric actions, namely,
\begin{itemize}
    \item when $a=\pm 1$, $2\lambda > |2b-am|+\epsilon_{m}$, and $b \not\in \{0,am\}$, the circle action $S^1(a,b;m)$ only extends to the toric action $\T_{m}$ and, possibly, to $\T_{|2b-am|}$.

\end{itemize}
\end{prop}
\begin{proof}
Suppose $S^{1}(a,b;m)$ is equivariantly symplectomorphic to $S^{1}(c,d;n)$ for some $n\neq m$. Note that we necessarily have $m=n\mod 2$. By assumption, the two actions share the same normalized labelled graph.

We first consider an action $S^{1}(a,b;m)$ whose fixed points are all isolated. For the moment, let's also assume that $m\neq 0$ and $n\neq 0$. We observe that if the $\T_{n}$ invariant curve $C_{n}$ of self-intersection $-n$ has non-trivial isotropy, then it must appear as some edge in the graph of $S^{1}(a,b;m)$. As the edge connecting the vertices $Q$ and $R$ is the only one that corresponds to an invariant sphere of negative self-intersection, namely $-m$, we conclude that $m=n$. This contradiction shows that $C_{n}$ must have trivial isotropy. By symmetry, the same is true of the invariant curve $C_{m}$ of self-intersection $-m$. It follows that $a=\pm 1$ and $c=\pm 1$. 

Assume $a=1$. Because there are no fixed surfaces, we know that $b\neq 0$ and $m-b\neq 0$. Figure~\ref{fig:PossibleGraphsToricExtensions} shows the three possible graphs for $S^{1}(1,b;m)$, which are then of types (G), (H), or (I). The dashed edges represent the possible locations of the $\T_{n}$ invariant curves $C_{n}$ and $C_{-n}$. We can compute the self-intersection of $C_{n}$ and $C_{-n}$ by applying Lemma~\ref{weight} to the normal bundle of these invariant spheres, and using the configurations of weights shown in Figure~\ref{fig:Self-Intersection}. The self-intersection of the curve between $P$ and $R$ is $2b-m$, while the self-intersection of the curve between $Q$ and $S$ is $m-2b$. Set $n=|2b-m|$. For the toric action $\T_{n}$ to exist, we must also have $0<2\oml (C_{n}) = 2\lambda-|2b-m|-\epsilon_{m}$. We conclude that the action $S^{1}(1,b;m)$, $m\neq 0$, may extend to $\T_{|2b-m|}$ whenever $2\lambda>|2b-m|+\epsilon_{m}$, and that this is the only other possible toric extension.

\begin{figure} 
\centering
~
\subcaptionbox*{}{} 
\setcounter{subfigure}{6} 
\subcaptionbox{When $a=1,~b>0$, and $m-b>0$\label{fig:a=1|b>0|m-b>0}}
[.43\linewidth]
{\begin{tikzpicture}[scale=0.9, every node/.style={scale=0.9}]
\draw [fill] (0,0) ellipse (0.1cm and 0.1cm);
\draw [fill] (1,-1.5) ellipse (0.1cm and 0.1cm);
\draw [fill] (0,-4) ellipse (0.1cm and 0.1cm);
\draw [fill] (-1,-2.5) ellipse (0.1cm and 0.1cm);
\draw (0,0) -- (1,-1.5);
\draw [dashed, blue] (0,0) -- (-1,-2.5);
\draw [dashed, blue] (1,-1.5) -- (0,-4);
\draw (-1,-2.5) -- (0,-4);
\node[left] at (-0.6,-3.2) {$b$};
\node[right] at (0.7,-0.5) {$m-b$};
\node[above] at (0,0.2) {$\mu = (\lambda + k)$};
\node[right] at (-0.6,-2.5) {$\mu = b$};
\node[right] at (1,-1.8) {$\mu = b + (\lambda - k-\epsilon_{m}) $};
\node[below] at (0, -4.2) {$\mu = 0$};
\node[left] at (-1,-2.5) {Q};
\node[left] at (0.9,-1.4) {R};
\node[right] at (0,0) {S};
\node[right] at (0,-4) {P};
\end{tikzpicture}}
\subcaptionbox{When $a=1,~b>0$ and $m-b<0$
}
[.52\linewidth]
{\begin{tikzpicture}[scale=0.9, every node/.style={scale=0.9}]
\draw [fill] (0,0) ellipse (0.1cm and 0.1cm); 
\draw [fill] (1,-2) ellipse (0.1cm and 0.1cm); 
\draw [fill] (0,-4) ellipse (0.1cm and 0.1cm); 
\draw [fill] (-1,-3) ellipse (0.1cm and 0.1cm); 
\draw [dashed, blue] (0,0) .. controls (-4,-1.5) and (-2,-3.5) .. (0,-4);
\draw [dashed, blue] (1,-2) -- (-1,-3);
\draw (0,0) -- (1,-2);
\draw (-1,-3) -- (0,-4);
\node[right] at (0,-4) {P};
\node[left] at (-1,-3) {Q};
\node[left] at (0.9,-2) {S};
\node[right] at (0,0) {R};
\node[above] at (0,0.2) {$\mu = b +  (\lambda -k-\epsilon_{m})$};
\node[right] at (1.2,-2) {$\mu = (\lambda + k) $};
\node[right] at (-0.7,-3) {$\mu = b$};
\node[right] at (-0.4,-3.5) {$b$};
\node[right] at (0.7,-1) {$b-m$};
\node[below] at (0,-4.2) {$\mu = 0$};
\end{tikzpicture}
}
\subcaptionbox{When $a=1,~b<0$\label{fig:a=1|b>0}}
[.5\linewidth]
{\begin{tikzpicture}[scale=0.9, every node/.style={scale=0.9}]
\draw [fill] (0,0) ellipse (0.1cm and 0.1cm);
\draw [fill] (1,-3) ellipse (0.1cm and 0.1cm);
\draw [fill] (0,-4) ellipse (0.1cm and 0.1cm);
\draw [fill] (-1,-2) ellipse (0.1cm and 0.1cm);
\draw [dashed, blue] (0,0) .. controls (-4,-1.5) and (-2,-3.5) .. (0,-4);
\draw [dashed, blue] (-1,-2) -- (1,-3);
\draw (0,0) -- (-1,-2) ;
\draw (1,-3) -- (0,-4);
\node[above] at (0,1) {\rule{0em}{1em}}; 
\node[above] at (0,0.2) {$\mu = (\lambda +k)-b$};
\node[right] at (1.4,-3) {$\mu = -b $};
\node[right] at (-0.8,-1.9) {$\mu = (\lambda-k-\epsilon_{m})$};
\node[below] at (0,-4.2) {$\mu = 0$};
\node[right] at (0.5, -3.6) {$-b$};
\node[right] at (-0.4,-1.0) {$m-b$};
\node[right] at (1,-3) {P};
\node[left] at (-1,-2) {R};
\node[right] at (0,0) {S};
\node[right] at (0,-4) {Q};
\end{tikzpicture}}
\caption{Graphs of $S^{1}(1,b;m)$ with possible locations of $C_{n}$ and $C_{-n}$}
\label{fig:PossibleGraphsToricExtensions}
\end{figure}

\begin{figure} 
\centering
   
\begin{tikzpicture}
\draw [fill] (0,0) ellipse (0.1cm and 0.1cm); 
\draw [fill] (4,0) ellipse (0.1cm and 0.1cm); 
\draw [fill] (7,0) ellipse (0.1cm and 0.1cm); 
\draw [fill] (11,0) ellipse (0.1cm and 0.1cm);
\draw [dashed, blue] (0,0) -- (4,0); 
\draw [dashed, blue] (7,0) -- (11,0);
\draw (0,-0.5) -- (0,0.5);
\draw (4,-0.5) -- (4,0.5);
\draw (7,-0.5) -- (7,0.5);
\draw (11,-0.5) -- (11,0.5);
\node[left] at (0,0) {$1\,$}; 
\node[above] at (0,0.5) {$b$}; 
\node[right] at (4,0) {$\,-1$}; 
\node[above] at (4,0.5) {$m-b$}; 
\node[right] at (0,-0.25) {$P$};
\node[left] at (4,-0.25) {$R$};
\node[above,blue] at (2,0) {$1$};
\node[left] at (7,0) {$1\,$}; 
\node[above] at (7,0.5) {$-b$}; 
\node[right] at (11,0) {$\,-1$};
\node[above] at (11,0.5) {$b-m$}; 
\node[right] at (7,-0.25) {$Q$};
\node[left] at (11,-0.25) {$S$};
\node[above,blue] at (9,0) {$1$};
\end{tikzpicture}
\caption{Configurations of weights along $C_{\pm n}$ when $a=1$}
\label{fig:Self-Intersection}
\end{figure}
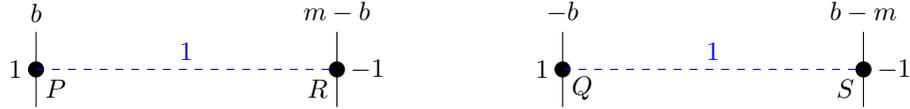

When $a=-1$, the same arguments apply to $S^{1}(-1,b;m)$ whose possible graphs are now of types (J), (K), and (L). The self-intersections of the curves $C_{\pm 1}$ are now $\pm |2b+m|$. We conclude that the action $S^{1}(-1,b;m)$ may extend to $\T_{|2b+m|}$ whenever $m\neq 0$ and $2\lambda>|2b+m|+\epsilon_{m}$.

In the special case $m=0$, any invariant sphere appearing in the graph of $S^{1}(a,b;0)$ has zero self-intersection. If $S^{1}(a,b;0)$ extends to $\T_{n}$ with $n\geq 2$, it again follows that the invariant curve $C_{n}$ of self-intersection $-n$ must have trivial isotropy. Consequently, $c=\pm 1$, and at least one of the weights $a$ or $b$ must be $\pm1$. If $a=1$, the possible graphs of $S^{1}(1,b;0)$ are the graphs (H) and (I) of Figure~\ref{fig:PossibleGraphsToricExtensions}. Then, the same argument as before shows that the self-intersection of invariant curves are $0$ and $\pm 2b$. Consequently, the action $S^{1}(1,b;0)$ may extend to $\T_{|2b|}$ provided $2\lambda>|2b|$. Similarly, when $a=-1$, the action $S^{1}(-1,b;0)$ may extend to $\T_{|2b|}$ whenever $2\lambda>|2b|$.

When $m=0$ and $b=1$, then the possible self-intersections of invariant curves are $0$ and $\pm 2a$. However, the two possible graphs are now of type (H) and (J). In both cases, the area of the tentative invariant curve $C_{\pm n}$ connecting the two interior fixed points is negative. It follows that it is impossible to have $m=0$ and $b=1$. Similarly, if $b=-1$, the possible graphs are of type (I) and (K) and, as before, the area of the tentative invariant curve $C_{\pm n}$ connecting the two interior fixed points is negative. Consequently, we cannot have $m=0$ and $b=\pm 1$. This concludes the proof of the statement for actions whose fixed points are all isolated.

\begin{figure} 
\centering
~
\subcaptionbox*{}{} 
\setcounter{subfigure}{7} 
\subcaptionbox{When $m=0$, $a>0,~b=1$}
[.45\linewidth]
{\begin{tikzpicture}[scale=0.9, every node/.style={scale=0.9}]
\draw [fill] (0,0) ellipse (0.1cm and 0.1cm); 
\draw [fill] (1,-2) ellipse (0.1cm and 0.1cm); 
\draw [fill] (0,-4) ellipse (0.1cm and 0.1cm); 
\draw [fill] (-1,-3) ellipse (0.1cm and 0.1cm); 
\draw [dashed, blue] (0,0) .. controls (-4,-1.5) and (-2,-3.5) .. (0,-4);
\draw[dashed, blue] (-1,-3) -- (1,-2);
\draw (0,0) -- (-1,-3) ;
\draw (1,-2) -- (0,-4);
\node[left] at (-1.1,-3) {Q};
\node[left] at (0.9,-1.9) {S};
\node[right] at (0,0) {R};
\node[above] at (0,0.2) {$\mu = 1 + a (\lambda -k-\epsilon_{m})$};
\node[right] at (1.2,-2) {$\mu = a(\lambda + k) $};
\node[right] at (-0.9,-3) {$\mu = 1$};
\node[right] at (0,-4) {P};
\node[right] at (0.5, -3) {$a$};
\node[left] at (-0.5,-1.5) {$a$};
\node[below] at (0,-4.2) {$\mu = 0$};
\end{tikzpicture}}
\setcounter{subfigure}{9} 
\subcaptionbox{When $m=0$, $a<0,~b=1$\label{fig:a<0|b=1}}
[.5\linewidth]
{\begin{tikzpicture}[scale=0.9, every node/.style={scale=0.9}]
\draw [fill] (0,0) ellipse (0.1cm and 0.1cm);
\draw [fill] (1,-3) ellipse (0.1cm and 0.1cm);
\draw [fill] (0,-4) ellipse (0.1cm and 0.1cm);
\draw [fill] (-1,-2) ellipse (0.1cm and 0.1cm);
\draw (0,0) -- (1,-3);
\draw (-1,-2) -- (0,-4);
\draw [dashed, blue] (0,0) .. controls (-4,-1.5) and (-2,-3.5) .. (0,-4);
\draw[dashed, blue] (-1,-2) -- (1,-3);
\node[above] at (0,1) {\rule{0em}{1em}}; 
\node[above] at (0,0.2) {$\mu = 1 - a (\lambda +k)$};
\node[right] at (1.4,-3) {$\mu = 1 $};
\node[above] at (-1,-2) {$\mu = -a(\lambda+k)$};
\node[below] at (0,-4.2) {$\mu = 0$};
\node[right] at (-1.2,-3.2) {$-a$};
\node[right] at (0.6,-1.3) {$-a$};
\node[left] at (0.9,-3) {R};
\node[left] at (-1,-2.2) {P};
\node[right] at (0,0) {Q};
\node[right] at (0,-4) {S};
\end{tikzpicture}}
\caption{Graphs of $S^{1}(a,1;0)$ with impossible locations of $C_{n}$ and $C_{-n}$}
\label{fig:ImpossibleGraphsToricExtensions}
\end{figure}
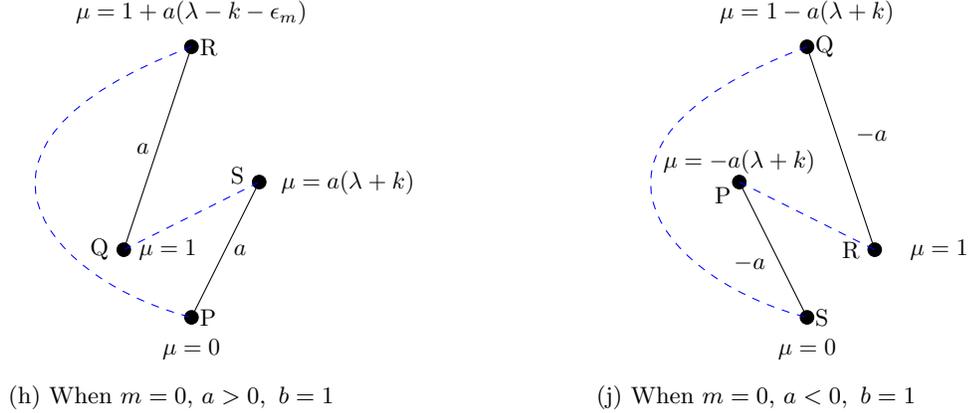

Now suppose $S^{1}(a,b;m)$ has non-isolated fixed points, that is, $a=0$, or $b=0$, or $b=am$. As the moment map values of graphs (A), (B), (E), and (F) depend only on $m$, it is impossible to get the same graphs from another action $S^{1}(c,d;n)$ with $n\neq m$. The same remark applies to the area labels of graphs (C) and (D), showing that $S^{1}(a,b;m)$ only extends to $\T_{m}$.
\end{proof}

It remains to investigate whether the action $S^{1}(\pm 1,b;m)$ extends to $\T_{|2b-m|}$ when $2\lambda > |2b-am|+\epsilon_{m}$ and $b \not\in \{0,am\}$. A straightforward but lengthy comparison of the possible graphs of $S^{1}(\pm1,b;m)$ and $S^{1}(\pm,d;|2b-am|)$ shows that this is always the case and yields the equivalences stated in the following two corollaries. The graphs giving the first equivalence are shown in Figure~\ref{fig:ExampleTwoToricExtensions}. The other cases are left to the reader.

\begin{cor} \label{cor:CircleExtensionsWith_a=1}
Consider a circle action $S^1(1,b;m)$ on $(\SSS,\oml)$  or $(\CCC,\oml)$ and suppose $2\lambda > |2b-m|+\epsilon_{m}$. Then under the following numerical conditions on $b$ and $m$, the  $S^1(1,b;m)$ action extends to the toric action $\T_{|2b-m|}$ and is equivariantly symplectomorphic to the following subcircle in $\T_{|2b-m|}$\,:
\begin{enumerate}
    \item if $b>0$ and $b>m$, then $S^1(1,b; m)$ is equivalent to $S^1(1,b; |2b-m|)$;
    \item if $b>0$, $m>b$, and $2b-m < 0$, then $S^1(1,b; m)$ is equivalent to $S^1(1,-b; |2b-m|)$;
    \item if $b>0$, $m>b$, and $2b-m > 0$, then $S^1(1,b; m)$ is equivalent to $S^1(1,b;|2b-m|)$;
    \item finally, if $b<0$, then $S^1(1,b; m)$ is equivalent to $S^1(1,-b;|2b-m|)$.\qed
\end{enumerate}
\end{cor}

\begin{cor}\label{cor:CircleExtensionsWith_a=-1}
Consider the $S^1$ actions $S^1(-1,b;m)$ on $(\SSS,\oml)$  or $(\CCC,\oml)$ and suppose $2\lambda > |2b+m| + \epsilon_m$. Then under the following numerical conditions on $b$ and $m$, the  $S^1(-1,b;m)$ action extends to the toric action $\T_{|2b+m|}$ and is equivariantly symplectomorphic to the following subcircle in $\T_{|2b+m|}$\,:
\begin{enumerate}
\item if $b<0$ and $m>-2b$, then $S^1(-1,b;m)$ is equivalent to\\ $S^1(-1,-b; |2b+m|)$;
\item if $b<0$, $m>-b$, and $-2b>m$, then $S^1(-1,b;m)$ is equivalent to 
$S^1(-1,b; |2b+m|)$;
\item if $b<0$ and $-b>m$, then $S^1(-1,b;m)$ is equivalent to $S^1(-1,b; |2b+m|)$;
\item if $b>0$, then $S^1(-1,b; m)$ is equivalent to $S^1(-1,-b; |2b+m|)$.\qed
\end{enumerate}
\end{cor}

\begin{figure}[H] 
\centering
\subcaptionbox*{$S^1(1,b;m)$ of type (H)\label{fig:ExtendedGraph1}}
[.38\linewidth]
{\begin{tikzpicture}[scale=0.8, every node/.style={scale=0.9}]
\draw [fill] (0,0) ellipse (0.1cm and 0.1cm); 
\draw [fill] (1,-2) ellipse (0.1cm and 0.1cm); 
\draw [fill] (0,-4) ellipse (0.1cm and 0.1cm); 
\draw [fill] (-1,-3) ellipse (0.1cm and 0.1cm); 
\draw (0,0) -- (1,-2);
\draw (-1,-3) -- (0,-4);
\draw[dashed, blue] (0,0) -- (-1,-3) ;
\draw[dashed, blue] (1,-2) -- (0,-4);
\node[right] at (-0.9,-3) {Q};
\node[left] at (0.9,-2) {S};
\node[right] at (0,0) {R};
\node[above] at (0,0.2) {$\mu = b + \lambda-k-\epsilon_{k}$};
\node[right] at (1.3,-2) {$\mu = \lambda+k$};
\node[left] at (-1.4,-3) {$\mu = b$};
\node[right] at (0,-4) {P};
\node[right] at (-1,-3.7) {$b$};
\node[right] at (0.5, -3) {$a=1$};
\node[left] at (-0.5,-1.5) {$a=1$};
\node[right] at (0.7,-1) {$b-m$};
\node[below] at (0,-4.2) {$\mu = 0$};
\end{tikzpicture}
}
~
\subcaptionbox*{$S^1(1,b;2b-m)$ of type (G)\label{sndstrata}}
[.61\linewidth]
{\begin{tikzpicture}[scale=0.8, every node/.style={scale=0.9}]
\draw [fill] (0,0) ellipse (0.1cm and 0.1cm);
\draw [fill] (1,-1.5) ellipse (0.1cm and 0.1cm);
\draw [fill] (0,-4) ellipse (0.1cm and 0.1cm);
\draw [fill] (-1,-2.5) ellipse (0.1cm and 0.1cm);
\draw (0,0) -- (1,-1.5);
\draw[dashed, blue] (1,-1.5) -- (-1,-2.5);
\draw[dashed, blue] (0,0) .. controls (-4,-1.5) and (-2,-3.5) .. (0,-4);
\draw (-1,-2.5) -- (0,-4);
\node[right] at (-0.35,-3.2) {$b$};
\node[right] at (- 1,-1.8) {$a=1$};
\node[left] at (-2.2,-1.5) {$a=1$};
\node[right] at (0.7,-0.5) {$(2b-m)-b=b-m$};
\node[above] at (0,0.2) {$\mu = \lambda+\frac{(2b-m)-\epsilon_{k}}{2}= b+\lambda - k-\epsilon_{k}$};
\node[right] at (-0.6,-2.5) {$\mu = b$};
\node[right] at (0.6,-1.8) {$\mu=b+\lambda-\frac{(2b-m)-\epsilon_{k}}{2}-\epsilon_{k}$};
\node[right] at (0.91,-2.25) {$=\lambda+k$};
\node[below] at (0, -4.2) {$\mu = 0$};
\node[left] at (-1,-2.5) {Q};
\node[left] at (0.9,-1.4) {R};
\node[right] at (0,0) {S};
\node[right] at (0,-4) {P};
\end{tikzpicture}}
\caption{The equivalence $S^{1}(1,b;m)\sim S^{1}(1,b;2b-m)$ when $b>m$}
\label{fig:ExampleTwoToricExtensions}
\end{figure}
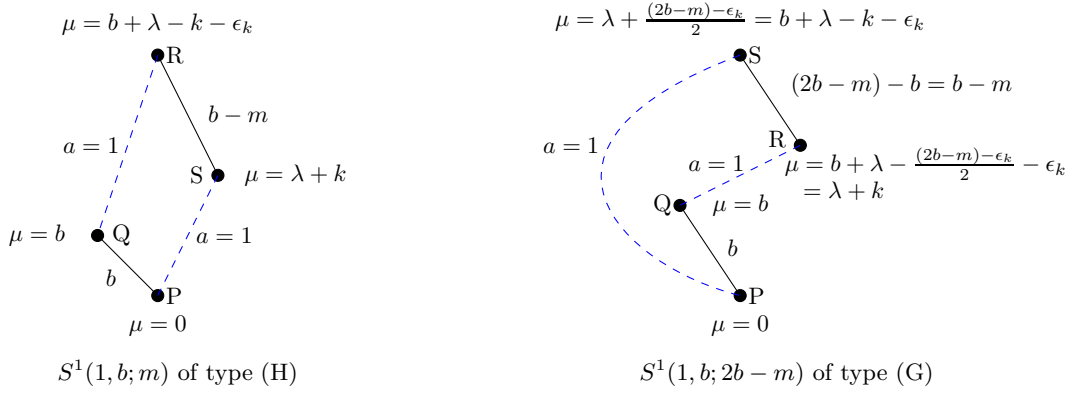


\chapter{Action of \texorpdfstring{$\Symp^{S^1}(\SSS,\oml)$}{Symp(S2xS2)} on \texorpdfstring{$\jsom$}{J\hat{}S1}}\label{ChapterActionOfSymp}
In this chapter we show that the space of $S^1$ invariant compatible almost complex structures $\jsom$ decomposes into strata that are homotopy equivalent to  homogeneous spaces under the action of the equivariant symplectomorphism group. 

\section{\texorpdfstring{$J$}{J}-Holomorphic Preliminaries}

We first recall a few facts about compatible almost complex structures and associated $J$-holomorphic curves.
\begin{defn}[Compatible almost complex structures] 
An almost complex structure $J$ on a symplectic manifold $(M,\om)$ is said to be compatible with $\om$ if $\om(u,Ju)>0$ and $\om(Ju, Jv)=\om(u,v)$ for all non-zero $u,v\in T_* M$.
\end{defn}
\begin{lemma}
The space $\jj_\om = \jj(M,\om)$ of all compatible almost complex structures on a symplectic manifold $(M,\om)$ is non-empty and contractible.
\end{lemma}
\begin{defn}{\textit{$J$-holomorphic spheres}:}
Let $(M,\om)$ be a symplectic manifold endowed with a compatible almost complex structure $J$. A rational $J$-holomorphic map, also called a parametrized $J$-holomorphic sphere, is a $C^\infty$ map
\[
u: ({S}^2, j) \lto (M,\om,J)
\]
satisfying the Cauchy-Riemann equation
\[
\delbar_J(u)=\frac{1}{2}(du\circ j - J\circ du) = 0
\]
where $j$ is the usual complex structure on the sphere. The image of a $J$-holomorphic rational map is called a rational $J$-holomorphic curve or simply a $J$-curve.
\end{defn}
\begin{defn}[Multi-covered and simple maps]
We say that a $J$-ho\-lo\-mor\-phic map $u:\mathbb{C}P^1 \longrightarrow (M,J)$ is multi-covered if $u = {u}^\prime \circ f$, where $f:\mathbb{C}P^1 \to \mathbb{C}P^1$ is a holomorphic map of degree greater than one and where $u':\mathbb{C}P^1 \to (M,J)$ is a $J$-holomorphic map. We call a $J$-holomorphic map simple if it is not multi-covered.
\end{defn}
\begin{remark}
We usually assume that a $J$-holomorphic map is somewhere injective, meaning that $\exists z\in {S}^2$ such that $du_z \neq 0$ and $u^{-1}u(z) = z$. In particular, somewhere injective maps do not factor through multiple covers $h:S^2\to S^2$. 
\end{remark}
\begin{defn}[Moduli spaces of $J$-holomorphic maps or curves]
Let $(M,\om)$ be a symplectic manifold and let $J \in \jj_\om$. Given $A \in H_2(M, \mathbb{Z})$ we denote by $\widetilde{\Mm}(A,J)$ the space of all $J$-holomorphic, somewhere injective maps representing the homology class $A$. The Mobius group $G =\PSL(2,\mathbb{C})$ acts freely on this space by reparametrization and the quotient space $\Mm(A,J):=\widetilde{\Mm}(A,J)/G$ is called the moduli space of (unparametrised) $J$-curves in class $A$.
\end{defn}
In dimension 4, the geometric properties of $J$-holomorphic curves are, to a large extend, controlled by homological data. As a result, many properties of complex algebraic curves in complex algebraic surfaces extend to $J$-holomorphic curves in $4$-dimensional symplectic manifolds. Below we list the key properties of $J$-holomorphic curves we will be relying on.
\begin{thm}[Positivity]\label{thm_Positivity}Let $(M,\om)$ be a 4-dimensional symplectic manifold. If a homology class A $\in H_2(M,\mathbb{Z})$ is represented by a nonconstant J-curve for some $J \in \mathcal{J_\om}$ then $\om(A) > 0$.
\end{thm}
\begin{thm}[Fredholm property and automatic regularity.]\label{thm_Regularity}Let $(M,\om)$ be a 4-dimensional symplectic manifold.  
Then the universal moduli space
\[
\widetilde{\Mm}(A,\jj_{\om}) := \{(u,J) \in C^l(S^2,M)\times\jj_{\om}~|~ u\in\widetilde{\Mm}(A,J)\}
\]
with $C^l$-topology ($l \geq 2$) is a smooth Banach manifold and the projection map
\[
\pi_A: \widetilde{\Mm}(A,\jj_{\om}) \longrightarrow \jj_{\om}
\]
is a Fredholm map of index $2(c_1(A) + 2)$ where $c_1 \in H^2(M,\mathbb{Z})$ is the first chern class of $(TM, J)$ (note that the Chern class is independent of choice of $J \in \jj_{\om}$).  An almost complex structure is said to be regular for the class $A$ if it is a regular value for the projection $\pi_A$.  If this is the case then the moduli spaces $\widetilde{\Mm}(A,J)$ and $\Mm(A,J)$ are smooth manifolds of dimensions $2(c_1(A)+2)$ and $2(c_1(A)-1)$ respectively.  The set of regular values $\jj \in \jj_\om$ is a subset of second category and is denoted by $\jj_{\om}^{\textrm{reg}}(A)$. If $J \in  J_\om$ is integrable and $S$ is an embedded $J$-holomorphic sphere with self-intersection number $[S]\cdot [S] \geq -1$, then $J$ is regular for the class $[S]$. In dimension $4$, the same conclusion holds without the integrability assumption.
\end{thm}
\begin{defn}[Cusp Curves] \label{defn_cusp}
Let $(M,\om)$ be a symplectic manifold. Let $J \in J_\om$. A $J$-holomorphic cusp curve $C$ is a connected finite union of $J$-holomorphic curves 
\[C = C_1 \cup C_2 \ldots \cup C_k\] where $C_i = u_i(\mathbb{C}P^1)$ and $u_i: \mathbb{C}P^1 \to (M,J)$ is a (possibly multi-covered) $J$-holomorphic map.
\end{defn}
\begin{thm}[Gromov's compactness theorem]\label{thm_Compactness} Let $(M,\om)$ be a compact symplectic manifold. Let $J_n \in \jj_{\om}$ be a sequence converging to $J$ in the $C^\infty$ topology and let $S_i$ be $J_i$-holomorphic spheres of bounded symplectic area $\om(S_i)$. Then there is a subsequence of the $S_i$ which converges weakly to a $J$-holomorphic curve or cusp-curve $S$. In particular if all the $S_i$'s belong to the class $A$, then $S$ also belongs the class $A$, and any cusp curve defines a homological decomposition of $A=\sum_i A_i$ such that $\om(A_i)>0$.
\end{thm}
\begin{thm}[Positivity of intersections]\label{thm_PositivityIntersections} Let $J \in \joml$ and $A$, $B$ be two distinct $J$-holomorphic curves in a $4$-dimensional manifold. Then they intersect at only finitely many points and each point contributes positively to the intersection multiplicity $[A]\cdot[B]$. Moreover,  $[A]\cdot[B]=1$ iff the curves intersect transversally at exactly one point, while $[A]\cdot[B]= 0$ iff the curves are disjoint.
\end{thm}
As a corollary of Positivity of intersections we have the following result under the presence of a group action. %
\begin{cor}\label{cor_pos}
Let $(M,\om)$ be a symplectic 4-manifold and let $G$ be a compact Lie group acting symplectically on $M$. Suppose that $G$ acts trivially on homology. Let $\J^{G}$ denote the space of $\om$ tame (or compatible) invariant almost complex structures and let $C$ be a $J$ holomorphic curve for some $J \in \J^G$. 
Then,
\begin{enumerate}
    \item if C has  negative self intersection, then $g \cdot C = C$ for all $g \in G$.
    \item If C has zero self intersection, then $g \cdot C = C$ or $g\cdot C \cap C = \emptyset$ for all $g \in G$.
\end{enumerate} 
\end{cor}
\begin{thm}[Adjunction formula]\label{thm_Adjunction} Let $u: ({S}^2, j) \lto (M^4,J)$ be a somewhere injective $J$-holomorphic map representing the homology class $A$ in a $4$-dim\-en\-sio\-nal manifold. Define the virtual genus of $A$ as
\[
g_v(A) = 1+\frac{1}{2}([A]\cdot[A]- c_1(A))
\]
where $c_1(A) = \langle c_1(TM,J), A \rangle$. Then $g_v(A)\geq 0$ with equality if, and only if, the map $u$ is an embedding.
\end{thm}

\section{\texorpdfstring{$J$}{J}-holomorphic spheres in \texorpdfstring{$\SSS$ and $\CCC$}{S2xS2 and CP2\#CP2}}

In this section, will show how the existence of certain $J$-holomorphic spheres in ruled $4$-manifolds induces a natural partition of the space $\jj_\om$. We shall state all the relevant results, and refer the reader to Chapter 9 of \cite{McD} for more details.
\begin{cor}\label{FSimple}
Let $\lambda = l + \delta$ where $l \in \mathbb{N}$ and $0<\delta\leq 1$.
Then we have 
\begin{enumerate}
    \item Any $J$-holomorphic representative of the class $F$ is a simple curve.
    \item The only $J$-holomorphic decomposition of the class $B$ are of the form $B = (B-kF) + kF$, where $0\leq k\leq\ell$. In this decomposition, the $J$-holomorphic representative of the class $(B-kF)$ is an embedded sphere, while the class $kF$ may be represented by a collection of (possibly multiply covered) spheres representing multiples of the class $F$.\qed
\end{enumerate}
\end{cor}
\begin{prop}\label{prop_FIndecomposable}
The moduli space $\widetilde{\Mm}(F,J) \neq \emptyset$ for all $J \in  J_\om$. In particular, for every compatible almost complex structure $J\in\J_{\om_\lambda}$, and for any given point $p\in S^2\times S^2$, there is a unique embedded $J$-holomorphic sphere representing the class $F$ passing through~$p$.\qed
\end{prop}
We also have an analogous result for $(\CCC,\oml)$.
\begin{thm}
Given any point $p \in \CCC$, and any $J \in \jsom$, there is a $J$-holomorphic curve in the class $F:=L-E$ passing through $p$. \qed
\end{thm}
The following Theorem of M. Abreu and D. McDuff~\cite{MR1775741} tells us about the decomposition of the space of compatible almost complex structures on $(\SSS,\oml)$ into finitely many strata. 
\begin{thm} \label{strata}
Let $\jj_{\om_\lambda}$ denote the space of all compatible almost complex structures (not necessarily invariant) for the form $\om_\lambda$, on $\SSS$ then the space $\jj_{\om_\lambda}$ admits a finite decomposition into disjoint Fréchet manifolds of finite codimensions
\[
\jj_{\om_\lambda} = U_0 \sqcup U_2 \sqcup U_4 \ldots \sqcup U_{2n}
\]
where $2n= \lceil 2\lambda \rceil -1$ and $\lceil \lambda \rceil$ is the unique integer $l$ such that $l < \lambda \leq l+1$ and where
\begin{multline*}
U_{2i} := \big\{ J \in \jj_{\om_\lambda}~|~ D_{2i}=B-iF \in H_2(S^2 \times S^2,\Z)\text{~is represented by}\\ \text{a $J$-holomorphic sphere}\big\}
\end{multline*}
\end{thm}
A completely analogous description holds true for $\CCC$.
\begin{thm}\label{strata-CCC}
Let $\jj_{\om_\lambda}$ denote the space of all compatible almost complex structures (not necessarily invariant) for the form $\om_\lambda$, then the space $\jj_{\om_\lambda}$ admits a finite decomposition into disjoint Fréchet manifolds of finite codimensions
\[
\jj_{\om_\lambda} = U_1 \sqcup U_3 \sqcup U_5\ldots \sqcup U_{2n-1}
\]
where $2n  = \lceil 2\lambda \rceil -1$, and $\lceil \lambda \rceil$ is the unique integer $l$ such that $l < \lambda \leq l+1$ and where
\begin{multline*} 
U_{2i-1} := \big\{ J \in \jj_{\om_\lambda}~|~ D_{2i-1}=B-iF \in H_2(\CCC,\Z)\text{~is represented by}\\ \text{a $J$-holomorphic sphere}\big\}
\end{multline*}
\end{thm}
\begin{remark}
We label the strata in $\SSS$ and $\CCC$ by the homological self-intersection of the classes $D_s$.
\end{remark}
\begin{remark}\label{remark-starta-toric}
We note that for both $\SSS$ and $\CCC$,  there is a canonical integrable almost complex structure $J_s$ in the strata $U_s$ coming from realizing $\SSS$ and $\CCC$ as the $s^{\text{th}}$- Hirzebruch surface $W_s$ of section 1.2. Recall also that associated to each $J_s$ we have a unique $J_s$-holomorphic Hamiltonian toric action $\mathbb{T}_s$. Thus the set of possible equivalence classes of toric actions on $\SSS$ and $\CCC$ are in one-to-one correspondence with the strata in the decomposition of $\jj_{\om_\lambda}$. This fact will be crucial in our later analysis of centralizer subgroups. 
\end{remark}

\section{Intersection of \texorpdfstring{$\jsom$}{J\hat{}S1} with the strata}\label{section:intersection}
In this section we determine the intersection of the stratification of $\jj$ with the subspace of $S^1(a,b;m)$-invariant almost complex structures.\\

In what follows, we will use the following simple observation several times. Let $(M,\om)$ be a simply connected symplectic $4$-manifold. There is a left-exact sequence
\begin{equation}
1 \to \Symp_h(M,\om) \to \Symp(M,\om) \to \Aut_{c_1,\om}\left(H_2(M,\Z)\right)\label{Sequence:ActionOnHomology}
\end{equation}
where $\Symp_h(M,\om)$ is the subgroup of symplectomorphisms acting trivially on homology, and where $\Aut_{c_1,\oml}\left(H_2(M,\Z)\right)$ is the group of automorphisms of $H_2(M,\Z)$ that preserve the intersection product and the Poincaré duals of the cohomology classes $c_1(TM)$ and $[\oml]$. This later group is the identity for $(\CCC,\oml)$ with $\lambda > 1$ and for $(\SSS,\oml)$ with $\lambda > 1$. In the case of $(\SSS,\oml)$ with $\lambda = 1$, the group $\Aut_{c_1,\oml}\left(H_2(M,\Z)\right)$ is equal to $\Z_2$ and is generated by the symplectomorphism that swaps the two $S^2$ factors. Consequently, for any symplectically ruled rational surface, the above sequence is also right-exact and splits.

\begin{lemma}\label{lemma:ActionOnHomology} The following equalities hold among symplectomorphism groups:
\begin{itemize}
\item $\Symp(\SSS,\oml) = \Symp_h(\SSS,\oml)\ltimes\Z_2$ when $\lambda=1$,
\item $\Symp(\SSS,\oml) = \Symp_h(\SSS,\oml)$ when $\lambda>1$, and
\item $\Symp_h(\CCC,\oml) = \Symp(\CCC,\oml)$ for all $\lambda\geq 1$.
\end{itemize}\qed
\end{lemma}

\begin{prop}\label{prop:ToricExtensionCorrespondance}
Fix a circle action $S^1(a,b;m)$. Then the space of invariant almost complex structures $\jsom$ intersects the stratum $U_n$ iff the circle action $S^1(a,b; m)$ extends to the toric action $\T_n$. 
\end{prop}
\begin{proof}
If $\lambda = 1$, the result is immediate as the only toric action is $\T_0$. 

Let $\lambda > 1$ and suppose the action $S^1(a,b;m)$ extends to $\T_n$. This means that there exists a circle action $S^1(c,d; n) \subset \T_n$ which is equivariantly symplectomorphic to the action $S^1(a,b;m)$ via some  symplectomorphism $\phi$. Let $J_{n}$ denote the standard complex structure in the stratum $U_n$. By Lemma \ref{lemma:ActionOnHomology}, when $\lambda > 1$, the group $\Symp_h(\SSS,\oml)$ is equal to $\Symp(\SSS,\oml)$ so that $\phi$ preserves homology. Consequently, $\phi^*J_{n}$  belongs to the stratum $U_n$ and is invariant with respect to the $S^1(a,b;m)$ action, showing that the space $\jsom$ intersects the stratum $U_n$.

Conversely, suppose that $J \in \jsom \cap U_n$. If $n\geq 1$, then there exists an invariant, embedded, symplectic sphere $C_{n}$ of self-intersection $-n$. Arguing as in Proposition~\ref{prop:AtMostTwoExtensions} and Corollaries~\ref{cor:CircleExtensionsWith_a=1} and~\ref{cor:CircleExtensionsWith_a=-1}, we conclude that $S^1(a,b;m)$ must extend to the torus $\T_n$. If $n=0$, there there exist invariant spheres representing the homology classes $B$ and $F$ passing to a common fixed point. Again, this implies that $S^1(a,b;m)$ extends to the torus $\T_0$. Alternatively, we can adapt the proofs of Lemma~\ref{lemma:EvaluationFibrationConfigurations} in the present document and of Proposition~4.7 in~\cite{Liat} to show that if an invariant $J$ is in the stratum $U_n$ the circle action $S^1(a,b; m)$ extends to the toric action $\T_n$.
\end{proof}

The following corollaries immediately follow from Proposition~\ref{prop:AtMostTwoExtensions} and Corollaries~\ref{cor:CircleExtensionsWith_a=1} and~\ref{cor:CircleExtensionsWith_a=-1}.

\begin{cor}\label{cor:lambda=1_IntersectingOnlyOneStratum}
Suppose $\lambda=1$. Then the space $\J_{\om_1}$ of compatible, almost-complex structures on $(\SSS,\om_1)$ is made of only one stratum. In particular, any Hamiltonian circle action extends to the toric action $\T_0$ and the contractible space of invariant almost-complex structures consists of only one stratum $\jsom=\jsom\cap U_0$. \qed
\end{cor}

\begin{cor}\label{cor:IntersectingOnlyOneStratum}
Suppose $\lambda >1$. Consider an $S^1(a,b;m)$ action $(\SSS,\oml)$ or $(\CCC,\oml)$ depending on whether $m$ is even or odd. Under the following numerical conditions on $a,b,m,\lambda$, the space $\jsom$ intersects only the stratum $U_{m}$:
\begin{itemize}
    \item when $a\neq\pm 1$;
    \item when $b=0$ or  $b=am$;
    \item when $a = \pm 1$ and $2 \lambda \leq |2b-m|+\epsilon_{m}$.\qed
\end{itemize}
\end{cor}

\begin{cor} \label{cor:IntersectingTwoStrata}
Consider a circle action of type $S^1(\pm 1,b,m)$ on $(\SSS,\oml)$ or $(\CCC,\oml)$.
\begin{enumerate}
    \item If $a=1$, $b \neq\{0,m\}$, $\lambda >1$, and $2\lambda > |2b-m|+\epsilon_{m}$, then the space of $S^1(1,b;m)$-equivariant almost complex structures $\jsom$ intersects the two strata $U_m$ and $U_{|2b-m|}$.
    \item If $a=-1$, $b \neq\{0,-m\}$, $\lambda >1$, and $2\lambda > |2b+m|+\epsilon_{m}$, then the space of $S^1(-1,b;m)$-equivariant almost complex structures $\jsom$ intersects the two strata $U_m$ and $U_{|2b+m|}$.\qed
\end{enumerate} 
\end{cor}

\section{Symplectic actions of compact abelian groups on~\texorpdfstring{$\R^4$}{R4}}

In order to study the action of the the equivariant symplectomorphism group $\Symp_h^{S^1}(\SSS,\oml)$ on each invariant stratum $\jsom \cap U_k$, we will need to understand the equivariant topology of linearised symplectic  actions. The main result is Theorem~\ref{thm:EqGr} which is an equivariant version of Gromov's Theorem on the contractibility of the group of compactly supported symplectomorphisms of star-shaped domains of $\R^{4}$. This relies on Lemma~\ref{lemma:remove} and Proposition~\ref{prop:linearize} that were proven by W. Chen in the unpublished manuscript \cite{ChenUnpub}. For completeness, we shall reproduce their proof here.

Let $G$ be a compact group acting effectively and symplectically  on $ \C^2 = \R^4$ with the symplectic form $\om_0:= dx_1 \wedge dy_1 +  dx_2 \wedge dy_2$. We say it acts linearly if it acts as a subgroup of $\U(2) \subset \Sp(4)$.

\begin{lemma}[Chen]\label{lemma:remove}
Let $G$ be a compact group acting linearly on $(\R^4, \om_0)$. Suppose $V$ is a $G$ invariant, compact, and star-shaped neighbourhood of $\{0\}$. Let $f: \R^4 \setminus V \to \R^4$ be a $G$-equivariant symplectic embedding which is the identity near infinity. Then, for every $G$-invariant neighbourhood $W$ of $V$, there exists a $G$-equivariant symplectomorphism $g:\R^4 \to \R^4$ such that $g|_{\R^4\setminus W} = f$.  
\end{lemma}

\begin{proof}
As $0 \in int(V)$ and as $f$ is the identity near infinity, there exist $T>0$ such that $f(Tx) = Tx$ for all $x \in \R^4 \setminus~ V$. Define $f_t(x) = \frac{f(tx)}{t}$ for $1 \leq t \leq T$. As the $G$ action is linear, we observe that $f_t$ is equivariant for all $t \in [1,T]$. By construction, $f_1 = f$, $f_T = id$  and $f_t^*\om_0 = \om_0$ for all $t$. Thus there are compact sets $V_t = f_t(V)$ and open neighbourhoods $W_t= f_t(W)$ of $V_t$ such that the restriction  $f_t: \R^4 \setminus V \to \R^4 \setminus V_t$ and $f_t: \R^4 \setminus W \to \R^4 \setminus W_t$  are diffeomorphic. As $G$ acts linearly, each of the sets $W_t$ and $V_t$ are $G$-invariant. 

Define $X_t$ as the vector field that satisfies $\frac{d}{dt} f_t = X_t \circ f_t$ and consider the one form $\alpha_t = i_{X_t}\om_0$. As $f_t$ is $G$ equivariant and symplectic, both $X_t$ and $\alpha_t$ are $G$-invariant. Let $H_t: \R^4 \setminus V_t \to \R$ be a one parameter family of Hamiltonians that are $G$-invariant and that satisfy $\alpha_t= dH_t$. Since $f_t$ is the identity near infinity, this implies that $H_t$ is constant near infinity and we can take this constant to be 0. 

Finally, we can take a family of $G$-invariant bump functions $\rho_t: \R^4 \to [0,1]$ such that $\rho_t \equiv 0$ in a neighbourhood of $V_t$ and $\rho_t \equiv 1$ on $\R^4 \setminus W_t$. Then the Hamiltonian $\rho_t H_t: \R^4 \to \R$ is defined on the whole $\R^4$ and is also $G$-invariant. The Hamiltonian isotopy $g_t$ generated by $\rho_t H_t$ is $G$ equivariant for all $1\leq t\leq T$ and satisfies the properties $g_T = \id$ and $g_1|_{\R^4 \setminus W} = f$. Consequently, $g_1$ is the symplectomorphism $g$ we were looking for.
\end{proof}

Any unitary representation of a compact abelian group $G$ on $\C^2$ induces a splitting into eigenspaces $\C^2 = \C_1 \bigoplus \C_2$. This simple fact turns out to be essential in our treatment of the equivariant Gromov's Theorem. Consequently, from now on, we assume that the group $G$ is abelian.

\begin{prop}[Chen]\label{prop:linearize}
Let $V\subset\R^{4}$ be a compact star-shaped neighbourhood of $\{0\}$ and let $\om$ be a symplectic form on $\R^{4}$ such that $\om = \om_0$ outside some smaller open neighbourhood of the origin $U\subset V$. Let $G$ be a compact abelian group acting on $(V,\om)$ via symplectomorphisms that are linear near the boundary of V. Then the G action is conjugate to a linear symplectic action of $G$ on $(V,\om_0)$ by a diffeomorphism $\Phi$ which is the identity near the boundary and which satisfies $\Phi^*\om = \om_0$.
\end{prop}

\begin{proof}
Identify $\R^4$ with $\C^2$. The linear action near $\partial V$ extends to a unitary action on $\R^4$. As $G$ is abelian this linear action splits into two eigenspaces  $\C_{1} \oplus \C_{2}$. We can then compactify $\C^{2}$ along these eigenspaces to obtain $\CP^{1}\times \CP^{1}$ (diffeomorphic to $S^2 \times S^2$) equipped with two symplectic forms $\tilde\om$ and $\tilde\om_{0}$ induced from $\om$ and $\om_{0}$. The compactified space inherits two actions of $G$, namely, $\rho: G \hookrightarrow \Symp(\SSS,\tilde\om)$ coming from extending the $G$ action on $\R^{4}$ and $\rho_{\lin}:G \hookrightarrow \Symp(\SSS,\tilde\om)$ which extends the linear action near $\partial V$.  
\\

By construction, there exists a star-shaped subset $V_1 \subset V$ such that both the action $\rho$ and $\rho_{\lin}$ agree on $\C^2 \setminus V_1$. Note that the point at infinity $p:=(\infty,\infty)$ is a fixed point for both the action $\rho$ and $\rho_{lin}$. Choose a $\tilde\om$ compatible $G$-invariant almost complex structure $J$ on $\SSS$ which is standard in some neighborhood $X_\epsilon:= (S^2 \times D_\epsilon) \cup (S^2 \times D_\epsilon)$ of the wedge $(S^2\times\{\infty\})\vee (\{\infty\}\times S^2)$ (where $D_\epsilon$ is a small disk of radius $\epsilon$ at $\infty$). As $(S^2\times\{\infty\})$ and $(\{\infty\}\times S^2)$ are $J$-holomorphic spheres representing the classes $B$ and $F$ respectively, we conclude that $J$ belongs to the stratum $U_{0}$ so that there exists foliations $\mathcal{B}_J$ and $\mathcal{F}_J$ by embedded $J$-holomorphic spheres in the classes $B$ and $F$. Given any $q=(z,w) \in S^2 \times S^2$, let $u_w$ denote the unique curve in the class B  passing through $(0,w)$ and, similarly, let $v_z$ be the curve in class F  passing through $(z,0)$. We can define a self-map of $S^2 \times S^2$ by setting
\begin{align*}
\Psi_J: S^2 \times S^2 &\longrightarrow S^2 \times S^2  \\
(z,w) &\longmapsto u_w \cap v_z
\end{align*}
As explain in~\cite{McD} Chapter 9, this map $\Psi_J$ is a diffeomorphism. Moreover, $\Psi_J$ is equivariant with respect to the linear action $\rho_{\lin}$ on the domain  and the action $\rho$ on the codomain. Furthermore, as $J$ is the standard complex structure in a neighbourhood $X_\epsilon$ of the wedge, $\Psi_J$ is the identity near the base point $p$. We now modify $\Psi_J$ in order to make it the identity in the neighbourhood $X_\epsilon$. As $J$ is the standard complex structure on $X_\epsilon$ we can write
\[
\Psi_J (z,w)= 
\begin{cases} 
(z,\phi_2(z,w)) & \text{~if~} z \in    D_\epsilon \subset \SSS\\
(\phi_1(z,w),w) & \text{~if~} w \in   D_\epsilon \subset \SSS
\end{cases}
\]
where $\phi_1(z,0)=z$ and $\phi_2(0,w)=w$ for all $z,w \in S^2$. Choose a $G$ equivariant (for the $G$ action on $\{\infty\}\times S^2)$) smooth map $\beta_1:S^2 \longrightarrow S^2$ such that $\beta_1(z) = z$ for all $z \in D_\epsilon$ and $\beta_1 =\infty$ in a neighbourhood $D_\delta$ contained in $D_\epsilon$ and such that $\det(d\beta_{1}(z)) \geq 0 ~\forall ~z \in S^2$.
Similarly define a $G$ equivariant map (for the $G$ action on $S^2\times\{\infty\}$) $\beta_2: S^2 \to S^2$ satisfying analogous conditions as $\beta_1$. Then we define a modification of $\Psi_{J}$ by setting
\[
\Psi^{\prime}_J(z,w) = 
\begin{cases}
\Psi_J   &\text{~if~}  z \in  (S^2 \times S^2) \setminus X_\epsilon \\
(z,(\phi_2(\beta_1(z),w)) &\text{~if~} z \in D_\epsilon \\
(\phi_1(z,\beta_2(w)),w) &\text{~if~}  w \in D_\epsilon
\end{cases}
\]
This modification $\Psi^{\prime}_J$ is the identity in a smaller neighbourhood $X_\delta:=(S^2 \times D_\delta) \cup (S^2 \times D_\delta)\subset X_{\epsilon} $.\\

The submanifolds $\{z\} \times S^2$ and $S^2 \times \{w\}$ for all $z,w \in S^2$ are symplectic for the form ${\Psi'_{\!J}}^* \tilde\om $ and hence $\tilde\om_0 \wedge {\Psi'_{\!J}}^* \tilde\om>0$. Thus the path $\om_t:= t\tilde\om +(1-t){\Psi'_{\!J}}^* \tilde\om$ is a path of non-degenerate symplectic forms for all $t \in [0,1]$. We can then apply an equivariant Moser isotopy to get an equivariant diffeomorphism $\alpha$ of $S^2 \times S^2$ such that $\alpha^* {\Psi'_{\!J}}^* \tilde\om = \tilde\om_0$. Further, as ${\Psi'_{\!J}}^* \tilde\om= \tilde \om_0$ on $X_\delta$, the restriction of $\alpha$ to $X_\delta$ is the identity. We then set $\tilde \Psi_J:= \Psi'_{\!J} \circ \alpha$.
\\

The restriction of $\tilde\Psi_J: \C^2 \to \C^2$ gives us a map which is $G$-equivariant with respect to the action $\rho_{\lin}$ on the domain and $\rho$ on the codomain. As noted before there exists a star-shaped subset $V_1 \subset V$ such that both the action $\rho$ and $\rho_{\lin}$ agree on $\C^2 \setminus V_1$. We can choose $V_1$ such that $0 \in int(V_1)$. We now apply Theorem \ref{lemma:remove} to the map $f:={\tilde\Psi_J}^{-1}|_{\C^2 \setminus V_1}$ and we choose W in Theorem \ref{lemma:remove} to be a $G$-invariant open subset of $V$ which contains $V_1$. Let $g:\C^2 \to \C^2$ be an equivariant symplectomorphism as in Theorem \ref{lemma:remove} such that $g|_{\R^4\setminus W} = f$, then the map $\Phi:= \left(\tilde\Psi_J \circ g\right)\Big|_V: V \to V$ is identity near the boundary and satisfies $\Phi^*\om = \om_0$ and is
$G$-equivariant where the action of G on the domain of $\Phi$ is the linear action $\rho_{\lin}$ while on the range of $\Phi$ it is the $G$ action $\rho$ on V that we started out with. 
Thus $\Phi$ is the required equivariant symplectomorphism that linearizes the given $G$ action and takes the form $\om$ to $\om_0$.
\end{proof}

Consider a polydisk $D^2 \times D^2$ whose product structure is compatible with the eigenspace decomposition $\R^{4}=\C^{2}=\C_{1}\oplus \C_{2}$. Consider the symplectic form $\om$ such that $\om = \om_0$ outside of some smaller polydisk of the form  $D_r \times D_r \subset D^2 \times D^2$  for some radius $r$. 

\begin{thm}[Equivariant Gromov's Theorem for Polydisks]\label{thm:EquivariantGromovForPolydisks}
Let $G$ be an abelian group. Let $\om$ be a symplectic form on $D^2 \times D^2$ which is equal to $\om_0$ near the boundary. Let $G$ act symplectically on $(D^2 \times D^2 ,\om)$ and suppose the action is linear near the boundary. Then the group $\Symp_c^{G}(D^2\times D^2, \om)$ of equivariant symplectomorphisms that are equal to the identity near the boundary of $D^2 \times D^2$ is non-empty and contractible.
\end{thm}
\begin{proof}
As the $G$ action outside of $D_r \times D_r \subset \R^4$  is linear, we can extend this $G$ action to the whole of $\R^4$. Then we can compactify each eigenspace of $\C$ to an $S^2$ and hence this $G$ action extends to a symplectic action $S^2 \times S^2$ with respect to the form $\tilde\om$ induced by $\om$. 
\\

By Proposition~\ref{prop:linearize} we can conjugate our $G$ action by a symplectomorphism which is identity near the boundary to get a linear $G$ action on the whole of $V$. As any two conjugate topological subgroups are  homeomorphic, we shall just study the homotopy type of the compactly supported equivariant symplectomorphism group $\Symp_{c,\lin}^G(D^2 \times D^2,\om)$ for the linear $G$ action on $V$.
\\

Let $\J^G_{\om}$  be the non-empty and contractible space of all equivariant almost complex structures on $D^2 \times D^2$ that are compatible with $\om$ and are the standard split almost complex structure $J_0$ outside of $D_r \times D_r$. As they are the standard almost complex structures outside of a neighbourhood these almost complex structure extend to $\SSS$ and are compatible with $\tilde\om_0$. Further once we pick a base point $p=(\infty,\infty)\in S^2\times S^2$ and identify $D^2\times D^2$ with the complement of a standard neighborhood $X_\epsilon:= (S^2 \times D_\epsilon) \cup (S^2 \times D_\epsilon)$  of the wedge $(S^2\times\{\infty\})\vee (\{\infty\}\times S^2)\subset S^2\times S^2$ (note that the wedge point $(\infty,\infty)$ is a fixed point for the extended action of G on $\SSS$),  then any element $J\in  \J^G_{\sigma}$ extends to a equivariant almost complex structure of $\SSS$ which is the standard product complex structure on $\SSS$ on a neighbourhood $X_\epsilon$ of the wedge $(S^2\times\{y\})\vee (\{x\}\times S^2)\subset S^2\times S^2$. Conversely any such equivariant almost complex structure compatible with $\tilde\om$ that is  standard in some neighbourhood of the wedge $(S^2\times\{\infty\})\vee (\{\infty\}\times S^2)\subset S^2\times S^2$ gives us an element of $\J^G_{\om}$ .
\\

In order to show that $\Symp_{c,\lin}^G(D^2 \times D^2,\om)$ is contractible, we shall prove that it is homotopy equivalent to the contractible space $\J^G_{\om}$. 
\\

Define the map $\tilde \Psi_J$ as in the proof of Proposition~\ref{prop:linearize}. Thus we have a map 
\begin{align*}
\tau: \J^G_\om  &\longrightarrow \Symp_{c,\lin}^G(D^2 \times D^2,\om) \\ 
J &\longmapsto \tilde\Psi_J 
\end{align*}

To prove that $\tau$ is a
homotopy equivalence we construct a homotopy inverse as follows.  Fix a $J' \in \J^G_\om$, then the homotopy inverse $\beta$ is defined as,
\begin{align*}
\beta: \Symp_{c,\lin}^G(D^2 \times D^2,\om)  &\longrightarrow \J^G_\om \\ 
\phi &\lmapsto \phi_{*}J'
\end{align*}
By construction we see that $\tau(\beta(\phi)) = \id$ and the other direction is homotopic to the identity as $J^G_\om$ is contractible.
\end{proof}

We shall repeatedly use  the following theorem in our analysis of the homotopy type of the equivariant symplectomorphism groups of $\SSS$.

\begin{thm}(Equivariant Gromov's Theorem)\label{thm:EqGr}
Let $(V,\om)$ be an  compact star-shaped symplectic domain of $\R^4$ such that $0 \in int(V)$ and let $\om$ be such that $\om = \om_0$ near the boundary of $V$. Let $G$ be a compact abelian group that acts symplectically and linearly near the boundary and that sends the boundary to itself, then the space of equivariant symplectomorphisms that act as identity near the boundary (denoted by $\Symp_c^G(V,\om)$) is non-empty and contractible.
\end{thm}

\begin{proof}
By Proposition~\ref{prop:linearize} we can conjugate our $G$ action by a symplectomorphism which is identity near the boundary to get a linear $G$ action on the whole of $V$ and such that it takes the form $\om$ to $\om_0$. As the homotopy type of the two conjugate equivariant symplectomorphism group is the same (they are in fact homeomorphic), we shall just study the homotopy type of the compactly supported equivariant symplectomorphism group  for the linear $G$ action on $(V,\om_0)$. We denote this group by $\Symp_{c,\lin}^G(V,\om_0)$.
\\

Choose real numbers $r>0$ and $T>1$ , $D_r \times D_r$ is a polydisk of radius $r$, such that $\frac{1}{T}V \subset D_r \times D_r \subset int(V)$,
and consider the family of maps $F_t : \Symp_{c,\lin}^G(V,\om_0) \to \Symp_{c,\lin}^G(V,\om_0)$ for $1 \leq t \leq T$ 
defined by $F_t(\phi)(x) = \frac{\phi(tx)}{t} $ for all $x \in V$.
\\

Then we have that $F_t$ is $G$ equivariant for all $1 \leq t \leq T$, $F_1(\phi) = \phi$ for all $\phi \in \Symp_{c,\lin}^G(V,\om_0)$,  $F_t(id) = id$ for all t, and $F_T \left(\Symp_{c,\lin}^G(V,\om_0)\right) \subset \Symp_{c,\lin}^G(D_r \times D_r,\om)$.
\\

The proof of Theorem \ref{thm:EquivariantGromovForPolydisks} tells us that the inclusion $i:\Symp_{c,\lin}^G(D_r \times D_r,\om) \hookrightarrow \Symp_{c,\lin}^G(V,\om_0)$ is contractible. Hence we can fix a contraction $\alpha_t$ for $T \leq t \leq T+1$ such that $\alpha_T = i$ and $\alpha_{T+1}(\phi) = id$ for all $\phi \in \Symp_{c,\lin}^G(D_r \times D_r,\om)$. Then the concatenation 
$$\Tilde{F_t}:=\begin{cases}
F_t ~~1 \leq t\leq T\\
F_T \circ \alpha_t ~~ T\leq t \leq T+1
\end{cases}$$ 
gives  us a retraction of  $\Symp_{c,\lin}^G(V,\om_0)$  to $id$ and hence $\Symp_{c}^G(V,\om)$ is contractible.
\end{proof}

\begin{remark}
To our knowledge, it is not known whether an equivariant version of Gromov's Theorem holds true for actions of arbitrary compact groups. If true, the proof of such statement would probably require new ideas that do not rely on the compactification of a polydisc to $\SSS$.
\end{remark}

\section{Homotopical description of the strata \texorpdfstring{$\jsom \cap U_k$}{J\hat{}S1\_k}} \label{Section:homotopical descriptionofjsom}
We now consider the action of the group of equivariant symplectomorphisms $\Symp_{h}^{S^1}(\SSS,\oml)$ on the contractible space $\jsom$ of invariant, compatible, almost-complex structures, and we investigate the orbit-type stratification of this action up to homotopy (See Lemma~\ref{contractibilityofinvariantacs} for proof of contractibility of $\jsom$). The similar analysis for circle actions on the nontrivial bundle $\CCC$ will be addressed in Section~\ref{Chapter-CCC}.

\subsection{Notation}\label{not}

We shall use the following notation in the rest of the document. Let $M$ denote the manifold $\SSS$. 
Consider a Hamiltonian $S^1$ action on $(M,\oml)$ and let $p_0$ be a fixed point for the group action. Given a $S^1$ invariant symplectic section $C$, and a $\omega_\lambda$-orthogonal invariant sphere $\overline{F}$ in the homology class $F$ that intersects $C$ at a point $p_0$, we define the following spaces: 

\begin{itemize}
\item $N(C)$:= The symplectic normal bundle to a symplectic submanifold $C$. 
\item $\Symp^{S^1}_h(M,\oml)$ := The group of $S^1$ equivariant symplectomorphisms on $(M,\oml)$ that acts trivially on homology.
\item $\Stab^{S^1}(C)$ := The group of all  $\phi \in \Symp^{S^1}_h(M,\oml)$ such that $\phi(C) = C$, that is, such that $\phi$ \emph{stabilises} $C$ but does not necessarily act as the identity on $C$.
\item $\Fix^{S^1}(C)$ := The group of all  $\phi \in \Symp^{S^1}_h(M,\oml)$ such that $\phi|_C = id$, that is, such that \emph{fixes $C$ pointwise}. 
\item $\Fix^{S^1}(N(C))$:= The group of all $\phi \in \Symp^{S^1}_h(M,\oml)$ such that $\phi|_C = \id$ and  $d\phi|_{N(C)}: N(C) \to N(C)$ is the identity on $N(C)$.
\item $\Aut^{S^1}(N(C))$:= The space of $S^1$-equivariant fiberwise symplectic automorphisms of the symplectic normal bundle to the sphere $C$.
\item $\Aut^{S^1}(N(C \vee \overline{F}))$ := The space of $S^1$-equivariant symplectic automorphisms of the the symplectic normal bundle to the wedge of spheres $C \vee \overline{F}$ which are equal to the identity in a neighbourhood of the wedge point. 
\item $\mathcal{S}^{S^1}_{K}$ := The space of unparametrized $S^1$-invariant symplectic embedded spheres in the homology class $K$. 
\item $\mathcal{S}^{S^1}_{K,p_0}$:= The space of unparametrized $S^1$-invariant symplectic embedded spheres in the homology class $K$ passing through $p_0$.
\item $\J_{\oml}^{S^1}(C)$ := The space of $S^1$-equivariant $\oml$ compatible almost complex structures s.t the curve $C$ is holomorphic.
\item $\Symp^{S^1}(C)$:= The space of all $S^1$-equivariant symplectomorphisms of the curve C.
\item $\Fix^{S^1}(N(C \vee {\overline{F}}))$ := The space of all $S^1$-equivariant symplectomorphisms that are the identity in the neighbourhood of $C \vee {\overline{F}}$.
\item $\Symp^{S^1}({\overline{F}}, N(p_0))$ := equivariant symplectomorphism of the sphere $\overline{F}$ that are the identity in an open set of $\overline{F}$ around $p_0$. 
\item $\overline{\mathcal{S}^{S^1}_{F,p_0}}$:= The space of unparametrized $S^1$-invariant symplectic spheres in the homology class $F$ that are equal to a fixed curve ${\overline{F}}$ in a neighbourhood of $p_0$.
\item $\Symp^{S^1}_{p_0,h}(M,\oml)$:=  The group  of all  $\phi \in \Symp^{S^1}_{h}(M,\oml)$ fixing $p_0$.
\item $\Stab^{S^1}_{p_0}(C)$:=  The group of all  $\phi \in \Stab^{S^1}(C)$ such that $\phi(p_0) = p_0$.  
\end{itemize}

All the above spaces are equipped with the $C^\infty$ topology.

\subsection{Case 1: \texorpdfstring{$\Symp^{S^1}_h(\SSS,\oml)$}{Symp(S2xS2)} action on \texorpdfstring{$\jsom \cap U_{2s}$}{J\hat{}S1\_2k} with \texorpdfstring{$s \neq 0$}{s>0}}\label{section:ActionOnU_2k}

Let $\lambda > 1$ and consider a $S^1$-action $S^1(a,b;m)$ on $(\SSS,\oml)$ with $m=2k$. Let  ${\mathcal{S}^{S^1}_{D_{2s}}}$ denote the space of all $S^1$ invariant symplectic embedded spheres in the class $D_{2s}=B-sF$. We shall assume that ${\mathcal{S}^{S^1}_{D_{2s}}}$ is non-empty which, by Theorems~\ref{cor:IntersectingOnlyOneStratum} and~\ref{cor:IntersectingTwoStrata}, means that $2s=m$ or $2s=|2b\pm m|$ depending on $a$.\\

The homotopy type of $\Symp^{S^1}_h(\SSS,\oml)$ is related to the strata $\jsom \cap U_{2s}$ through the following sequence of fibrations and weak homotopy equivalences:

\begin{equation}\label{MainFibrations}
\begin{gathered}
\Stab^{S^1}(\overline{D}) \longrightarrow \Symp^{S^1}_{h}(\SSS,\oml) \longtwoheadrightarrow {\mathcal{S}^{S^1}_{D_{2s}}} \mathbin{\textcolor{blue}{\xrightarrow{\text{~~~$\simeq$~~~}}}} \jsom \cap U_{2s}\rule{0em}{2em}\\
\Fix^{S^1}(\overline{D}) \longrightarrow \Stab^{S^1}(\overline{D}) \longtwoheadrightarrow  \Symp^{S^1}(\overline{D}) \mathbin{\textcolor{blue}{\xrightarrow{\text{~~~$\simeq$~~~}}}} S^1 ~\text{or}~ \SO(3)\rule{0em}{2em}\\
\Fix^{S^1} (N(\overline{D})) \longrightarrow \Fix^{S^1}(\overline{D}) \longtwoheadrightarrow  \Aut^{S^1}(N(\overline{D})) \mathbin{\textcolor{blue}{\xrightarrow{\text{~~~$\simeq$~~~}}}} S^1\rule{0em}{2em}\\
\Stab^{S^1}(\overline{F}) \cap \Fix^{S^1}(N(\overline{D})) \longrightarrow \Fix^{S^1}(N(\overline{D})) \longtwoheadrightarrow  \overline{\mathcal{S}^{S^1}_{F,p_0}} \mathbin{\textcolor{blue}{\xrightarrow{\text{~~~$\simeq$~~~}}}} \mathcal{J}^{S^1}(\overline{D})\simeq \{*\}\rule{0em}{2em}\\
\Fix^{S^1}(\overline{F})  \longrightarrow \Stab^{S^1}(\overline{F}) \cap \Fix^{S^1}(N(\overline{D})) \longtwoheadrightarrow  \Symp^{S^1}(\overline{F}, N(p_0))  \mathbin{\textcolor{blue}{\xrightarrow{\text{~~~$\simeq$~~~}}}} \left\{*\right\}\rule{0em}{2em}\\
\left\{*\right\} \mathbin{\textcolor{blue}{\xleftarrow{\text{~~~$\simeq$~~~}}}} \Fix^{S^1}(N(\overline{D} \vee \overline{F})) \longrightarrow \Fix^{S^1}(\overline{F}) \longtwoheadrightarrow  \Aut^{S^1}(N(\overline{D} \vee \overline{F})) \mathbin{\textcolor{blue}{\xrightarrow{\text{~~~$\simeq$~~~}}}} \left\{*\right\}\rule{0em}{2em}
\end{gathered}
\end{equation}
Here $\overline{F}$ and $\overline{D}$ are the $\omega_\lambda$-orthogonally intersecting invariant curves in the $2s^{\,\text{th}}$-Hirzebruch surface $W_{2s}$. We denote by $\overline{F}$ the unique curve in class $F$ intersecting the given curve $\overline{D} \in {\mathcal{S}^{S^1}_{D_{2s}}}$ $\omega_\lambda$-orthogonally at $p_0$. In the second fibration, the group $\Symp^{S^1}(\overline{D})$ is homotopy equivalent to $\SO(3)$ when the $S^1$ action fixes the curve $\overline{D}$ pointwise. Otherwise, it is homotopy equivalent to~$S^1$.

Assuming  the homotopy equivalence in the first fibration, we immediately get
\begin{thm}
Consider the $S^1(a,b;m)$ action on $(\SSS,\oml)$ with $\lambda >1$. If $ \jsom \cap U_{2s}$ is non-empty, then $\Symp^{S^1}_h(\SSS,\oml)/\Stab^{S^1}(\overline{D})  \simeq \jsom \cap U_{2s}$.
\end{thm}
Furthermore, tracking down the various homotopy equivalences in the other fibrations, we will prove that the equivariant stabilizer of the curve $\overline{D}$, namely  $\Stab^{S^1}(\overline{D})$, is homotopy equivalent to the equivariant stabilizer of the corresponding complex structure under the natural action of $\Symp^{S^1}_h(\SSS,\oml)$. More precisely, 
\begin{itemize}
    \item $\Stab^{S^1}(\overline{D}) \simeq \T_{2s}$ when $(a,b) \neq (0,\pm1)$;
    \item  $\Stab^{S^1}(\overline{D}) \simeq SO(3) \times S^1$ when $(a,b) = (0,\pm1)$.
\end{itemize}
In the following two sections, we explain why the above maps~\eqref{MainFibrations} are fibrations, and we establish the claimed homotopy equivalences. 

\subsubsection{The first fibration $\Symp^{S^1}_{h}(\SSS,\oml) \to {\mathcal{S}^{S^1}_{D_{2s}}}$}

We first make the following observation on the action of $\Symp^{S^1}_h(\SSS,\oml)$ on the fixed point set of the $S^1$ action.

\begin{lemma}[Lemma 2.5 in \cite{Liat}]{\label{LR}}
Let $(M,J)$ be an almost complex manifold such that $J$ is invariant under the action of a compact Lie group $G$ on $M$. Let $S$ be an embedded sphere in $M$. Assume that $S$ is a connected component of the fixed point set of a non-trivial subgroup $H\subseteq G$. Then $S$ is a $J$-holomorphic sphere.
\end{lemma}
\begin{proof}
The surface $S$ is $J$-holomorphic if for all $x\in S$, $T_xS$ is $J$-invariant. As $S$ is a connected component of the fixed point set of a non-trivial subgroup $H$, all tangent vectors $v \in T_xS$ are characterised by the property that $dh \cdot v = v$ for all $h \in H$. Thus in order to show that that  $Jv \in T_xS$, it suffices to prove that $dh \cdot Jv = Jv$ for all $h \in H$. But this immediately follows from the invariance of $J$ since
\[dh \cdot Jv = J(dh \cdot v) = Jv.\]
\end{proof}
\begin{cor}
All the invariant spheres that are visible in the graph of a Hamiltonian circle action (which, by definition, have non-trivial stabilizers isomorphic to either $S^1$ or $\Z_k$ for some $k\geq 2$) are $J$-holomorphic for all invariant $J\in\jsom$.
\end{cor}

\begin{cor}\label{Cor:SinglecurveInclassD2s}
Consider a circle action $S^1(a,b;m)$ with $a\neq \pm1$ that leaves invariant a curve of negative self-intersection $-2s$. Then $2s=m$, and the space $\mathcal{S}^{S^1}_{D_{2s}}$ of all $S^1$-invariant symplectic embedded spheres representing the class $D_{2s}=B-sF$ contains a single curve $\overline{D}_{2s}$. Consequently when $a\neq\pm 1$, the action map \[\Symp^{S^1}_h(\SSS,\oml) \twoheadrightarrow {\mathcal{S}^{S^1}_{D_{2s}}}
\]
is a fibration in a trivial way, and we have weak equivalences \[\{\overline{D}_{2s}\}=\mathcal{S}^{S^1}_{D_{2s}} \simeq \jsom \cap U_{2s}=\jsom \simeq\{*\}.\]
\end{cor}
\begin{proof}
Let $\overline{D}_m$ denote the unique $S^1$ invariant curve in class $D_{m}$ through the fixed points $Q$ and $R$ depicted in Figure~\ref{hirz}. Looking at Table~\ref{table_weights}, we see that the only circle actions for which this curve is a free invariant sphere are the ones for which $a = \pm1$. Thus for all other values $a\neq\pm1$, the curve $\overline{D}_m$ is an invariant sphere whose isotropy group is non-trivial. By Lemma~\ref{LR},  $\overline{D}_m$ is holomorphic for all $J\in\jsom$ (which shows, again, that $\jsom=\jsom\cap U_m$ when $a\neq\pm1$). Conversely, by positivity of intersection (Theorem~\ref{thm_PositivityIntersections}), given any $J\in\jsom$, there is exactly one curve in class $D_m$ which is $J$-holomorphic. Since any invariant curve is $J$-holomorphic for some $J\in\jsom$, this implies that ${\mathcal{S}^{S^1}_{D_{2s}}}=\{\overline{D}_m\}$.
\end{proof}

The rest of this subsection is devoted to proving that the action map
\[\Symp^{S^1}_h(\SSS,\oml) \twoheadrightarrow {\mathcal{S}^{S^1}_{D_{2s}}}\] 
is a fibration for $S^1(a,b;m)$ actions with $a = \pm1$. 

\begin{lemma}\label{Lemma:SympActionOnFixedPointGenericStrata}
Consider a $S^1(\pm1,b;m)$ action that leaves invariant a symplectic embedded sphere $C$ in a class $D_{2s}=B-sF$, $s\neq0$. Then there exists a fixed point $p_0\in C$ such that
\begin{enumerate}
\item all invariant curves in class $B-sF$ pass through $p_0$,
\item $\phi(p_0)=p_0$ for all $\phi\in\Symp^{S^1}_h(\SSS,\oml)$,
\item the weights $(w_1,w_2)$ of the circle action at $p_0$ are distinct, that is, $w_1\neq w_2$.
\end{enumerate}
\end{lemma}

\begin{proof} 
 
Without loss of generality, we can assume $a=1$.\\

If $b=0$ or $b=m$, Corollary~\ref{cor:IntersectingOnlyOneStratum} shows that the circle action only extends to $\T_m$ so that $m=2s\neq0$. Consequently, the curve $C$ represents the class $D_m$ and connects the fixed points $Q$ and $R$ of Figure~\ref{hirz}. Conversely, looking at Table~\ref{table_weights} and using Lemma~\ref{weight}, we see that any other invariant curve in class $D_m$ passes through these two fixed points. If $b=0$, then the weights at $R$ are $(-1,m)$ and are distinct from all the weights at the other fixed points. By Corollary~\ref{cor:ActionPreservesWeights}, $R$ is invariant under the action of $\Symp^{S^1}_h(\SSS,\oml)$. We thus set $p_0=R$. Similarly, if $b=m$, we set $p_0=Q$.\\

Now suppose that $a=1$ and $b\not\in\{0,m\}$. Then all the fixed points are isolated. By Corollary~\ref{cor:IntersectingTwoStrata}, the $S^1(1,b;m)$ action may extend to the tori $\T_m$ and $\T_{|2b-m|}$. The class $D_{2s}$ may be either $B-\frac{m}{2}$ with $m\neq0$, or $B-\frac{|2b-m|}{2}F$. In both cases, the invariant sphere with negative self intersection $-2s$ passes through the fixed points $Q$ and $R$ of the corresponding moment polytope $\Delta_{2s}$ (see Figure~\ref{hirz}). As before, any other invariant curve in class $D_{2s}$ passes through these two fixed points. Consequently, the action of $\Symp^{S^1}_h(\SSS,\oml)$ can only swap $Q$ and $R$. Since the weights at $Q$ and $R$ are $(1,-b)$ and $(-1,m-b)$, this is possible only if $b=1$ and $m=2$. In all other cases, one of the points $Q$ or $R$ satisfies the conditions of the theorem. In the special case $a=1$, $b=1$, and $m=2$, the points $Q$ and $R$ are distinguished by their moment map values. As the action of $\Symp^{S^1}_h(\SSS,\oml)$ preserves the moment map, we conclude that $Q$ or $R$ satisfies the conditions of the theorem.
\end{proof}

\begin{remark}
A similar statement holds for circle actions on the non-trivial bundle $\CCC$.
\end{remark}

Let $p$ be the fixed point satisfying the three properties stated in Lemma~\ref{Lemma:SympActionOnFixedPointGenericStrata}. Because the two weights at $p$ are distinct, if an invariant curve in the class of the fiber $F$ intersects an invariant curve in class $D_{2s}$ at $p$, then the two curves intersect $\oml$-orthogonally. Let $\mathcal{C}(D_{2s}\vee F,p)^{S^1}$ be the space of invariant configurations made of curves in classes $D_{2s}$ and $F$ intersecting orthogonally at $p$. We now have all the ingredients to show that ${\mathcal{S}^{S^1}_{D_{2s}}}$ is a homogeneous space under the natural action of $\Symp^{S^1}_h(\SSS,\oml)$.
\begin{lemma}\label{lemma:EvaluationFibrationConfigurations} 
Fix an action of the form $S^1(\pm1,b;m)$. Choose a point $p$ satisfying the condition of Lemma~\ref{Lemma:SympActionOnFixedPointGenericStrata}. Then, the evaluation map at a standard configuration $\overline{D}_{2s}\vee \overline{F}$ through $p$
\[\Symp^{S^1}_h(\SSS,\oml) \twoheadrightarrow \mathcal{C}(D_{2s}\vee F,p)^{S^1}\]
is a Serre fibration.
\end{lemma}
\begin{proof}
We first show that the action is transitive. Given an invariant configuration $C\vee A \in \mathcal{C}(D_{2s}\vee F,p)^{S^1}$, the equivariant symplectic neighbourhood theorem implies that we can find an invariant neighbourhood V of $C \vee A$, an invariant neighbourhood $V'$ of $\overline{D}_{2s} \cup \overline{F}$, and an equivariant symplectomorphism $\alpha:V \to V'$. We claim that $\alpha$ can be extended to an ambient diffeomorphism $\beta$ of $\SSS$. Assume this for the moment. By construction, the pullback form $\om_\beta:=\beta^*\om_\lambda$ is invariant under the the conjugate action $\beta^{-1}\rho\beta$.

We observe that the complement of the standard configuration $\overline{D}_{2v} \vee\overline{F}$ in $W_{2s}$ is symplectomorphic to $\C^2$ with the symplectic form $\om_f:=\frac{1}{2\pi}\del\delbar f$ where $f=\log\left(\left(1+||w||^2\right)^{\lambda}\left(1+||w||^{4k}+||z||^2\right)\right)$, see~\cite[Lemma~3.5]{abreu}. Under this identification, the standard $\T_{2s}$ action on $W_{2s}$ become linear on $\C^2$. It follows that near infinity, the form $\om_\beta$ is equal to $\om_f$ and that the action $\beta^{-1}\rho\beta$ is linear. By Proposition~\ref{prop:linearize} we get that there is an equivariant symplectomorphism $\gamma$ that is equal to the identity near infinity, and that identifies $(\C^2,\om_f, \rho)$ with $(\C^2,\beta^*\om_\lambda,\beta^{-1}\rho\beta)$. By construction, the equivariant symplectomorphism $\phi := \gamma \circ \beta$ takes the configuration $C\vee A$ to $\overline{D}_{2s}\vee\overline{F}$. 

It remains to see that the local diffeomorphism $\alpha$ can be extended to an ambient diffeomorphism $\beta$ of $\SSS$. By the Isotopy Extension Theorem (see~\cite[Theorem~1.4, p.180]{Hi}), it suffices to show that any two configurations of embedded spheres in classes $\overline{D}$ and $F$ intersecting transversely and positively are isotopic. In turns, this follows from the fact that any two $F$-foliations corresponding to two almost-complex structures $J$ and $J'$ are diffeotopic, and that any two sections of the product $\SSS$ are diffeotopic through sections iff they are homotopic. This shows that $\Symp^{S^1}_h(\SSS,\oml)$ acts transitively on $\mathcal{C}(D_{2s}\vee F,p)^{S^1}$.

To prove the homotopy lifting property, consider any family of maps $\gamma: D^n \times [0,1] \rightarrow \mathcal{C}(D_{2s}\vee F,p)^{S^1}$ from a $n$ dimensional disk $D^n$ to $\mathcal{C}(D_{2s}\vee F,p)^{S^1}$, and choose a lift $\overline{\gamma_0}:D^n \rightarrow \Symp^{S^1}_h(\SSS,\oml)$ of $\gamma_0$. Since the complement of a configuration is contractible, the equivariant version of Banyaga's Extension Theorem for families implies that there exists a lift $\overline{\gamma}: D^n \times [0,1] \rightarrow \Symp^{S^1}_h(\SSS,\oml) $ extending $\overline{\gamma_0}$. 
\[
\begin{tikzcd}
D^n \times \{0\} \arrow[d,hookrightarrow] \arrow[r,"\overline{\gamma_0}"]    &\Symp^{S^1}_h(\SSS,\oml) \arrow[d,"\theta"] \\
    D^n \times [0,1]\arrow[r,"\gamma"] \arrow[ur,dashrightarrow,"\exists ~ \overline{\gamma}"] &\mathcal{C}(D_{2s}\vee F,p)^{S^1}
\end{tikzcd}
\]
Alternatively, one can apply the equivariant Gromov-Auroux Lemma~\ref{Au} to show the existence of the lift  $\overline{\gamma}$. In both cases, this concludes the proof.
\end{proof}

\begin{cor}\label{trans} 
Fix the action $S^1(\pm 1,b;m)$ for which $\mathcal{S}^{S^1}_{D_{2s}}$ is nonempty. Then $\Symp^{S^1}_h(\SSS,\oml)$ acts transitively on ${\mathcal{S}^{S^1}_{D_{2s}}}$.\qed
\end{cor}

\begin{lemma}\label{first}
Fix an action of the form $S^1(\pm1,b;m)$. Let $\overline{D}$ be an invariant symplectic sphere in the class $B - sF$ for which $\mathcal{S}^{S^1}_{D_{2s}}$ is nonempty. Then the evaluation map
\begin{align*}
    \theta: \Symp^{S^1}_h(\SSS,\oml) &\twoheadrightarrow {\mathcal{S}^{S^1}_{D_{2s}}} \\
    \phi &\mapsto \phi(\overline{D})
\end{align*}
is a Serre fibration with fibre over $\overline{D}$ given by
\[\Stab(\overline{D}):= \left\{ \phi \in \Symp^{S^1}(\SSS,\oml)~|~ \phi(\overline{D}) = \overline{D}\right\}\]
\end{lemma}
\begin{proof}
The evaluation map is transitive by Corollary~\ref{trans}. The homotopy lifting property follows from Lemma~\ref{Au} as in the proof of Lemma~\ref{lemma:EvaluationFibrationConfigurations}. Alternatively, one can also note that the action map factors through the restriction map 
\[\mathcal{C}(D_{2s}\vee F,p)^{S^1}\to\mathcal{S}^{S^1}_{D_{2s}}\]
which is itself a fibration. To see this, note that the restriction maps fits into a commuting diagram
\[
\begin{tikzcd}
\jsom\arrow[r,"f_1"]\arrow[rd,"f_2"] & \mathcal{C}(D_{2s}\vee F,p)^{S^1} \arrow[d]\\
& \mathcal{S}^{S^1}_{D_{2s}}
\end{tikzcd}
\]
where the maps $f_1$ and $f_2$ are fibrations. Observe that the map $f_1$ is well defined because the weights at the chosen fixed point $p$ are not equal. Hence, for any choice of an invariant almost complex structure $J \in \jsom$, the unique invariant $J$-holomorphic curve $C$ in class $D_{2s}$ intersects the unique invariant $J$-holomorphic fiber through $\oml$-orthogonally at~$p$. 
\end{proof}

\begin{remark}
As both $\Symp^{S^1}_h(\SSS,\oml)$ and $\mathcal{S}^{S^1}_{D_{2s}}$ can be shown to be homotopically equivalent to CW-complexes (see \cite{McD} Remark~9.5.5), we see from Theorem~1 in~\cite{Serrefib} (with proof corrected in \cite{error}), that a Serre fibration in which the total space and base space are both CW complexes is necessarily a Hurewicz fibration. Thus the map $\theta: \Symp^{S^1}_h(\SSS,\oml) \twoheadrightarrow {\mathcal{S}^{S^1}_{D_{2s}}}$ is in fact a Hurewicz fibration and hence the fibre over any arbitrary $D \in \mathcal{S}^{S^1}_{D_{2s}}$ is homotopy equivalent to $\Stab(\overline{D})$. 
\end{remark}

\begin{lemma}\label{first2}
Fix an action of the form $S^1(\pm1,b;m)$. The natural map $\alpha: \jsom \cap U_{2s} \to {\mathcal{S}^{S^1}_{D_{2s}}}$ defined by sending an almost complex structure $J \in \jsom \cap U_{2s}$ to the unique $J$-holomorphic curve in class $D_{2s}$ is a weak homotopy equivalence. 
\end{lemma}
\begin{proof}
To show that $\alpha$ is a weak homotopy equivalence, we first show that the map is Serre fibration.  To do so, consider an arbitrary element ${D} \in {\mathcal{S}^{S^1}_{D_{2s}}}$. As in the proof of Lemma \ref{first}, it suffices to show that given a family of map $\gamma_t$ from a n-dimensional disk $D^n$, such that $\gamma_0(0) = {D}$, and a lift $\gamma_0^\prime:D^n \rightarrow \jsom \cap U_{2s}$ lifting $\gamma_0$, then there exists a lift $\gamma_t^\prime$ to $ \jsom \cap U_{2s}$.
\[
\begin{tikzcd}
    D^n \times \{0\} \arrow[d,hookrightarrow] \arrow[r,"\gamma_0^\prime"] &\jsom \cap U_{2s} \arrow[d,"\alpha"] \\
    D^n \times [0,1] \arrow[r,"\gamma"] \arrow[ur,dashrightarrow,"\exists ~ \gamma^\prime"] &\mathcal{S}^{S^1}_{D_{2s}}
\end{tikzcd}
\]
As in the proof of Lemma \ref{first}, we have that there exists a lift $\overline{\gamma}: D^n \times [0,1] \to \Symp^{S^1}_h(\SSS,\oml)$ of $\gamma$. Pick an element $ J \in \alpha^{-1}({D})$ and define $\gamma^\prime(s) := \overline{\gamma}^*J$. This defines a lift $\gamma^\prime$ of $\gamma$. Hence $\alpha$ is a fibration. The fact that the fibres of $\alpha$ are contractible follows from Lemma~\ref{Lemma:Contractibilityofacs}. Thus we get the required result that $\jsom \cap U_{2s} \simeq {\mathcal{S}^{S^1}_{D_{2s}}}$.
\end{proof}

\subsubsection{The other fibrations }

The arguments showing that the other maps in~\eqref{MainFibrations} are fibrations are essentially local and apply uniformly to all Hamiltonian circle actions $S^1(a,b;m)$.

\begin{lemma}{\label{stab}}
Consider an action of the form $S^1(a,b;m)$. The restriction map $\Stab^{S^1}(\overline{D}) \twoheadrightarrow \Symp^{S^1}(\overline{D})$ is a fibration.  
\end{lemma}

\begin{proof}
To show that the restriction map is a fibration we use Theorem \ref{palais}, in which we set $X= \Symp^{S^1}(\overline{D})$, $G=\Stab^{S^1}(\overline{D})$ and the action  is given by 
\begin{align*}
    G \times X &\to X \\
    (\phi, \psi) &\to \phi|_{\overline{D}} \circ \psi
\end{align*}
Hence in order to show that the restriction map $r :\Stab^{S^1}(\overline{D}) \to \Symp^{S^1}(\overline{D})$ is a fibration, we only need to show that the action described above  admits local cross sections. Suppose we only  show that a neighbourhood of identity admits local cross sections and that  $\Stab^{S^1}(\overline{D})$ acts transitively on $\Symp^{S^1}(\overline{D})$ this would suffice to show that $r$ is a fibration as by Theorem \ref{palais}, its a local fibration near the identity and the map $r$ is equivariant with respect to the action of $\Stab^{S^1}(\overline{D})$, thus completing the proof.
\\

Consider the identity $\id \in \Symp^{S^1}(\overline{D})$. Let $\alpha:N(\overline{D}) \to U$ be an equivariant diffeomorphism between the symplectic normal bundle $N(\overline{D})$ and a neighbourhood $U$ of $\overline{D}$. As $\Symp^{S^1}(\overline{D})$ is locally contractible (this can be seen for example by noticing that the proof of Prop 3.3.14 in \cite{MS} can be made equivariant) we can find a neighbourhood V of $\id$, and a fixed retraction $\beta_t$ of the neighbourhood $V$ onto the identity. Hence given any $\psi \in \Symp^{S^1}(\overline{D})$, we get a one parameter family $\beta_t(\psi)$ of symplectomorphisms.  As $\pi_1(\overline{D}) = 0$, $\beta_t(\psi)$ is Hamiltonian and is generated by  a function $H_t$. Let $\pi:N(\overline{D}) \to \overline{D}$ be the projection of the normal bundle. Define $\Tilde{H_t}:= \alpha \circ \pi^*H_t$. Thus $\Tilde{H_t}$  defines an invariant function on a U. Fix an invariant bump function $\rho$ with support in $U$ and is 1 in a small neighbourhood around $\overline{D}$, then $\rho \Tilde{H_t}$ is an invariant function and the corresponding symplectomorphism it generates $\Tilde{\psi}$ belong to $\Stab^{S^1}(\overline{D})$ and extends $\psi$.  Note that if we fix the neighbourhood $U$, the bump function and the retraction of the neighbourhood in $\Symp^{S^1}(\overline{D})$ then this  procedure gives us a lift of $\psi$ near id in $\Symp^{S^1}(\overline{D}) $ to $\Stab^{S^1}(\overline{D})$. By Theorem~\ref{palais} this shows $r$ is fibration. 
\end{proof}

\begin{lemma}
The group $\Symp^{S^1}(\overline{D})$ retracts onto $\SO(3)$ for the circle action $S^1(0,\pm 1,m)$, while $\Symp^{S^1}(\overline{D})$ retracts onto $S^1$ for all other circle actions.
\end{lemma}

\begin{proof}
Consider the circle action induced on $\overline{D}$. The action $S^1(0,\pm 1,m)$ fixes $\overline{D}$ pointwise. Hence $\Symp^{S^1}(\overline{D}) = \Symp(\overline{D})$. By Smale's theorem we know that $\Symp(\overline{D})$ is homotopy equivalent to $\SO(3)$. 
\\

The restriction for all other actions not equal to $S^1(0,\pm 1,m)$, do not pointwise fix the curve $\overline{D}$. For all other actions, we have the following two subcases. Assume that the action is effective. Let $\mu: \overline{D} \to \R$ be it's moment map. Then as explained in the proof of Proposition \ref{prop:CentralizersToricActions} we have that  $\Symp^{S^1}(\overline{D}) \simeq C^\infty(\mu(\overline{D}),S^1)$, where $C^\infty(\mu(\overline{D}),S^1)$ denotes the space of smooth maps from the image of the moment map to $S^1$. As the image of the moment map is an interval, and as the space of smooth maps from an interval to $S^1$ is homotopy equivalent to $S^1$, we have the required result that $ \Symp^{S^1}(\overline{D}) \simeq S^1$.
\\

Finally if the induced symplectic $S^1$ action on $\overline{D}$ is not effective and has $\Z_k$ stabilizer, the action of $S^1/\Z_k \cong S^1$, is effective and the space of symplectomorphisms equivariant with respect to this quotient  effective action is the same as space of symplectomorphisms equivariant with respect to the non-effective $S^1$ action. Thus the homotopy type of $\Symp^{S^1}(\overline{D}) \simeq S^1$.
\end{proof}

\begin{lemma}\label{gauge}
The map
\begin{align*} 
   \alpha: \Fix^{S^1}(\overline{D}) &\twoheadrightarrow \Aut^{S^1}(N(\overline{D})) \\
     \phi &\mapsto d\phi|_{N(\overline{D})}
\end{align*}
 is a Serre fibration with fibre homotopic to $\Fix^{S^1}(N(\overline{D}))$. The group $\Aut^{S^1}(N(\overline{D}))$  is homotopy equivalent to $S^1$. 
\end{lemma}

\begin{proof}
The fact that $\Aut^{S^1}(N(\overline{D})) \simeq S^1$ is explained in Appendix~\ref{AppendixAutomorphisms}. Thus we only need to prove that the restriction of the derivative is a  fibration and that the fibre is homotopy equivalent to $\Fix^{S^1}(N(\overline{D}))$.
\\

Consider the action 
\begin{align*}
     \Fix^{S^1}(\overline{D}) \times \Aut^{S^1}(N(\overline{D})) &\to\Aut^{S^1}(N(\overline{D})) \\
     (\phi, \psi) &\to  d\phi|_{N(\overline{D})} \circ \psi 
\end{align*}
Again, by Theorem \ref{palais}, it suffices to show that the above map admits local sections, and such sections exist by Lemma~\ref{EqSymN}. The fibre is made of all equivariant symplectomorphisms that act as identity on the normal bundle of $\overline{D}$. By Proposition~\ref{prop:HomotopyEquivalenceSympN-Near}, the space is homotopy equivalent to  $\Fix^{S^1}(N(\overline{D}))$.
\end{proof}

Let $\overline{\mathcal{S}^{S^1}_{F,p_0}}$ be the space of unparametrized $S^1$-invariant symplectic spheres in the homology class $F$ that are equal to a fixed invariant curve ${\overline{F}}$ in a neighbourhood of $p_0$.

\begin{lemma}
The map \begin{align*}
    \Fix^{S^1}(N(\overline{D})) &\to \overline{\mathcal{S}^{S^1}_{F,p_0}} \\
    \phi \mapsto \phi(\overline{F})
    \end{align*} 
is a fibration and $\overline{\mathcal{S}^{S^1}_{F,p_0}} \simeq \jsom(\overline{D}) \simeq \{*\}$
\end{lemma}

\begin{proof}
We first note that if $F' \in \overline{\mathcal{S}^{S^1}_{F,p_0}}$ then the map $\phi$ constructed in the proof of Corollary~\ref{trans} belongs to $\Fix^{S^1}(N(\overline{D}))$. Hence, the group $\Fix^{S^1}(N(\overline{D}))$ acts transitively on $\overline{\mathcal{S}^{S^1}_{F,p_0}}$, and the action map induces a fibration.\\

The fact that $\jsom(\overline{D}) \simeq \{*\}$ is given in Appendix~\ref{ElementaryEquivariantTopology}, lemma~\ref{Lemma:Contractibilityofacs}.\\
 
To show that $\overline{\mathcal{S}^{S^1}_{F,p_0}} \simeq \jsom(\overline{D}) \simeq \{*\}$, let $\mathcal{S}^\perp_{F,p_0}$ denote the space of all $S^1$ invariant symplectically embedded spheres S in class F such that $S \cap \overline{D} =p_0$ and $S$ and $\overline{D}$ intersect $\oml$-orthogonally at $p_0$. By Lemma~\ref{transverse} we see that for every $S \in  \mathcal{S}^\perp_{F,p_0}$ there exists a $J \in \jsom(\overline{D})$ such that the configuration $S \vee \overline{D}$ is $J$-holomorphic.  We now have the following fibration
\begin{equation*}
\jsom(\overline{D}) \longtwoheadrightarrow  \mathcal{S}^\perp_{F,p_0}
\end{equation*}
where the map $\gamma: \jsom(\overline{D}) \to  \mathcal{S}^\perp_{F,p_0}$ sends $J \in \jsom(\overline{D})$ to the corresponding curve in class $F$ passing through $p_0$. Note that the above map is well defined because the fixed $p_0$ was chosen such that the weights at $p_0$ were distinct. Hence any $S^1$ invariant $F$ curve passing through $p_0$ must be $\oml$-orthogonal to $\overline{D}$. Now we show that $\gamma$ is a homotopy equivalence. To do that we consider the following commutative diagram
\[ 
\begin{tikzcd}
 &T \arrow[d,"\pi_2"] \arrow[r,"\pi_1"] & \mathcal{S}^\perp_{F,p_0} \\
 &\jsom(\overline{D}) \arrow[ur, "\gamma"]
\end{tikzcd} 
\]
where $T:= \left\{ (A,J) \in  \mathcal{S}^\perp_{F,p_0} \times \jsom(\overline{D})~|~  A ~~\text{is J-holomorphic} \right\}$. Both the maps $\pi_1$ and $\pi_2$ are fibrations (this can be argued as in Lemma \ref{first}) with contractible fibres. As the diagram commutes, the map $\gamma$ must be a homotopy equivalence. The proof then follows by showing that $\overline{\mathcal{S}^{S^1}_{F,p_0}} \simeq \mathcal{S}^\perp_{F,p_0}$, which is a consequence of Theorem~\ref{gmpf}.
\end{proof}

\begin{lemma}
\begin{align*}
\Stab^{S^1}(\overline{F}) \cap \Fix^{S^1}(N(\overline{D})) &\to \Symp^{S^1}(\overline{F}, N(p_0)) \\
\phi &\mapsto  \phi|_{\overline{F}}
\end{align*}
is a fibration and $\Symp^{S^1}(\overline{F}, N(p_0)) \simeq \{*\}$
\end{lemma}
 
\begin{proof}
The fact that this is a fibration follows from applying the proof of Lemma \ref{stab} mutatis mutandis. The proof that $\Symp^{S^1}(\overline{F}, N(p_0)) \simeq \{*\}$, is similar to Lemma \ref{stab}.  $\Symp^{S^1}(\overline{F}, N(p_0))$ is homotopy equivalent to maps from the interval $[0,1]$ to $S^1$ that is identity near 0. The space of such maps is contractible thus giving the result. 
\end{proof}

\begin{lemma}\label{ngauge}
\begin{align*}
    \Fix^{S^1}(\overline{F}) &\to \Aut^{S^1}(N(\overline{D} \vee \overline{F}))\\
    \phi &\mapsto d\phi|_{N(\overline{D} \vee \overline{F})}
\end{align*}
is a fibration and $\Aut^{S^1}(N(\overline{D} \vee \overline{F}))\simeq \{*\}$ and the fibre $\Fix^{S^1}(N(\overline{D} \vee \overline{F})) \simeq \{*\}$
\end{lemma}

\begin{proof}
The proof that this is a fibration is similar to the proof of Lemma \ref{gauge}. The fact that $\Aut^{S^1}(N(\overline{D} \vee \overline{F}))\simeq \{*\}$ follows from by Lemma \ref{Gauge(N(D))}. The fact that $\Fix^{S^1}(N(\overline{D} \vee \overline{F})) \simeq \{*\}$ follows from Theorem \ref{thm:EqGr}.
\end{proof}

Putting all the fibrations together gives the following theorem.
\begin{thm}\label{homogenous}
Consider the $S^1(a,b;m)$ action on $(\SSS,\oml)$ with $\lambda >1$. If $ \jsom \cap U_{2s}$ is non-empty, then we have the following homotopy equivalences:
\begin{enumerate}
\item when $(a,b) \neq (0,\pm1)$, we have $\Symp^{S^1}_h(\SSS,\oml)/\T_{2s} \simeq \jsom \cap U_{2s}$;
\item when $(a,b) = (0,\pm1)$, we have $\Symp^{S^1}_h(\SSS,\oml)/(SO(3) \times S^1) \simeq \jsom \cap U_{2s}$.
\end{enumerate} 
\end{thm} 
\begin{proof}

When $(a,b;m) \neq (0,\pm 1;0)$ we have a commutative diagram of fibrations
\[
\begin{tikzcd}
&\Fix^{S^1}(\overline{D}) \arrow[r] &\Stab^{S^1}(\overline{D}) \arrow[r,twoheadrightarrow] &\Symp^{S^1}(\overline{D}) \\
&S^1 \arrow[u,hookrightarrow] \arrow[r] &\T_{2s} \arrow[u,hookrightarrow] \arrow[r] &S^1 \arrow[u,hookrightarrow]
\end{tikzcd}
\]
while in the case $(a,b) = (0,\pm1)$, we have the diagram
\[
\begin{tikzcd}
&\Fix^{S^1}(\overline{D}) \arrow[r] &\Stab^{S^1}(\overline{D}) \arrow[r,twoheadrightarrow] &\Symp^{S^1}(\overline{D}) \\
&S^1 \arrow[u,hookrightarrow] \arrow[r] & S^1 \times SO(3) \arrow[u,hookrightarrow] \arrow[r] &SO(3) \arrow[u,hookrightarrow]
\end{tikzcd}
\]
From the discussion above, in both the diagrams the leftmost and the rightmost arrows are homotopy equivalences.  As the diagram commutes, the 5 lemma implies that the middle inclusion $\T \hookrightarrow \Stab^{S^1}(\overline{D})$  or $\left(S^1 \times SO(3)\right) \hookrightarrow \Stab^{S^1}(\overline{D})$ are also homotopy equivalences. This gives us the required result.
\end{proof}

\begin{remark}\label{ActionOnJsom}
Let $J_{2s}$ be the standard complex structure on $W_{2s}$. We note that for the circle action $S^1(0,\pm1,;m)$ the  stabiliser of $J_{2s}$ under the natural action of $\Symp^{S^1}_h(\SSS,\oml)$ on $\jsom \cap U_{2s}$ is the group of K\"ahler isometries $S^1\times\SO(3)$. For all other circle actions $S^1(a,b;m)$ with $(a,b)\neq(0,\pm1)$, the stabiliser of $J_{2s}$ is the maximal torus $\T_{2s}\subset S^1 \times \SO(3)$.
\end{remark}

\subsection{Case 2: \texorpdfstring{$\Symp^{S^1}_h(\SSS,\oml)$}{Symp(S2xS2)} action on \texorpdfstring{$\jsom \cap U_0$}{J\hat{}S1\_0}}\label{IntwithU0}
In order to describe the action of $\Symp^{S^1}_h(\SSS,\oml)$ on the open stratum $\jsom \cap U_0$, we need to  modify slightly the setting introduced in the previous section. The main difference comes from the fact that for an almost-complex structure $J\in\jsom\cap U_0$, there is no invariant curve with negative self-intersection  representing a class $B-kF$, $k\geq 1$. Instead, each such $J$ determines a regular 2-dimensional foliation of $J$-holomorphic curves in the class $B$. Consequently, there is no natural map between the stratum $\jsom \cap U_0$ and the space $\mathcal{S}^{S^1}_{B}$ of invariant curves in the class $B$. However, once we choose a fixed point $p_0$, given any $J \in \jsom \cap U_0$, there is a unique invariant $J$-holomorphic curve in the class $B$ passing through $p_0$. This defines a map $\jsom \cap U_0 \to \mathcal{S}^{S^1}_{B,p_0}$ that can be used to prove that the space $\jsom \cap U_0$ is homotopy equivalent to an orbit of $\Symp^{S^1}_h(\SSS,\oml)$. To do so, because the fixed point $p_0$ is not unique, we must also investigate how the group $\Symp^{S^1}_h(\SSS,\oml)$ acts on the fixed point set of the circle action. This is done in Lemma~\ref{lemma:SymphPreservesAnIsolatedFixedPoint}. Before we proceed to prove this lemma we first describe the action of $\Symp^{S^1}_h(\SSS,\oml)$ on $\jsom \cap U_0$. Note that by Theorem~\ref{cor:IntersectingOnlyOneStratum}, Corollary~\ref{cor:CircleExtensionsWith_a=1} and Corollary~\ref{cor:CircleExtensionsWith_a=-1}, the space $\jsom \cap U_0$ is non-empty only for the following circle actions:
\begin{itemize}
    \item $S^1(a,b;0)$, or
    \item  $S^1(1,b;m)$ with $|2b-m|=0$ and $2\lambda > |2b-m|$, or
    \item $S^1(-1,b;m)$ with $|2b+m|=0$ and $2\lambda > |2b+m|$.
\end{itemize}
Secondly, we observe that all these actions have at least one isolated fixed point \emph{except} the actions of the forms
\begin{itemize}
    \item $S^1(\pm 1,0;0)$ and
    \item $S^1(0,\pm 1;0)$
\end{itemize}
\subsubsection{Actions with an isolated fixed point} We now consider circle actions $S^1(a,b;m)$ with an isolated fixed point $p_0$. We can choose $p_0$ to correspond to the vertex $R$ in the Hirzerbruch surface $W_m$ shown in Figure~\ref{hirz}. Given $J\in\jsom\cap U_0$, there is a unique $J$-holomorphic curve $B_{p_0,J}$ in class $B$ that passes through $p_0$. Because $p_0$ is fixed, $J$ is invariant, and $B\cdot B=0$, positivity of intersection implies that $B_{p_0,J}$ is $S^1$-invariant. We thus get a well-defined map 
\[\jsom \cap U_0 \to \mathcal{S}^{S^1}_{B,p_0}\]
where $\mathcal{S}^{S^1}_{B,p_0}$ denotes the space of invariant, embedded, symplectic spheres representing the class $B$ and containing the point $p_0$. 
\begin{lemma}\label{projection_surjective}
Consider any $S^1(a,b; m)$ action on $(\SSS,\oml)$. Let $p_0$ and $p_1$ be two fixed points such that there exists an invariant fibre $\{*\} \times S^2$ passing through them . Then there exists no $S^1$ invariant curve in the class $B-kF$ for $k \geq 0$ passing through $p_0$ and $p_1$.
\end{lemma}
\begin{proof}
Suppose not, let $\overline{D_{2s}}$ be a $S^1$ invariant curve in the class $B-kF$ with $k \geq 0$ passing through $p_0$ and $p_1$. Then the projection onto the first factor
$$ \pi_1 : \overline{D_{2s}} \rightarrow S^2 \times \{0\} \subset \SSS$$
is surjective. Hence the curve $\overline{D_{2s}}$ passes through a third  fixed point $p_2$. As the symplectic $S^1$ action on $\overline{D_{2s}}$ has three fixed points, it has to fix $\overline{D_{2s}}$ pointwise. This is a contradiction as all fixed surfaces for $S^1$ actions must be either a maximum or minimum for the moment map, but the fixed points $p_2$, $p_1$ and $p_0$ cannot have the same moment map value.
\end{proof}

\begin{lemma}\label{lemma:SymphPreservesAnIsolatedFixedPoint}
Let $S^1(a,b;m)$ be a circle action for which the space $\jsom \cap U_0$ is non-empty. Assume there is an isolated fixed point $p_0$ corresponding to the vertex $R$ in Figure~\ref{hirz}.  Then any equivariant symplectomorphism that preserves homology $\phi\in \Symp_h^{S^1} (\SSS,\oml)$ fixes~$p_0$.
\end{lemma}
\begin{proof}
\emph{Case 1: $\lambda >1$:}
By Lemma~\ref{lemma:CharacterizationCentralizer} and Corollary~\ref{cor:ActionPreservesWeights} any such $\phi$ must preserve the moment values and the weights of the fixed points (up to change of order of the tuples). These weights are given in Table~\ref{table_weights} and  the moment map values are given in the graphs~\ref{fig:GraphsWithFixedSurfaces} and \ref{fig:GraphsIsolatedFixedPoints}. The fact that $\jsom \cap U_0$ is non-empty for the circle action implies  that either $m=0$, $|2b-m|=0$, or $|2b+m|=0$. It is now easy to see that under any of these three numerical conditions, the weights and moment map values at $R$ differ from the weights at all other fixed points. The result follows.
\\

\emph{Case 2: $\lambda = 1$:}
If the actions are \emph{not} of  the form $S^1(1,1; 0)$ or $S^1(-1,-1;0)$ with  $\lambda = 1$, then an argument similar to Case 1 holds. 
The only case left are the actions of the form $S^1(1,1; 0)$ or $S^1(-1,-1;0)$ with  $\lambda = 1$. In this case, the  homology classes $F$, $B$ have the same area and the fixed points $R$ and $Q$ have the same weights (up to change of order of tuples) and the same moment map values. We again argue by contradiction in this case. Let $\overline{B}$ denote a fixed curve in class $B$ passing through $R$ and $P$. Suppose $\phi \in \Symp_h^{S^1} (\SSS,\oml)$ doesn't fix the point ~$p_0 = R$. Then $\phi$ has to take the point $R$ to the point $Q$. Further by Lemma~\ref{lemma:CharacterizationCentralizer}, $\phi$ fixes the maximum and minimum and hence $\phi(P) = P$. As $\phi$ preserves homology, the curve $\phi(\overline{B})$ has homology class $B$ and has as to pass through $Q$ and $P$ which contradicts  Lemma~\ref{projection_surjective}. 
\end{proof}

Let $J_0\in U_0$ be the complex structure of the Hirzebruch surface $W_0$ and let $B_{p_0}$ be the unique $J_0$-holomorphic curve containing $p_0$ and representing the homology class $B$.
\begin{cor}
Let $S^1(a,b;m)$ be a circle action with an isolated fixed point and for which the structure $J_0\in U_0$ is invariant. Then the group $\Symp^{S^1}_h(\SSS,\oml)$ acts transitively on the space $\mathcal{S}^{S^1}_{B,p_0}$, and the action map 
\begin{align*}
\Symp^{S^1}_{h}(\SSS,\oml) &\longtwoheadrightarrow \mathcal{S}^{S^1}_{B,p_0}\\
\phi & \mapsto \phi(B_{p_0})
\end{align*}
is a Serre fibration.
\end{cor}
\begin{proof}

Since any element of $\Symp^{S^1}_h(\SSS,\oml)$ fixes $p_0$, it follows that this group acts on $\mathcal{S}^{S^1}_{B,p_0}$. Let $p_1$ be the other fixed point corresponding to the point $Q$ in Figure~\ref{hirz}. All curves in  $\mathcal{S}^{S^1}_{B,p_0}$ pass through $p_1$ and $p_0$. Since the weights at one of $p_0$ or $p_1$ are always distinct, we can use the fixed point with distinct weights and proceed as in the proofs of Corollary ~\ref{trans} and Lemma~\ref{first} to show that the action defines a fibration.
\end{proof}
As before, we can now show that the stratum $\jsom \cap U_{0}$ is homotopy equivalent to a space of invariant curves.
\begin{lemma}
The natural map $\alpha: \jsom \cap U_{0} \to {\mathcal{S}^{S^1}_{B,p_0}}$ defined by sending an almost complex structure $J \in \jsom \cap U_{0}$ to the unique $J$-holomorphic curve in class $B$ passing through $p_0$ is a weak homotopic equivalence. 
\end{lemma}
\begin{proof}
The argument is identical to the proof of Lemma~\ref{first2}
\end{proof}

From now on, we can determine the homotopy type of $\jsom \cap U_{0}$ by going through a similar sequence of fibrations and homotopy equivalences as in Section~\ref{section:ActionOnU_2k}, namely,
\[\Stab^{S^1}(B_{p_0}) \to \Symp^{S^1}_{h}(\SSS,\oml) \longtwoheadrightarrow \mathcal{S}^{S^1}_{B,p_0} \mathbin{\textcolor{blue}{\xrightarrow{\text{~~~$\simeq$~~~}}}}  \jsom \cap U_{0}\rule{0em}{2em}\]
    
\[\Fix^{S^1}(B_{p_0}) \to \Stab^{S^1}_{p_0}(B_{p_0}) \longtwoheadrightarrow  \Symp^{S^1}(B_{p_0}) \mathbin{\textcolor{blue}{\xrightarrow{\text{~~~$\simeq$~~~}}}} S^1\rule{0em}{2em}\]

\[\Fix^{S^1} (N(B_{p_0})) \to \Fix^{S^1}(B_{p_0}) \longtwoheadrightarrow  \Aut^{S^1}(N(B_{p_0})) \mathbin{\textcolor{blue}{\xrightarrow{\text{~~~$\simeq$~~~}}}} S^1\rule{0em}{2em} \]

\[\Stab^{S^1}(\overline{F}) \cap \Fix^{S^1}(N(B_{p_0})) \to \Fix^{S^1}(N(B_{p_0})) \longtwoheadrightarrow  \overline{\mathcal{S}^{S^1}_{F,p_0}} \mathbin{\textcolor{blue}{\xrightarrow{\text{~~~$\simeq$~~~}}}} \mathcal{J}^{S^1}(B_{p_0})\rule{0em}{2em}\]
   
\[\Fix^{S^1}(\overline{F}) \to \Stab^{S^1}(\overline{F}) \cap \Fix^{S^1}(N(B_{p_0})) \longtwoheadrightarrow  \Symp^{S^1}(\overline{F}, N(p_0))  \mathbin{\textcolor{blue}{\xrightarrow{\text{~~~$\simeq$~~~}}}} \left\{*\right\}\rule{0em}{2em}\]

\[\left\{*\right\} \mathbin{\textcolor{blue}{\xleftarrow{\text{~~~$\simeq$~~~}}}} \Fix^{S^1}(N(B_{p_0} \vee \overline{F})) \to \Fix^{S^1}(\overline{F}) \longtwoheadrightarrow  \Aut^{S^1}(N(B_{p_0} \vee \overline{F}))  \mathbin{\textcolor{blue}{\xrightarrow{\text{~~~$\simeq$~~~}}}} \left\{*\right\}\rule{0em}{2em}\]
where $\overline{\mathcal{S}^{S^1}_{F,p_0}}$ denotes the space of all symplectically embedded curve in the class $F$ that pass through $p_0$ and agree with a standard curve $F_{p_0}$ in a neighbourhood of $p_0$. The proofs that these maps are fibrations, and the proofs of the homotopy equivalences are exactly the same as before.
Consequently, we obtain the following homotopical description of $\jsom \cap U_{0}$.
\begin{thm}\label{Thm_isolatedfixedpt}
Consider one of the following circle actions on $(\SSS,\oml)$
\begin{itemize}
    \item $S^1(a,b;0)$ with $(a,b)\neq (\pm1,0)$ and $(a,b)\neq(0,\pm1)$, or
    \item  $S^1(1,b;m)$ with $|2b-m|=0$ and $2\lambda > |2b-m|$, or
    \item $S^1(-1,b;m)$ with $|2b+m|=0$ and $2\lambda > |2b+m|$.
\end{itemize}
Then the stratum $\jsom \cap U_0$ is non-empty and 
\[\Symp^{S^1}_{h,p_0}(\SSS,\oml)/ \T_0 \simeq \jsom \cap U_{0}\]
\end{thm}\qed


\subsubsection{Actions without isolated fixed points} 

We now turn our attention to the action of the group of equivariant symplectomorphisms $\Symp^{S^1}_h(\SSS,\oml)$ on the stratum $\jsom \cap U_{0}$ when the circle action is either
\begin{enumerate}
    \item $S^1(\pm 1,0;0)$ or
    \item $S^1(0,\pm 1;0)$.
\end{enumerate}

These actions has no isolated fixed points and the associated graphs are of the form
\begin{figure}[H]
    \centering


\subcaptionbox{Subcase 1: $S^1(\pm 1,0;0)$  }
[.45\linewidth]
{\begin{tikzpicture}
[scale=0.85, every node/.style={scale=0.85}]
\draw [fill] (0,0) ellipse (0.3cm and 0.1cm); 
\draw [fill] (0,2.8) ellipse (0.3cm and 0.1cm); 
\node[above] at (0,2.9) {$F_{max}$}; 
\node[below] at (0,-0.1) {$F_{min}$}; 
\node[left] at (-0.3,2.8) {$\mu = \lambda$};
\node[left] at (-0.3,0){$\mu=0$};
\node[right] at (0.3,2.8){$A= 1$};
\node[right] at (0.3,0){$A= 1$};
\end{tikzpicture}
}
\quad
\subcaptionbox{Subcase 2: $S^1(0,\pm 1;0)$}
[.45\linewidth]
{\begin{tikzpicture}
[scale=0.85, every node/.style={scale=0.85}]
\draw [fill] (0,0) ellipse (0.3cm and 0.1cm); 
\draw [fill] (0,2.8) ellipse (0.3cm and 0.1cm); 
\node[above] at (0,2.9) {$B_{max}$}; 
\node[below] at (0,-0.1) {$B_{min}$}; 
\node[left] at (-0.3,2.8) {$\mu = 1$};
\node[left] at (-0.3,0){$\mu=0$};
\node[right] at (0.3,2.8){$A= \lambda$};
\node[right] at (0.3,0){$A= \lambda$};
\end{tikzpicture}
}
\end{figure}\label{subcase2:Fixedsurface}
\noindent where $\mu$ denotes the value of the moment map and $A$ denotes the area of the fixed surface. We notice that there are pointwise fixed curves in the class $F$ for the circle action $S^1(\pm 1,0;0)$ and pointwise fixed curves in class $B$ for the action $S^1(0,\pm 1;0)$. We denote the fixed surface which is a minimum for the moment map as $F_{min}$, $B_{min}$ respectively and the maximum by $F_{max}$, $B_{max}$.
\\

Consider the action $S^1(0,\pm 1;0)$. By Lemma~\ref{lemma:CharacterizationCentralizer} we note that any symplectomorphism $\phi \in \Symp^{S^1}_h(\SSS,\oml)$ must send $B_{max}$ to itself. Then, given $p_0 \in B_{max}$ corresponding to the fixed point $R$ in Figure~\ref{hirz}, we define the following sequence of fibrations and homotopy equivalences:
\[\Fix^{S^1}(B_{max}) \longrightarrow \Symp^{S^1}_h(\SSS,\oml)  \longtwoheadrightarrow \Symp(B_{max}) \mathbin{\textcolor{blue}{\xrightarrow{\text{~~~$\simeq$~~~}}}} SO(3)\rule{0em}{2em}\]
\[\Stab^{S^1}(F_{p_0}) \longrightarrow \Fix^{S^1}(B_{max}) \longtwoheadrightarrow \overline{\mathcal{S}^{S^1}_{F,p_0}} \mathbin{\textcolor{blue}{\xrightarrow{\text{~~~$\simeq$~~~}}}} \jsom \simeq \{*\}\rule{0em}{2em}\]
\[\Fix^{S^1} (F_{p_0}) \longrightarrow \Stab^{S^1}(F_{p_0}) \longtwoheadrightarrow \Symp^{S^1}(F_{p_0}) \mathbin{\textcolor{blue}{\xrightarrow{\text{~~~$\simeq$~~~}}}} S^1\rule{0em}{2em}\]
\[\left\{*\right\}\mathbin{\textcolor{blue}{\xleftarrow{\text{~~~$\simeq$~~~}}}} \Fix^{S^1}(N(B_{max} \vee F_{p_0})) \longrightarrow  \Fix^{S^1}(F_{p_0})\longtwoheadrightarrow \Aut^{S^1}(N(B_{max} \vee \overline{F}))
\mathbin{\textcolor{blue}{\xrightarrow{\text{~~~$\simeq$~~~}}}}  \left\{*\right\}\rule{0em}{2em}\]

For the other circle action $S^1(\pm 1,0;0)$, we obtain a similar sequence of fibrations and homotopy equivalences in which $B_{\max}$ is replaced by the curve $F_{\max}$. As before, putting all the homotopy equivalences together, we obtain the following theorem:
\begin{thm}\label{Thm_fixedsurfaces}
Consider the following two circle actions on $(\SSS,\oml)$  
\begin{itemize}
\item $S^1(\pm 1,0;0)$ or
\item  $S^1(0,\pm 1;0)$
\end{itemize}
Then there is a homotopy equivalence
\[\Symp^{S^1}_h(\SSS,\oml)/(S^1 \times \SO(3)) \simeq \jsom \cap U_0\]
\end{thm}\qed

For convenience, we collect together the two main results of this section in the theorem below. 
\begin{thm}{\label{homog}}
Consider the action $S^1(a,b;m)$ on $(\SSS,\oml)$ such that  one of the following  hold:
\begin{itemize}
    \item $S^1(a,b;0)$ with $(a,b)\neq (\pm1,0)$ and $(a,b)\neq(0,\pm1)$, or
    \item  $S^1(1,b;m)$ with $|2b-m|=0$ and $2\lambda > |2b-m|$, or
    \item $S^1(-1,b;m)$ with $|2b+m|=0$ and $2\lambda > |2b+m|$.
\end{itemize}
Then the stratum $\jsom \cap U_0$ is non-empty and 
\[\Symp^{S^1}_{h,p_0}(\SSS,\oml)/ \T_0 \simeq \jsom \cap U_{0}\]
If instead the $S^1(a,b;m)$ action satisfies 
\begin{itemize}
\item $(a,b;m) = (\pm 1,0;0)$ or
\item  $(a,b;m) = (0,\pm 1;0)$
\end{itemize} then we have that \[\Symp^{S^1}_h(\SSS,\oml)/(S^1 \times \SO(3)) \simeq \jsom \cap U_0\]
and $\jsom$ intersects only the strata $U_0$.
\end{thm}\qed


\chapter{The homotopy type of the symplectic centralisers of \texorpdfstring{$S^1(a,b;m)$}{S1(a,b;m)}}\label{Chapter:homotopy type of the symplectic centralisers}
Given any Hamiltonian circle action on $(\SSS,\oml)$, the two Theorems~\ref{cor:IntersectingOnlyOneStratum} and~\ref{cor:IntersectingTwoStrata} give us a complete understanding of which strata the space $\jsom$ intersects. Together with Theorems~\ref{homogenous}, and~\ref{homog} describing the strata as homogeneous spaces, this allows us to compute the homotopy type of the group of equivariant symplectomorphisms.

\section{When \texorpdfstring{$\jsom$}{J\hat S1} is homotopy equivalent to a single symplectic orbit}
The homotopy type of $ $ can be easily determined whenever the circle action $S^1(a,b;m)$ has a single invariant stratum. This occurs precisely when one of the conditions on $a,b,m,\lambda$ listed in Table~\ref{table:CentralizersCircleSSS1Stratum} are met.
\begin{thm}\label{table}
Consider the  circle action $S^1(a,b;m)$ on $(\SSS,\oml)$ for which there is a single invariant stratum. The homotopy type of the symplectic centralizer $\Symp^{S^1}(S^2 \times S^2,\oml)$ is given in Table~\ref{table:CentralizersCircleSSS1Stratum}.
\end{thm}
{\centering
\begin{table}[h!]
\begin{tabular}{|p{40mm}|p{48mm}|p{26mm}|}
\hline
 $S^1$ action $(a,b;m)$ & Conditions on $\lambda$ &$\Symp^{S^1}(S^2 \times S^2)$\\
\hline
\hline
{$(0,\pm 1;m)$ ~~ $m\neq 0$}  &$2\lambda > m$ & $\simeq S^1 \times SO(3)$ \\
\hline
$(0,\pm 1;0)$ or $(\pm 1,0;0)$ &$\lambda \geq 1$ &$\simeq S^1 \times SO(3)$ \\

\hline
$(\pm 1,\pm1,0)$&$\lambda = 1$ & $\simeq \T \times \mathbb{Z}_2$ \\
\hline
 $(\pm 1,0;m) ~ m\neq0$ &$2\lambda > m$ &$\simeq \T$\\
 \hline
$(\pm 1,\pm m;m)$ $m \neq 0$ & $2\lambda > m$ & $\simeq \T$\\
\hline
$( 1,b;m)$ $b \neq \{ m,0\}$ &$2\lambda > m$ and $|2b-m| \geq2 \lambda \geq 1$ &$\simeq \T$ \\
\hline
$(-1,b;m), b \neq \{ -m,0\}$ & $2\lambda > m$ and $|2b+m| \geq2 \lambda \geq 1$ &$\simeq \T$ \\
\hline
All other values of\newline $(a,b;m)$ except $(\pm 1,b;m)$ &$ 2\lambda > m$   &$\simeq \T$ \\
\hline
\end{tabular}\caption{Centralizers of circle actions on $(\SSS,\oml)$ with a single invariant stratum}
\label{table:CentralizersCircleSSS1Stratum}
\end{table}
}

\begin{proof}
By Theorem~\ref{cor:IntersectingOnlyOneStratum}, in each of the above $S^1(a,b;m)$ actions, the space of $S^1$ invariant compatible almost complex structures $\jsom$ only intersects the stratum $U_m$. Consequently,
\[\Symp_h^{S^1}(\SSS,\oml)/\Stab(J_m) \simeq \jsom \cap U_m = \jsom \simeq \{*\}\]
where $\Stab(J_m)$ denotes the stabiliser of the standard complex structure $J_m \in U_m$. Thus, for all the actions in the table, we have that $\Symp_h^{S^1}(\SSS,\oml) \simeq \Stab(J_m)$. For the $S^1$ action given by the triples $(0,\pm 1,m)$, $(\pm1,0,0)$ or the circle action $S^1(0,\pm1,0)$ when $\lambda =1$, Theorems~\ref{homogenous} and~\ref{homog} imply that $\Stab(J_m) \simeq S^1 \times \SO(3)$. For all other $S^1$ actions in the table, the stabilizers are homotopy equivalent to $\T$. 
\\

We now show how to recover the homotopy type of the full group $\Symp^{S^1}(\SSS,\oml) $ from the homotopy type of the subgroup $\Symp_h^{S^1}(\SSS,\oml)$. When $\lambda > 1$, this is immediate as we have the equality $\Symp^{S^1}(\SSS,\oml) = \Symp_h^{S^1}(\SSS,\oml)$ which follows from  Lemma~\ref{lemma:ActionOnHomology}.
\\

When $\lambda = 1$ and $a\neq b$, 
there exists standard $S^1(a,b;m)$ invariant curves in classes $B$ and $F$ such that the isotropy weight of the action on the curve in class $B$ is $a$ and the isotropy weight of the $S^1$ action on the curve in the class $F$ is $b$. Hence, as $\phi$ is an equivariant symplectomorphism, Lemma~\ref{lemma:ActionOnHomology} implies that must have $\phi_*[F] = [F]$ and $\phi_*[B] = [B]$. Consequently, $\Symp_h^{S^1}(\SSS,\oml) = \Symp^{S^1}(\SSS,\oml)$.
\\

In the special case when $\lambda =1$ and $a=b = \pm 1$, then we have an equivariant version of the exact sequence~\ref{Sequence:ActionOnHomology}
\begin{equation*}
   1\to \Symp_h^{S^1}(\SSS,\oml)\to  \Symp^{S^1}(\SSS,\oml)\to \Aut_{c_1,\oml}(H^2(\SSS))\to 1
\end{equation*}

where $\Aut_{c_1,\oml}(H^2(\SSS)) \cong \Z_2$. The map \begin{align*}
    \phi: \SSS \to \SSS \\
    (z,w) \mapsto (w,z)
\end{align*} 
is a $S^1$ equivariant symplectomorphism (for the action $S^1(1,1,0)$ or $ (-1,-1,0)$) and gives a section from $\Z_2 \cong \Aut_{c_1,\oml}(H^2(\SSS))$ to $\Symp^{S^1}(\SSS,\oml)$. Thus we have $\Symp^{S^1}(\SSS,\oml) \cong  \Symp_h^{S^1}(\SSS,\oml) \rtimes \Z_2$. As the semidirect product of two topological groups is homotopy equivalent(as a topological space) to the direct product of the groups,  we have that  $\Symp^{S^1}(\SSS,\oml) \cong  \Symp_h^{S^1}(\SSS,\oml) \rtimes \Z_2 \simeq  \Symp_h^{S^1}(\SSS,\oml) \times \Z_2 \simeq \T \times \Z_2$. This completes the proof.
\end{proof}

\section[The two orbits case]{When \texorpdfstring{$\jsom$}{J\hat{}S1} is homotopy equivalent to the union of two orbits}\label{section:TwoOrbits}

Theorem~\ref{table} gives the homotopy type of the group of equivariant symplectomorphisms for all circle actions on $\SSS$ apart from the following two families of actions:
\begin{itemize}
\item (i) $a=1$, $b \neq \{0, m\}$,  and $2\lambda > |2b-m|$; or
\item (ii) $a=-1$, $b \neq \{0, -m\}$, and $2 \lambda > |2b+m|$.
\end{itemize}
For convenience, we will write $m'$ for either $|2b-m|$ or $|2b+m|$ depending on which of the above families we consider. Up to swapping $m$ and $m'$, we will also assume $m'>m$. The goal of this section is to show that the symplectic stabilizer of any of these circle actions is homotopy equivalent, \emph{in the category of topological group}, to the pushout of the two tori $\T_m$ and $\T_{m'}$ along the common $S^1$.\\ 

Before delving into technicalities, it may be useful to outline the argument, which is an adaptation of the Anjos-Granja idea used in~\cite{AG} to compute the homotopy type of the full group of symplectomorphisms of $\SSS$ for $1<\lambda\leq 2$. The first step is to show that the two inclusions \[\T_{m} \into \Symp_h^{S^1}(\SSS,\oml)\quad \text{~and~}\quad \T_{m'} \into \Symp_h^{S^1}(\SSS,\oml)\]
induce injective maps in homology. By the Leray-Hirsch theorem, it follows that the cohomology modules of the total space of the fibrations
\[\T_m\to\Symp^{S^1}_h(\SSS,\oml)\to \Symp^{S^1}_h(\SSS,\oml)/\T_m \simeq \jsom\cap U_m\]
\[\T_{m'}\to\Symp^{S^1}_h(\SSS,\oml)\to\Symp^{S^1}_h(\SSS,\oml)/\T_{m'} \simeq \jsom\cap U_{m'}\]
split (with coefficients in an arbitrary field $k$). Using the fact that the contractible space of invariant compatible almost-complex structures decomposes as the disjoint union
\[\jsom=(\jsom\cap U_m)\sqcup (\jsom\cap U_{m'})\]
the rank of $H^i(\Symp^{S^1}_h(\SSS,\oml);k)$ can be computed inductively from Alex\-an\-der-Eells duality. We then compute the cohomology algebra and the Pontryagin algebra of the pushout
\[P = \pushout(\T_m\from S^1\to \T_{m'})\]
in the category of topological groups. We then show that the natural map
\[\Upsilon:P\to \Symp^{S^1}_h(\SSS,\oml)\]
is a homotopy equivalence in the category of topological groups. We further prove that $P$ is weakly homotopy equivalent, as a topological space, to the product $\Omega S^{3}\times S^{1}\times S^{1}\times S^{1}$.

\subsection{Homological injectivity}
We first show that the two inclusions $\T_{m} \into \Symp_h^{S^1}(\SSS,\oml)$ and $\T_{m'} \into \Symp_h^{S^1}(\SSS,\oml)$ induce injective maps in homology. As the argument does not depend on $m$, we shall only provide the details for the inclusion $\T_{m} \into \Symp_h^{S^1}(\SSS,\oml)$.\\

Fix a symplectomorphism $\phi_m: W_m\to (\SSS,\oml)$ compatible with the fibration structures. Let $\{*\}$ be the $S^1(1,b,m)$ fixed point $\left([0,1][0,0,1]\right)$ in $W_m$, and let $\mathcal{E}(W_m, *)$ denote the space of orientation preserving, pointed, homotopy self-equivalences of $(W_m,*)$. Similarly, define $\mathcal{E}(S^2, *)$ to be the space of all orientation preserving homotopy self-equivalences of the sphere preserving a base point $\{*\}$.
\\

We now observe that for the above two families of circle actions (i) and (ii), the same argument as in Lemma~\ref{lemma:SymphPreservesAnIsolatedFixedPoint} shows that any $\phi\in\Symp_h^{S^1}(\SSS,\oml)= \Symp_h^{S^1}(W_{m})$ fixes the base point $\{*\}$.\\

Now, recall that the zero section $s_0$ of $W_m$ is given by
\begin{align*}
    s_{0}: S^2 &\to W_m \\
    \left[z_{0}, z_{1}\right] &\mapsto\left(\left[z_{0}, z_{1}\right],[0,0,1]\right)
\end{align*} 
and the projection to the first factor is
\begin{align*}
    \pi_1: W_m &\to S^2 \\
    \left(\left[z_{0}, z_{1}\right],\left[w_{0}, w_{1}, w_{2}\right]\right) &\mapsto\left[z_{0}, z_{1}\right]
\end{align*}
We define a continuous map $h_1:\Symp_h^{S^1} (\SSS,\oml) \to \mathcal{E}\left(S^{2}, *\right)$ by setting
\begin{equation*}
  \begin{aligned}
h_1:\Symp_h^{S^1} (\SSS,\oml) &\to \mathcal{E}\left(S^{2}, *\right) \\
\psi &\mapsto \psi_1:= \pi_{1} \circ \psi \circ s_{0}
\end{aligned}
\end{equation*}
Similarly, using the inclusion of $S^{2}$ as the fiber 
\begin{align*}
    f: S^2 &\to W_m \\
    \left[z_{0}, z_{1}\right] &\mapsto\left([0,1],\left[ 0,z_{0}, z_{1}\right]\right)
\end{align*}
and the projection to the second factor $\pi_{2}: \SSS \to S^{2}$, we can define a map 
\begin{align*}
   h_2: \Symp_h^{S^1} (\SSS,\oml) &\to \mathcal{E}\left(S^{2}, *\right) \\
 \psi &\mapsto \psi_2:= \pi_{2} \circ \psi \circ f
\end{align*}

We thus get a continuous map 
\begin{align*}
    h: \Symp_h^{S^1} (\SSS,\oml) &\to \mathcal{E}(S^2, *) \times \mathcal{E}(S^2,*) \\
    \psi &\mapsto \left(h_1(\psi),h_2(\psi)\right) 
\end{align*}
\begin{lemma} \label{inj}
The inclusion $i_m:\T_m \hookrightarrow \Symp_h^{S^1} (\SSS,\oml) $ induces a map which is injective in homology with coefficients in any field $k$.
\end{lemma}
\begin{proof}
As $\T$ is connected, $i_m: H_0(\T_m; k) \to H_0(\Symp_h^{S^1}(\SSS,\oml); k)$ is injective. To show that the inclusion map induces an injection at the $H_1$ level, we consider the composition $\alpha: \T_m \to \mathcal{E}(S^2,*) \times \mathcal{E}(S^2,*)$ given by
\[
\begin{tikzcd}
    &\T_m \arrow[r,hookrightarrow]  &\Symp_h^{S^1} (\SSS,\oml)  \arrow[r,rightarrow,"h"]  &\mathcal{E}(S^2, *) \times \mathcal{E}(S^2, *) 
\end{tikzcd}
\]
and show that $\alpha$ induces a map which is injective in homology. 

We claim that $H_1(\mathcal{E}(S^2,*);\Z)\simeq \Z$. Indeed, the standard action of $\SO(3)$ on $S^2$ gives rise to a diagram of fibrations
\[
\begin{tikzcd} 
        & \mathcal{E}(S^2,*)\arrow[r] &\mathcal{E}(S^2) \arrow[r,twoheadrightarrow,"\ev"] &S^2 \\
     &S^1= \SO(2) \arrow[u,hookrightarrow] \arrow[r] &\SO(3) \arrow[u,hookrightarrow] \arrow[r,"\ev",twoheadrightarrow] &S^2 \arrow[u,equal]
\end{tikzcd}
\]
where the maps $\ev$ are evaluations at the base point $\{*\}$. This induces a long exact ladder of homotopy groups
\[
\begin{tikzcd}[cramped]
     & \arrow[r] &\cancelto{\Z}{\pi_2(S^2)} \arrow[r] &\pi_1(\mathcal{E}(S^2,*)) \arrow[r] &\pi_1(SO(3)) \times \cancelto{0}\pi_1(\Omega) \arrow[r] &\cancelto{0}{\pi_1(S^2)} \\
     & \arrow[r] &\Z \arrow[u,equal] \arrow[r]&\cancelto{\Z}{\pi_1(S^1)} \arrow[r] \arrow[u,"\beta"] &{\pi_1(SO(3))}\arrow[u] \arrow[r] &\cancelto{0}{\pi_1(S^2)} \arrow[u,equal]
\end{tikzcd}
\]
where we have used the fact, proven by Hansen in~\cite{Hans}, that $\mathcal{E}(S^2) \simeq SO(3) \times \widetilde{\Omega^2}$, where $\widetilde{\Omega^2}$ denotes the universal covering space for the connected component of the double loop space of $S^2$ containing the constant based map, and where the $\SO(3)$ component is just the inclusion. Consequently, $\pi_1(\widetilde{\Omega^2})= 0 $ and the map $\pi_1(SO(3)) \to \pi_1(SO(3)) \times \pi_1(\widetilde{\Omega^2})$ is an isomorphism. From the commutativity of the middle square, it follows that $\beta:\pi_1(S^1) \to \pi_1(\mathcal{E}(S^2, *))$ is also an isomorphism. As the spaces we consider are topological groups, $\pi_1$ is abelian and hence $\pi_1 = H_1$, proving the claim.\\

Now, the classes $a$, $b$, of the subcircles $(0,1)$ and $(1,0)$ form a basis of $H_1(\T_m; k)$. We claim that $\alpha_*[0,1]$ and $\alpha_*[1,0]$ generate a subgroup of rank 2. To see this, let write $\alpha_*^1$ and $\alpha_*^2$ for the components of $\alpha_*$. Then, $\alpha^1_*(0,1) = 0$ as the circle $(0,1)$ fixes the zero section $\left([x_1,x_2],[0,0,1]\right) \subset W_m$ pointwise, while $\alpha^2_*[0,1] \neq 0$ by the reasoning in the previous paragraph. Similarly, $\alpha^1_*[1,0] \neq 0$ and $\alpha^2_*[1,0]= 0$, proving our claim. We conclude that $\alpha$ is injective on $H_1(\T_m; k)$. 
\\

Finally, to show that $i_*$ is injective on $H_2(\T_m;k)$, we will prove the dual statement, namely, that the map $i^*:H^2(\Symp^{S^1}_h(\SSS,\oml);k) \to H^2(\T_m;k)$ is surjective. A generator of $H^2(\T_m;k) \cong k$ is given by $a \cup b$. Because $i_*$ is injective at the $H_1$ level, $i^*: H^1(\Symp^{S^1}_h(\SSS,\oml);k) \to H^1(\T;k)$ is surjective, hence there exists elements $a^\prime$, $b^\prime \in H^1(\Symp^{S^1}_h(\SSS,\oml);k)$ such that $i^*(a^\prime) = a$ and $i^*(b^\prime) = b$. Since $i^*(a^\prime) \cup i^*(b^\prime) = a \cup b$, it follows that $i^*:H^2(\Symp^{S^1}_h(\SSS,\oml);k) \to H^2(\T_m;k)$ is surjective.
\end{proof}

\subsection[Cohomology module of the centralizer]{Cohomology module of the centralizer of $S^1(\pm1,b;m)$}\label{subsection:CohomologyModule}
We are now ready to compute the cohomology module of the centralizer of $S^1(\pm1,b;m)$ with coefficients in in a field $k$. By duality, this is equivalent to determining the homology module.

Recall that the contractible space of invariant compatible almost-complex structures $\jsom$ decomposes as the disjoint union
\[\jsom=(\jsom\cap U_m)\sqcup (\jsom\cap U_{m'})=:U_{m}^{S^1} \sqcup U_{m'}^{S^1}\]
where, for convenience, we set $U_m^{S^1}=\jsom\cap U_m$ and $U_{m'}^{S^1}=\jsom\cap U_{m'}$. We will show in Chapter~\ref{Chapter-codimension} the following two important facts:
\begin{itemize}
\item the strata $U_{m}^{S^1}$ and $U_{m'}^{S^1}$ are submanifolds of $\jsom$ (see Corollary~\ref{cor:StrataAreSubmanifolds}), and 
\item the stratum $U_{m}^{S^1}$ is open in $\jsom$, while $U_{m'}^{S^1}$ is of codimension $2$ (see Theorem~\ref{codimension_calc}). 
\end{itemize}
In particular, it follows that $U_{m}^{S^1}=\jsom-U_{m'}^{S^1}$ is connected. As explained in Appendix~\ref{Appendix-Alexander-Eells}, Proposition~\ref{prop:AlexanderEellsGeometric}, the Alexander-Eells duality induces an isomorphism of homology groups
\begin{equation}\label{eq:AlexanderEellsIsomorphismHomology}
\lambda_{*}:H_{p}(U_{m'}^{S^1};k)\to H_{p+1}(U_{m}^{S^1};k)
\end{equation}

Now recall that we also have fibrations
\begin{align}\label{eq:TheTwoMainFibrations}
\T_m\to\Symp^{S^1}_h(\SSS,\oml)\xrightarrow{p_m} \Symp^{S^1}_h(\SSS,\oml)/\T_m \simeq U_{m}^{S^1}\\
\T_{m'}\to\Symp^{S^1}_h(\SSS,\oml)\xrightarrow{p_{m'}}\Symp^{S^1}_h(\SSS,\oml)/\T_{m'} \simeq U_{m'}^{S^1}\notag
\end{align}
From the first fibration, the connectedness of the open stratum $U_{m}^{S^1}$ implies that the group $\Symp^{S^1}_h(\SSS,\oml)$ is connected. In turns, the second fibration implies that the codimension 2 stratum $U_{m'}^{S^1}$ is also connected.
Because the two inclusions \[\T_{m} \into \Symp_h^{S^1}(\SSS,\oml)\quad \text{~and~}\quad \T_{m'} \into \Symp_h^{S^1}(\SSS,\oml)\]
induce surjective maps in cohomology, the Leray-Hirsch theorem implies that the cohomology module of $\Symp^{S^1}_h(\SSS,\oml)$ splits as
\begin{align}\label{eq:SplittingCohomology}
H^*(\Symp^{S^1}_h(\SSS,\oml),k) \cong H^*(U_{m}^{S^1};k) \otimes H^*(\T_m;k)\\
H^*(\Symp^{S^1}_h(\SSS,\oml),k) \cong H^*(U_{m'}^{S^1};k) \otimes H^*(\T_{m'};k)\notag
\end{align}
By duality, we have corresponding splittings in homology, namely,
\begin{align}\label{eq:SplittingHomology}
H_*(\Symp^{S^1}_h(\SSS,\oml),k) \cong H_*(U_{m}^{S^1};k) \otimes H_*(\T_m;k)\\
H_*(\Symp^{S^1}_h(\SSS,\oml),k) \cong H_*(U_{m'}^{S^1};k) \otimes H_*(\T_{m'};k)\notag
\end{align}
It follows that 
\[H_{p}(U_{m};k)\simeq H_{p}(U_{m'};k)\text{~for all~}p\geq 0\]
Together with the Alexander-Eells isomorphism~(\ref{eq:AlexanderEellsIsomorphismHomology}) and the connectedness of $U_{m'}$, this implies that
\[H_{p}(U_{m};k)\simeq k\text{~for all~}p\geq 0\]
Using the splitting~\ref{eq:SplittingHomology} and dualizing, we can finally compute the cohomology module of $\Symp^{S^1}_h(\SSS,\oml)$.

\begin{thm}{\label{cohom}} Consider any of the following circle actions:
\begin{itemize}
\item (i) $a=1$, $b \neq \{0, m\}$,  and $2\lambda > |2b-m|$; or
\item (ii) $a=-1$, $b \neq \{0, -m\}$, and $2 \lambda > |2b+m|$.
\end{itemize} Then, the cohomology groups of the symplectic centralizer are
$$H^p\left(\Symp^{S^1}(\SSS,\oml); k\right) \simeq \begin{cases}
k^4 ~~p \geq 2\\
k^3 ~~p =1 \\
k ~~ p=0\\
\end{cases}$$
for any field $k$. In particular, the topological group $\Symp^{S^1}(\SSS,\oml)$ is of finite type.
\end{thm}

\subsection{The homotopy pushout \texorpdfstring{$T_m\from S^1(\pm1,b;m)\to T_{m'}$}{the two inclusions}}\label{section:full_homotopy}
As explained in Corollary~\ref{cor:CircleExtensionsWith_a=1} and Corollary~\ref{cor:CircleExtensionsWith_a=-1}, the circle actions $S^1(\pm1,b;m)$ we are considering in this section extend to exactly two toric actions $\T_m$ and $\T_{m'}$. Geometrically, this means that the two tori $\T_m$ and $\T_{m'}$ intersect in $\Symp^{S^1}_h(\SSS,\oml)$ along the circle $S^1(\pm1,b;m)$ and, in particular, that we have two inclusions of Lie groups
\[
\begin{tikzcd}[sep=small]
S^{1} \arrow{r}{(1,b')} \arrow[swap]{d}{(1,b)} & T^{2}_{m'} \\
T^{2}_{m}  & 
\end{tikzcd}
\]
In this section we consider the homotopy pushout of these two inclusions, namely,
\[P:=\pushout(T_m\from S^1\to T_{m'})\]
This pushout is to be understood in the category of topological groups. As we will show later, the topological group $P$ turns out to be a model for the homotopy type of the centralizer $\Symp^{S^1}_h(\SSS,\oml)$.

\subsubsection*{The Pontryagin algebra of the pushout}
In what follows, all k algebras are graded, and the commutator of two elements is given by
\[[a,b] = ab - (-1)^{|a|\cdot|b|}ba\]
For any field $k$, and for any abelian group $A$, the Pontryagin algebra $H_{*}(A;k)$ is isomorphic to the cohomology algebra $H^{*}(A;k)$. It follows that $H_{*}(S^{1})$ is isomorphic to $\Lambda(t)$, where $t$ is of degree $1$. Similarly, the Pontryagin algebra $H_{*}(T^{2};k)$ is isomorphic to the to an exterior algebra $\Lambda(z_{1},z_{2})$ generated by two elements of degree one. The pushout diagram of topological groups 
\[
\begin{tikzcd}
S^{1} \arrow{r}{(1,b')} \arrow[swap]{d}{(1,b)} & T^{2}_{m'} \arrow{d}{} \\
T^{2}_{m} \arrow{r}{} & P
\end{tikzcd}
\]
is homologically free (see Definition~3.1 in~\cite{AG}). As before, $P$ denotes the pushout in the category of topological groups. By Theorem~3.8 of~\cite{AG}, the Pontryagin algebra of $P$ is the pushout of $k$ algebras
\[
\begin{tikzcd}
H_{*}(S^{1};k) \arrow{r}{H(1,b')} \arrow[swap]{d}{H(1,b)} & H_{*}(T^{2}_{m'};k) \arrow{d}{} \\
H_{*}(T^{2}_{m};k) \arrow{r}{} & H_{*}(P;k)
\end{tikzcd}
\]
which is isomorphic to
\[
\begin{tikzcd}
\Lambda(t) \arrow{r}{(1,b')} \arrow[swap]{d}{(1,b)} & \Lambda(y_{1},y_{2}) \arrow{d}{} \\
\Lambda(x_{1},x_{2}) \arrow{r}{} & P^{alg}_{*}
\end{tikzcd}
\]
where $P^{alg}_{*}\simeq H_{*}(P;k)$. By the description of the pushout of $k$ algebras as amalgamated products (see \cite{AG} for more details), the $k$ algebra $P^{alg}_{*}$ can be identified with equivalence classes of finite linear combinations of words in the letters $\{x_{1},x_{2},y_{1},y_{2}\}$ under the relations $x_{i}x_{i}=0$, $y_{i}y_{i}=0$, $[x_{1},x_{2}]=0$, $[y_{1},y_{2}]=0$, and $x_{1}+bx_{2}=y_{1}+b'y_{2}$.  From the last equality, we can write $y_{1}=(x_{1}+bx_{2})-b'y_{2}$, which means that we can choose, as generators, the elements \[\{t=x_{1}+bx_{2}, ~x_{2}, ~y_{2}\}\]
with the relations $t^{2}=x_{2}^{2}=y_{2}^{2}=0$, $[t,x_{2}]=[t,y_{2}]=0$. The remaining commutator $w=[x_{2},y_{2}]$ is nonzero and commutes with $t$, $x_{2}$ and $y_{2}$. It follows that any word in $t,x_{2},y_{2}$ is equivalent to a linear combination of words of the form
\[w^{\alpha}x_{2}^{\beta}y_{2}^{\gamma}t^{\delta}\]
with $\alpha\in\N\cup\{0\}$, and $\beta,\gamma,\delta\in\{0,1\}$. Hence, there is an isomorphism of graded algebras
\[P^{alg}_{*}\cong \frac{F(x_{2},y_{2})}{\langle x_{2}^{2},y_{2}^{2}\rangle}\otimes \Lambda(t)\]
where $F(x_{2},y_{2})$ denotes the free graded algebra over $k$ generated by the elements $x_{2}$ and $y_{2}$, and where $x_{2},y_{2},t$ are of degree one. In particular,
\[P^{alg}_{n}\simeq
\begin{cases}
k & n=0\\
k^{3} & n=1\\
k^{4} & n\geq 2
\end{cases}\]
and the words $w^{\alpha}x_{2}^{\beta}y_{2}^{\gamma}t^{\delta}$ form an additive basis of the homology module $P^{alg}_{*}$. By duality, the cohomology modules $P^{alg,*}$ are $P^{alg,0} \simeq k$, $P^{alg,1}\simeq k^3$, and $P^{alg,n}\simeq k^4$ for all $n\geq 2$.\\

The algebra structure of $P^{alg,*}$ can be determined using the Hopf-Borel theorem. Let $\hat t$, $\hat x_{2}$, and $\hat y_{2}$ be the duals of the generators of degree $1$, and let $\hat w$ be the dual of the generator $w=[x_2,y_2]$ of degree~$2$.
\begin{thm}[Hopf-Borel, see~\cite{McCleary} Theorem 6.36]
Let k be a field of characteristic p where p may be zero or a prime. A connected Hopf algebra $H$ over $k$ is said to be monogenic if $H$ is generated as an algebra by 1 and one homogeneous element $x$ of degree strictly greater than 0. If $H$ is a monogenic Hopf algebra, then
\begin{enumerate}
    \item if $p \neq 2$ and degree $x$ is odd, then $H \cong \Lambda(x)$,
    \item if $p \neq 2$ and degree $x$ is even, then $H \cong k[x] /\left\langle x^{s}\right\rangle$ where $s$ is a power of p or is infinite i.e $H \cong k[x]$,
    \item if $p=2$, then $H \cong k[x] /\left\langle x^{s}\right\rangle$ where $s$ is a power of 2 or is infinite.
\end{enumerate}
Moreover, any associative, graded commutative Hopf algebra of finite type over $k$ is a tensor product of monogenic Hopf algebras.
\end{thm}
From the discussion above, the Hopf algebra $P^{alg,*}$ satisfies the conditions of the Hopf-Borel theorem. For a field $k$ of characteristic $p$ different from $2$, including $p=0$, $P^{alg,*}$ contains a subalgebra of the form 
\[A^*=\Lambda(\hat t,\hat x_2,\hat y_2)\otimes k[\hat w]/\langle \hat w^s\rangle\]
where $s$ is a power of $p$ or is infinite. Suppose $s=p^n\geq 3$ is finite. Then, the rank of $A^i$ would coincide with the rank of $P^{alg,i}$ up to degree $i=2s-1$, and we would have $A^i=0$ for $i\geq 2s$. Therefore, we would need $4$ more generators of degree $2s$ to account for the rank of $P^{alg,2s}$, and their pairwise products would imply that $\rk P^{alg,4s}>4$. This contradiction shows that $s$ must be infinite and that the rank of $A^i$ equals the rank of $P^{alg,i}$ for all $i\geq 0$. Consequently, for a field $k$ of characteristic $p\neq 2$, the $k$-algebra $P^{alg,*}$ is isomorphic to
\[P^{alg,*}\cong \Lambda(\hat t,\hat x_{2},\hat y_{2})\otimes S(\hat w)\]
In characteristic $p=2$, $P^{alg,*}$ is the tensor product of truncated polynomial algebras $k[z_i]/z^{s_i}_{i}$ where $s_i$ is a power of $2$. It contains a subalgebra of the form
\[A^*=k[\hat t,\hat x_2,\hat y_2]/ \langle \hat t^2,\hat x_2^2,\hat y_2^2 \rangle \otimes k[\hat w]/\langle \hat w^s\rangle\]
Again, assuming $s$ is finite forces the existence of $4$ new generators in degree $2s$ whose products would yield too many generators in degree $4s$. Therefore, in characteristic $p=2$, the cohomology algebra of $P$ is isomorphic to
\[P^{alg,*}\cong k[\hat t,\hat x_{2},\hat y_{2}]/ \langle \hat t^2,\hat x_2^2,\hat y_2^2 \rangle \otimes k[\hat w]\]
In characteristic zero, the computation of the cohomology ring yields the minimal model of $H^{*}(P)\otimes\Q$. As $P$ is a H-space, it is a nilpotent space (see Exercise~1.13 in \cite{dgcRationalHomotopy}), so that the main theorem of dgc rational homotopy theory applies (see \cite{dgcRationalHomotopy}, Theorem~2.50) namely, the dimension $\pi_{p}(P)\otimes\Q$ for $p\geq 2$ is equal to the number of generators of degree $p$ in the minimal model. For $p=1$, as $P$ is a topological group, the dimension of $\pi_1(P) \otimes \Q$ is the same as the rank of $H_1(P,\Q)$. Consequently,
\[\pi_{p}(P)\otimes\Q\simeq
\begin{cases}
\Q & p=0\\
\Q^{3} & p=1\\
\Q & p= 2\\
0 & p\geq 3
\end{cases}\]

\subsubsection*{The homotopy type of $P$}
We want to better understand the homotopy type of the space $P$. To this end, consider the embeddings
\begin{align}
f_{m}:T^{2}_{m}&\to S^{1}\times S^{1}\times S^{1}\\
(x_{1},x_{2})&\mapsto (x_{1},x_{2},b'x_{1})\notag\\
&\notag\\
f_{m'}:T^{2}_{m'}&\to S^{1}\times S^{1}\times S^{1}\\
(y_{1},y_{2})&\mapsto (y_{1},by_{1},y_{2})\notag
\end{align}
The universal property of pushouts implies that there is a unique map $f_{P}:P\to S^{1}\times S^{1}\times S^{1}$ making the following diagram commutative
\[
\begin{tikzcd}
BS^{1} \arrow{r}{B(1,b')} \arrow[swap]{d}{B(1,b)} & BT^{2}_{m'} \arrow{d}{} \arrow[bend left=10]{ddr}{Bf_{m'}} & \\
BT^{2}_{m} \arrow{r}{}\arrow[bend right=10,swap]{drr}{Bf_{m}} & BP \arrow[dotted]{rd}{Bf_{P}} & \\
 & & BS^{1}\times BS^{1}\times BS^{1}
\end{tikzcd}
\]
By Theorem~3.9 of~\cite{AG}, the homotopy fiber of $Bf_{P}$ is the pushout of the homotopy fibers of the other maps in the diagram. To determine this fiber, we first replace the maps in the diagram of groups by homotopy equivalent fibrations
\[
\begin{tikzcd}[column sep=huge]
\Z\ar[swap]{d}{(1,1,a_{1})} & \Z\times\Z \ar{d}{(1,a_{1},a_{2})}\ar[swap]{l}{a_{2}} \ar{r}{a_{1}}& \Z \ar{d}{(1,1,a_{1})}\\
T^{2}_{m}\times\R \ar[swap]{d}{(a_{1},a_{2},b'a_{1}e({a_{3}}))} & S^{1}\times\R\times\R \ar[swap]{l}{(a_{1},ba_{1}e({a_{2}}),a_{3})} \ar{r}{(a_{1},b'a_{1}e({a_{3}}),a_{2})} \ar{d}{(a_{1},ba_{1}e({a_{2}}),b'a_{1} e({a_{3}}))}& T^{2}_{m'}\times\R\ar{d}{(a_{1},ba_{1}e({a_{3}}),a_{2})}\\
S^{1}\times S^{1}\times S^{1}\ar{r}{=} & S^{1}\times S^{1}\times S^{1} &S^{1}\times S^{1}\times S^{1}\ar[swap]{l}{=}
\end{tikzcd}
\]
where $a_{i}$ denote the $i^{\text{th}}$ coordinate function and $e(a_j) = e^{2\pi i a_j}$. Applying the classifying space functor, this gives
\[
\begin{tikzcd}[column sep=normal]
S^{1}\ar[swap]{d}{} & S^{1}\times S^{1}\ar{d}{}\ar[swap]{l}{\text{pr}_{2}} \ar{r}{\text{pr}_{1}}& S^{1} \ar{d}{}\\
BT^{2}_{m} \ar[swap]{d}{} & BS^{1} \ar[swap]{l}{} \ar{r}{} \ar{d}{}& BT^{2}_{m'}\ar{d}{}\\
BS^{1}\times BS^{1}\times BS^{1}\ar{r}{=} & BS^{1}\times BS^{1}\times BS^{1} &BS^{1}\times BS^{1}\times BS^{1}\ar[swap]{l}{=}
\end{tikzcd}
\]
which shows that the homotopy fiber of the canonical map $BP\to BS^{1}\times BS^{1}\times BS^{1}$ is homotopy equivalent to
\[\hocolim\{S^{1}\xleftarrow{\text{pr}_{2}}S^{1}\times S^{1}\xrightarrow{\text{pr}_{1}}S^{1}\}\simeq S^{1}*S^{1}\simeq S^{3}\]
Consequently, $BP$ is the total space of a fibration
\[S^{3}\to BP\to BS^{1}\times BS^{1}\times BS^{1}\]
that, after looping, becomes
\[
\begin{tikzcd}[column sep=normal]
 & T^{2}_{m'}\ar[swap]{d}{j_{m'}} \ar{rd}{f_{m'}=(a_{1},ba_{1},a_{2})}& \\
\Omega S^{3}\ar{r} & P\ar{r}{f_{P}} & S^{1}\times S^{1}\times S^{1}\\
 & T^{2}_{m}\ar{u}{j_{m}} \ar[swap]{ru}{f_{m}=(a_{1},a_{2},b'a_{1})}& 
\end{tikzcd}
\]
The map $f_{P}$ admits a section given by 
\[s(a_{1},a_{2},a_{3})= j_{m'}(a_1, {b'}^{-1}a_3)j_{m}(1,{b}^{-1}a_1^{-1}a_{2})\]

It follows that, as a space, $P$ is weakly homotopically equivalent to the product
\[P\simeq \Omega S^{3}\times S^{1}\times S^{1}\times S^{1}\]
which is consistent with the algebraic computations of the previous section.

\subsection{Homotopy type of centralizers for the circle actions \texorpdfstring{$S^1(\pm1,b;m)$}{S1(1,b;m)}}

We are now able to determine the homotopy type of the group $\Symp_h^{S^1}(\SSS,\oml)$ for the circle actions
\begin{itemize}
\item $S^1(1,b,m)$ when $2\lambda > |2b-m|$, and
\item $S^1(-1,b,m)$ when $2\lambda > |2b+m|$.
\end{itemize}
Since the arguments are identical in the two cases, we will only discuss the first one. Again, in order to keep the notation simple, we write $\T_m$ and $\T_{m'}$ for the two tori the circle extends to, assuming $m'>m$, and we write $(1,b):S^1\to\T_m$ and  $(1,b'):S^1\to\T_{m'}$ for the two inclusions.\\
 
From the universal property of pushouts, there is a canonical map 
\[\Upsilon:P^{alg}_{*}\to H_{*}(\Symp_h^{S^1}(\SSS,\oml);k)\]
making the following diagram commutative
\[
\begin{tikzcd}
\Lambda(t) \arrow{r}{(1,b')} \arrow[swap]{d}{(1,b)} & \Lambda(y_{1},y_{2}) \arrow{d}{} \arrow[bend left=10]{ddr}{i_{m'}} & \\
\Lambda(x_{1},x_{2}) \arrow{r}{}\arrow[bend right=10,swap]{drr}{i_{m}} & P^{alg}_{*} \arrow[dotted]{rd}{\Upsilon} & \\
 & & H_{*}(\Symp_h^{S^1}(\SSS,\oml);k)
\end{tikzcd}
\]

\begin{prop}
For every field $k$, the map $\Upsilon:P^{alg}_{*}\to H_{*}(\Symp_h^{S^1}(\SSS,\oml);k)$ is an isomorphism of $k$-algebras.
\end{prop}
\begin{proof}
By definition, the map $\Upsilon$ is an homomorphism of $k$-algebras. Since $P^{alg}_{i}\cong H_{i}(\Symp_h^{S^1}(\SSS,\oml);k)$ for each $i$, it is sufficient to show that $\Upsilon$ is surjective.\\

Let $R$ be the image of $\Upsilon$. Since the maps $i_{m}$ and $i_{m'}$ are injective, $R$ is the subring generated by the classes $t,x_{2},y_{2}$ viewed as elements in $H_{*}(\Symp_h^{S^1}(\SSS,\oml);k)$. Consider the two fibrations induced by the action maps
\begin{align*}
\T_m\to\Symp^{S^1}_h(\SSS,\oml)\xrightarrow{p_m} \Symp^{S^1}_h(\SSS,\oml)/\T_m \simeq U_{m}^{S^1}\\
\T_{m'}\to\Symp^{S^1}_h(\SSS,\oml)\xrightarrow{p_{m'}}\Symp^{S^1}_h(\SSS,\oml)/\T_{m'} \simeq U_{m'}^{S^1}
\end{align*}
Observe that $p_{m}(t)=0$, $p_{m}(x_{2})=0$, $p_{m'}(t)=0$, and $p_{m'}(y_{2})=0$. Now suppose there is an element $z\in H_{*}(\Symp_h^{S^1}(\SSS,\oml);k)$, not in $R$, and of minimal degree $d$. Since
\begin{multline}
H_{d}(\Symp_h^{S^1}(\SSS,\oml);k)\cong\\
H_{d}(U_{m}^{S^1};k)\otimes H_{0}(T^{2}_{m};k)
~\oplus~ H_{d-1}(U_{m}^{S^1};k)\otimes H_{1}(T^{2}_{m};k)
~\oplus~ H_{d-2}(U_{m}^{S^1};k)\otimes H_{2}(T^{2}_{m};k)
\end{multline}
we would have a decomposition
\[z=c_{1}\otimes\mathbf{1}~\oplus~ c_{t}\otimes t~\oplus~ c_{x_{2}}\otimes x_{2}+c_{T}\otimes [T^{2}_{m}]\]
with at least one coefficient $c_{j}$ which is not a polynomial in the classes $p_{m}(w)$ and $p_{m}(y_{2})$. Let $c_{\ell}$ be such coefficient of minimal degree $d-2\leq\ell\leq d$. The inverse of the Alexander-Eells isomorphism of Proposition~\ref{prop:AlexanderEellsGeometric}
\[\lambda_{*}^{-1}:H_{p+1}(U_{m}^{S^1})\to H_{p}(U_{m'}^{S^1})\]
would map $c_{\ell}$ to a class $c_{\ell-1}'\in H_{\ell-1}(U_{m'}^{S^1};k)$. This class could not be a polynomial in $p_{m'}(w)$ and $p_{m'}(x_{2})$ since, otherwise,  
\[c_{\ell} = \lambda_{*}(c_{\ell-1}')=p_{m}\big([y_{2}\otimes c_{\ell-1}']\big)\]
would be a polynomial in the classes $p_{m}(w)$ and $p_{m}(y_{2})$. 
In turn, this class $c_{\ell-1}'$ would have to be the image of some element in $H_{\ell-1}(\Symp_h^{S^1}(\SSS,\oml);k)$ not in $R$, contradicting the minimality of~$z$.
\end{proof}

\begin{cor}
The map $\Upsilon:P^{alg}_{*}\to H_{*}(\Symp_h^{S^1}(\SSS,\oml);\Z)$ is an isomorphism of Pontryagin algebras over the ring of integers.
\end{cor}
\begin{proof}
This follows from the well known fact that a map induces isomorphisms on homology with $\Z$ coefficients iff it induces isomorphisms on homology with $\Q$ and $\Z_{p}$ coefficients for all primes $p$, see~\cite{Ha}, Corollary~3A.7~(b).
\end{proof}

\begin{thm}\label{full_homo}
The map $\Upsilon:P\to\Symp_h^{S^1}(\SSS,\oml)$ is an homotopy equivalence.
\end{thm}
\begin{proof}
The map $\Upsilon$ is a homology equivalence on integral homology. Because $P$ and $\Symp_h^{S^1}(\SSS,\oml)$ are topological groups, it follows that it is a weak equivalence, see~\cite{Dror-WhiteheadTheorem}, Example~4.2. Because both spaces are homotopy equivalent to CW-complexes, this weak equivalence is a homotopy equivalence (See~\cite{Ha}, Proposition~4.74).
\end{proof}

\section{Centralizers of Hamiltonian \texorpdfstring{$S^1$}{circle} actions on \texorpdfstring{$\SSS$}{the product}}

We summarise all the results we have obtained in this chapter in the following theorem. 
\begin{thm} \label{circle}
Consider any Hamiltonian circle action $S^1(a,b;m)$ on $(S^2 \times S^2, \oml)$. The homotopy type of the symplectic centralizer $\Symp^{S^1}(S^2 \times S^2,\oml)$ is given in table~\ref{table:CentralizersAllCiclesActionsOnSSS}.
\end{thm}

{\centering
\begin{table}[h!]
\begin{tabular}{|p{37mm}|p{47mm}|p{32mm}|}
\hline
Values of $(a,b ;m)$,\newline $m$ even& Conditions on $\lambda$ &Homotopy type of\newline $\Symp^{S^1}(S^2 \times S^2,\oml)$\\
\hline
\hline
{$(0,\pm 1;m)$, $m\neq 0$}  &$2\lambda > m$ & $S^1 \times SO(3)$ \\
\hline
$(0,\pm 1;0)$ or $(\pm 1,0;0)$ &$\lambda \geq 1$  &$S^1 \times SO(3)$ \\\cline{2-3}
\hline
$(\pm 1,\pm1;0)$&$\lambda = 1$ & $\T \times \mathbb{Z}_2$ \\
\hline
$(\pm 1,0;m)$, $m\neq0$ &$2\lambda >m$ &$\T$\\
\hline
$(\pm1,\pm m;m)$, $m \neq 0$ &$2\lambda>m$  & $\T$\\
\hline
\multirow{2}{10em}{$(1,b;m)$, $b \neq \{ m,0\}$} 
&$2\lambda>m$ and $|2b-m| \geq2 \lambda \geq 1$ &$\T$ \\\cline{2-3}
&$2\lambda > m$ and $2 \lambda >|2b-m| \geq 0$  &$\Omega S^3 \times S^1 \times S^1 \times S^1$ \\
\hline
\multirow{2}{10em}{$(-1,b;m)$, $b\neq\{-m,0\}$} 
&$2\lambda > m$ and $|2b+m| \geq2 \lambda \geq 1$ &$\T$ \\\cline{2-3}
&$2\lambda > m$ and $2 \lambda >|2b+m| \geq 0$  &$\Omega S^3 \times S^1 \times S^1 \times S^1$ \\
\hline
All other values $(a,b;m)$ &$2\lambda > m$  &$\T$ \\
\hline
\end{tabular}
\caption{Centralizers of circle actions on $(\SSS,\oml)$}
\label{table:CentralizersAllCiclesActionsOnSSS}
\end{table}
}
\begin{remark}
Recall that in the above table for the  $S^1$ action $S^1(a,b;m)$ for $m \neq 0$ on $(\SSS,\oml)$, to be Hamiltonian we need the condition $\lambda > \frac{m}{2}$.
\end{remark}

\chapter{Partition of the space of invariant almost-complex structures}\label{Chapter-codimension}
In the previous section, we calculated the homotopy type of the group of $S^1(\pm1,b;m)$ equivariant symplectomorphisms assuming that the codimension of the the invariant strata $\jsom \cap U_{m^\prime}$ in $\jsom$ was 2. In this section, we use deformation theory to show that the invariant strata $\jsom \cap U_{m^\prime}$ is a submanifold of $\jsom$ and to characterize its normal bundle. We then show that its codimension is indeed $2$. \\

We adapt the techniques of~\cite{AGK} to the equivariant setting. Consider a K\"ahler 4-manifold $(M,\om,J)$ equipped with an $S^1$ action on $(M,\om,J)$ for which $\om$ and $J$ are invariant. The holomorphic $S^1$ action on the base manifold $M$ induces a natural action on the various tensor spaces such as $T^{1,0}M$ or $\Omega^{0,k}_{J}(M, TM)$. As usual, we use an $S^1$ exponent like in $(T^{1,0}M)^{S^1}$ or $\Omega^{0,k}_{J}(M, TM)^{S^1}$ to denote the $S^1$ invariant elements of tensor spaces.

\section{Space of invariant complex structures}
Let $\J_l$ be the space of almost complex structures of regularity $C^l$ on $M$, endowed with $C^l$ topology. Being a space of sections, $\J_l$ is a smooth Banach manifold. An explicit atlas can be constructed using the Cayley transform, see for instance~\cite{Smolentsev}. Given $J\in\J_l$, let $\Omega^{0,1}_{J,l}(M,TM) \subset \End_l(T M)$ be the space of  endomorphisms of the tangent bundle of regularity $C^l$ that anticommute with $J$, that is,
$$
\Omega^{0,1}_{J,l}(M,TM)=\left\{A  \in \End_l(T M) \mid A J+J A=0\right\} 
$$
The map $\phi_{J}: \Omega^{0,1}_{J,l}(M,TM) \rightarrow \J_l$ given by
$$
\phi_{J}(A)= J e^{A}
$$
is a local diffeomorphism sending $C^k$ endomorphisms ($k \geq l$) to $C^k$ almost complex structures. If $J$ is $S^1$ invariant, then $\phi$ gives a bijection between invariant endomorphisms near $0$ in $\Omega^{0,1}_{J,l}(M,TM)$ and invariant almost complex structures in a neighborhood of $J$. This shows that the space $\J^{S^1}_l$ of invariant almost complex structures is a Banach submanifold of $\J_l$ whose tangent space $T_J\J^{S^1}_l$ at $J$ is naturally identified with the linear subspace $\Omega^{0,1}_{J,l}(M,TM)^{S^1}$ of $S^1$ invariant endomorphisms of the tangent bundle of regularity $C^l$ that anticommute with $J$.\\

Let $I^{S^1}_l$ denote the space of invariant and integrable almost complex structures of $M$ with regularity $C^l$. We now show  that $I^{S^1}_l$is a Banach submanifold of $\J^{S^1}_l$. To this end, let  $N_J(X,Y) = [X,Y] + J\left([JX,Y] + [X,JY]\right) - [JX,JY]$ denote the Nijenhuis tensor with respect to $J$. By the Newlander-Nirenberg theorem, we know that $J \in \J^{S^1}_l$ is integrable iff $N(J)=0$.
\\

Consider the vector bundle $\Omega^{0,2}_{l-1}(M,TM)^{S^1}$ over $\mathcal{J}^{S^1}_l$ whose fibre over $J$ is the space $\Omega^{0,2}_{J,l-1}(M,TM)^{S^1}$ of $S^1$-invariant $(0,2)$ forms of regularity $C^{l-1}$ with values in the holomorphic tangent bundle $TM$.  
The Nijenhuis tensor can be interpreted as a section $N:\mathcal{J}_l\to\Omega^{0,2}_{l-1}(M,TM)$. This section is equivariant since, for all $g \in S^1$,
\begin{align*}
\begin{split}
g \cdot N_J(X,Y) &= g \cdot [X,Y] + g \cdot J\left([JX,Y] + [X,JY]\right) - g \cdot [JX,JY] \\ 
&= g_* [g_*^{-1}X,g_*^{-1} Y]  + g_*J\left([Jg_*^{-1}X,g_*^{-1}Y] \right.\\
&\phantom{+++++++++}+\left.[g_*^{-1}X,Jg_*^{-1}Y]\right) - g_* \cdot [Jg_*^{-1}X,Jg_*^{-1}Y]  \\ 
&= [X,Y] + J\left([JX,Y] + [X,JY]\right) - [JX,JY] 
\end{split}
\end{align*}
where the last equality follows from the facts that $g_*[X,Y] = [g_*X,g_*Y]$ and that $J$ is invariant. In particular, $N$ takes invariant tensors to invariant tensors, that is,
\begin{align*}
    N: \J^{S^1}_l &\to \Omega^{0,2}_{l-1}(M,TM)^{S^1} \\
      J &\mapsto N_J
\end{align*}
To show that $I^{S^1}_l$ is a Banach submanifold, it suffices to show that Nijenhuis tensor intersects the 0-section of the bundle transversally. This is equivalent to showing that, for an integrable $J$, the projection of the derivative to the vertical tangent bundle is surjective. We denote this projection of the derivative of $N$ to the vertical tangent bundle as $\nabla N$. A priori, $\nabla N$ depends on a choice of connection on $\Omega^{0,2}_{l-1}(M,TM)^{S^1}$. However, as we are on the zero section the projection is independent of the connection chosen. As shown in Appendix~\ref{AppendixAutomorphisms} of \cite{AGK}, given an arbitrary almost complex structure $J$, we can extend the usual $\delbar_J$ operator to an operator $\overline{\partial}_J:\Omega^{0,1}_{J,l}(M, TM) \to \Omega^{0,2}_{J,l-1}(M, TM)$ so that $\nabla N_J:=\nabla N(J)$ is given by the following composition. 
\[
\begin{tikzcd}
     &\Omega^{0,1}_{J,l}(M, TM)^{S^1} \arrow[r,"dN_J"] \arrow[rr,bend right=15,"\nabla N_J"]&\left(\Omega^{2}_{l-1}(M, TM \otimes \C)\right)^{S^1} \arrow[r,"\pi"] &\Omega^{0,2}_{J,l-1}(M,TM)^{S^1}
\end{tikzcd} 
\]
\noindent where $\pi$ is the canonical projection of 
\[
\Omega^{2}_{J,l-1}(M, TM \otimes \C)^{S^1} 
\cong \Omega^{2,0}_{J,l-1}(M,TM)^{S^1} \oplus \Omega^{1,1}_{J,l-1}(M,TM)^{S^1} \oplus \Omega^{0,2}_{J,l-1}(M,TM)^{S^1}
\]
onto the last summand. 

\begin{thm}[\cite{AGK}, Corollary A.9]
$\nabla N(J) = -2 J \overline \partial_J$. \qed
\end{thm}
We are lead to show that $\overline \partial_J: \Omega^{0,1}_{J,l}(M,TM)^{S^1} \to \Omega^{0,2}_{J,l-1}(M,TM)^{S^1}$ is surjective. This is trivially true  whenever the manifold  as $M$ is 4 dimensional and $H_J^{0,2}(M,TM)^{S^1} = 0$. 

\begin{lemma}{\label{Lemma:averaging_surj}}
Consider a complex manifold $(M,J)$ with a holomorphic $S^1$ action. Then the averaging map 
\begin{align*}
    \rho: H_J^{0,2}(M,TM) &\rightarrow H_J^{0,2}(M,TM)^{S^1} \\
             [\beta] &\mapsto \left[\int_{S^1}g^*\beta ~dg\right]
\end{align*}
is surjective.
\end{lemma}
\begin{proof}
The fact that the above map is well defined follows by noting that the $\overline \partial_J$ operator commutes with the averaging operator. The surjectivity follows as the averaging operator is the identity on invariant forms.
\end{proof}

Note that the above theorem works for any compact group. By the discussion in the previous paragraph and Lemma~\ref{Lemma:averaging_surj} we can conclude that $\overline \partial_J: \Omega^{0,1}_{J,l}(M,TM)^{S^1} \to \Omega^{0,2}_{J,l-1}(M,TM)^{S^1}$ is surjective for any holomorphic $S^1$ action on a complex 4-manifold satisfying $H_J^{0,2}(M,TM)=0$.

\begin{thm}
Let $(M,J)$ be a $4$-manifold endowed with an integrable complex structure $J$, and with a holomorphic $S^1$ action. Suppose $H_J^{0,2}(M,TM) = 0$. Then the space $I^{S^1}_l$ of invariant complex structures is a Banach submanifold of $\J^{S^1}_l$ in a neighbourhood of $J$ with tangent space at $J$ identified with $ \ker \overline \partial_J:\Omega^{0,1}_{J,l}(M,TM)^{S^1} \to \Omega^{0,2}_{J,l-1}(M,TM)^{S^1} $. Equivalently,
\[
T_J I_l \cong \left(\im  \overline \partial_J: \Omega^{0,0}_{J,l}(M,TM)^{S^1} \to \Omega^{0,1}_{J,l}(M,TM)^{S^1}\right) \oplus H^{0,1}_J(M,TM)^{S^1}
\]
\end{thm}

Let us now assume that $M$ is symplectic. Let $\J^{S^1}_{\om,l}$ denote the space of all $S^1$ equivariant \emph{compatible} almost complex structures of regularity $C^l$ endowed with the $C^l$-topology. Our next goal is to show that under some cohomological restrictions, the space of equivariant integrable \emph{compatible} almost complex structures of regularity $C^l$ denoted by $I^{S^1}_{\om,l}$ is a Banach submanifold of $\J^{S^1}_{\om,l}$. We first note that given $J \in \J^{S^1}_{\om,l}$, the equivariant metric 
$h_J (\cdot, \cdot) := \om(\cdot, J\cdot) - i \om(\cdot, \cdot)$
 induced by the pair $(\om, J)$ identifies 
$T_J \J^{S^1}_{l} = \Omega^{0,1}_l(M,TM)^{S^1}$ 
with the space  $\left(T^{0,2}\right)^{S^1}:= \left(\Omega^{0,2}(M)\right)^{S^1} \otimes \left(\Omega^{0,2}(M)\right)^{S^1}$ of complex equivariant $(0,2)$-tensors via the map
\begin{align*}
\theta:\left(T^{0,2}\right)^{S^1} &\to \Omega^{0,1}_l(M,TM)^{S^1} \\
A &\mapsto \theta(A):= h_J(A \cdot, \cdot)
\end{align*}

Let us denote by $S\Omega^{0,1}_{J,l}(M,TM)^{S^1}$ the tangent space of $T_J \J^{S^1}_{\om,l} \subset T_J \J^{S^1}_{l}$ of all equivariant compatible almost complex structures. More explicitly, the tangent space consists of elements $A \in \Omega^{0,1}_{J,l}(M,TM)^{S^1} $ such that $\om(A\cdot, \cdot) = - \om(\cdot, A\cdot)$. Under the above identification, we can check that $S\Omega^{0,1}_{J,l}(M,TM)^{S^1}$ gets mapped to the space of symmetric $S^1$ invariant $(0,2)$-tensors which we denote by $\left(S^{0,2}\right)^{S^1}$. 

Further, the quotient $T_{J} \mathcal{J}_l^{S^1} / T_{J} \jsoml$ may be identified with the space of invariant $(0,2)$ forms on $M$ since
\begin{multline*}
T_{J} \mathcal{J}_l^{S^1} / T_{J} \jsoml = \Omega^{0,1}_{J,l}(M,TM)^{S^1} / S \Omega^{0,1}_{J,l}(M,TM)^{S^1}\\
\cong T_{J}^{0,2}(M) / \left(S^{0,2}\right)^{S^1} = \Omega_{J}^{0,2}(M)^{S^1}.
\end{multline*}
As before, the Nijenhuis tensor defines a map 
\[N:\J^{S^1}_{\om,l} \to \Omega^{0,2}_{l-1}(M,TM)^{S^1}\]
whose kernel is precisely the subspace $I^{S^1}_{\om,l}$. We want to show that the derivative $\nabla N$ is surjective at all $J\in I^{S^1}_{\om,l}$. As we know that $\nabla N(J) = -2 J \overline \partial_J $, we would need to show that $\overline\partial_J:S\Omega^{0,1}_{J,l}(M,TM)^{S^1} \to  \Omega^{0,2}_{J,l-1}(M,TM)^{S^1}$ is surjective. As $M$ is a 4-manifold, all forms in $\Omega^{0,2}_{l-1}(M,TM)^{S^1}$ are closed, hence to show that the restriction of $\overline\partial_J$ to $S\Omega^{0,1}_{J,l}(M,TM)^{S^1}$ is surjective, it would suffice to show that the vector space $SH_{J}^{0,2}(TM)^{S^1}$ defined below is 0.
\begin{align*}
SH_{J}^{0,2}(TM)^{S^1}
&:= \frac{\ker \overline\partial:\Omega^{0,2}_{J,l-1}(M,TM)^{S^1} \to \Omega^{0,3}_{J,l-2}(M,TM)^{S^1}}{\im \overline\partial:S\Omega^{0,1}_{J,l}(M,TM)^{S^1} \to  \Omega^{0,2}_{J,l-1}(M,TM)^{S^1}}\\
&=\frac{\Omega^{0,2}_{J,l-1}(M,TM)^{S^1}}{\im \overline\partial:S\Omega^{0,1}_{J,l}(M,TM)^{S^1} \to  \Omega^{0,2}_{J,l-1}(M,TM)^{S^1}}
\end{align*}
As the above condition is not easy to check directly, we consider the following commutative diagram 
\[
\begin{tikzcd}[cramped]
     &0 \arrow[d] \arrow[r] &S\Omega^{0,1}_{J,l}(M,TM)^{S^1} \arrow[r, "\overline\partial"] \arrow[d] &\Omega^{0,2}_{J,l-1}(M,TM)^{S^1} \arrow[d] \arrow[r] &0 \\
     &\Omega^{0,0}_{J,l+1}(M,TM)^{S^1} \arrow[r,"\overline\partial"] \arrow[d,"\alpha"] &\Omega^{0,1}_{J,l}(M,TM)^{S^1} \arrow[r,"\overline\partial"] \arrow[d] &\Omega^{0,2}_{J,l-1}(M,TM)^{S^1} \arrow[d] \arrow[r] &0 \\
     &\Omega^{0,1}_{J,l+1}(M)^{S^1} \arrow[r, "\overline\partial"] &\Omega^{0,2}_{J,l}(M)^{S^1} \arrow[r, "\overline\partial"] &0 \arrow[r] &0
\end{tikzcd}  
\]
where the map $S\Omega^{0,1}_{J,l}(M,TM)^{S^1} \to \Omega^{0,1}_{J,l}(M,TM)^{S^1}$ is just the inclusion and the where the map $\Omega^{0,1}_{J,l}(TM)^{S^1} \to \Omega^{0,2}_{J,l}(M)^{S^1}$ is the quotient \[\Omega^{0,1}_{J,l}(TM)^{S^1} \to \Omega^{0,1}_{J,l}(TM)^{S^1}/S\Omega^{0,1}_{J,l}(M,TM)^{S^1}\]
followed by identifying $\Omega^{0,1}_{J,l}(TM)^{S^1}/S\Omega^{0,1}_{J,l}(M,TM)^{S^1}$ with
$\Omega^{0,2}_{J,l-1}(M,TM)^{S^1}$ (see \cite{AGK} p.548 for more details about this identification). The map $\alpha$ is defined~as
\begin{align*}
    \alpha: \Omega^{0,0}_{J,l+1}(M,TM)^{S^1} &\to \Omega^{0,1}_{J,l+1}(M)^{S^1} \\
    X &\mapsto \alpha(X)(Y):= \om(X,JY)-i\om(X,Y)
\end{align*} 
where $J \in I^{S^1}_{\om,l}$ and $X,Y \in \Omega^{0,0}_{J,l+1}(M,TM)^{S^1}$. We refer the reader to  Appendix~B in~\cite{AGK} for  the proof of commutativity of the diagram in the non-equivariant case, and we note that it still holds in the equivariant setting due to the fact that $\overline\partial_J$ is equivariant. The above diagram gives rise to a long exact sequence is cohomology

\begin{equation}\label{les}
\begin{aligned} 0 \longrightarrow H_{J}^{0}(T M)^{S^1} & \longrightarrow cl\Omega_{J}^{0,1}(M)^{S^l} \stackrel{\delta}{\longrightarrow} cl S\Omega_{J}^{0,1}(M,TM)^{S^1} \stackrel{q}\longrightarrow H_{J}^{0,1}(T M)^{S^1} \longrightarrow \\ & \longrightarrow H_{J}^{0,2}(M)^{S^1} \longrightarrow SH_{J}^{0,2}(TM)^{S^1} \longrightarrow H_{J}^{0,2}(T M)^{S^1} \longrightarrow 0 
\end{aligned}
\end{equation}

\noindent where $cl\Omega_{J}^{0,1}(M)^{S^l}$ denotes the kernel of $\overline\partial_J$ in $\Omega_{J}^{0,1}(M)^{S^l}$ and where, similarly, $cl S\Omega_{J}^{0,1}(M,TM)^{S^1}$ is the kernel of $\overline\partial_J: S\Omega^{0,1}_{J,l}(M,TM)^{S^1} \to  \Omega^{0,2}_{J,l-1}(M,TM)^{S^1}$. Thus if we had a 4-manifold $M$ with an $S^1$ invariant compatible integrable almost complex structure $J$ such that  $H_J^{0,2}(M) =0$ and $H_J^{0,2}(TM)=0$, noting that $\overline{\partial}_J$ takes $S^1$ invariant elements to $S^1$ invariant elements we can conclude that $H_J^{0,2}(M)^{S^1} =0$ and $H_J^{0,2}(M,TM)^{S^1}=0$. Further, it follows from equation \ref{les} for such a manifold $(M,\om,J)$ as above, that $SH_{J}^{0,2}(TM)^{S^1}=0$ and hence $I^{S^1}_{\om,l}$ would indeed be a manifold in a neighbourhood of such a $J$. Thus $H_J^{0,2}(M) =0$ and $H^{0,2}_J(TM)=0$ gives us a simpler condition for when $I^{S^1}_{\om,l}$ would indeed be a manifold in a neighbourhood of $J$ as required.
\\

Additionally, as  the averaging operator commutes with the $\overline{\partial}_J$ operator, the assumption $H_J^{0,2}(M)=0$ implies that $H_J^{0,2}(M)^{S^1} =0$. This tells us that the map $q:cl S\Omega_{J}^{0,1}(M,TM)^{S^1} \to H_{J}^{0,1}(T M)^{S^1} $ is surjective and hence by the first isomorphism theorem we have $\frac{cl S\Omega_{J}^{0,1}(M,TM)^{S^1}}{ker ~q:cl S\Omega_{J}^{0,1}(M,TM)^{S^1} \to H_{J}^{0,1}(T M)^{S^1}}$ is isomorphic to $H_{J}^{0,1}(T M)^{S^1}$.
Then the above long exact sequence \ref{les} gives us 
\[\frac{cl S\Omega_{J}^{0,1}(M,TM)^{S^1}}{ker~ q} \cong \frac{cl S\Omega_{J}^{0,1}(M,TM)^{S^1}}{\im \delta} \cong H_{J}^{0,1}(M,TM)^{S^1}.\]
Putting all this together we obtain the following local description of $I^{S^1}_{\om,l}$.
\begin{thm}\label{integrable}
Let $(M,\om,J)$ be a  K\"ahler 4-manifold with a K\"ahler $S^1$ action. Suppose that $H_J^{0,2}(M)=0$ and $H_J^{0,2}(TM)=0$. Then  $I^{S^1}_{\om,l}$ is a Banach submanifold of $\J^{S^1}_{\om,l}$ in a neighbourhood of $J$ with tangent space at $J \in I^{S^1}_{\om,l}$ identified with
\[
T_J I^{S^1}_{\om,l} = cl S\Omega_{J}^{0,1}(M,TM)^{S^1} = \ker \overline\partial_J: S\Omega^{0,1}_{J,l}(M,TM)^{S^1} \longrightarrow  \Omega^{0,2}_{J,l-1}(M,TM)^{S^1}.
\]
Equivalently,
\[T_J I^{S^1}_{\om,l} \cong \im \delta \bigoplus H_{J}^{0,1}(T M)^{S^1}.\]
\end{thm}

\begin{prop}
The conditions $H_J^{0,2}(M) =0$ and $H^{0,2}_J(M,TM)=0$ are satisfied for all the Hirzebruch surfaces.
\end{prop}
\begin{proof}
To prove $H^{0,2}_J(M,TM)=0$ for all Hirzebruch surfaces see the computation in Example 6.2b) p.312 in \cite{Ko}. To prove, $H_J^{0,2}(M) =0$ we note that the the rank of $H_J^{0,2}(M) =0$ (usually called the geometric genus $p_g$) is a birational invariant. As all Hirzebruch surfaces are birationally equivalent, the result follows from the computation on p.220 in \cite{Ko}. 
\end{proof}
Finally, we would like to show that the strata $U_{s,l} \cap \J^{S^1}_{\om,l}$ is a Banach submanifold of $\J^{S^1}_{\om,l}$. The most direct way to prove this would be to consider the universal moduli space $\mathcal{M}(D_s,\joml)$ of curves in the class $D_s$ (where $D_s$ is defined to be the class $B -\frac{s}{2}F$ if $M =\SSS$ or the class $B- \frac{s+1}{2}F$ if $M= \CCC$) and try to prove that the inclusion of $\jsoml$ is transverse to the projection of $\mathcal{M}(D_s,\joml)$ to the space of all compatible almost complex structures of regularity $C^l$:
\[
\begin{tikzcd}[row sep=small]
&\mathcal{M}(D_s,\joml)\arrow[d,"\pi"]\\
\jsoml \arrow[r,"i"] &\joml
\end{tikzcd} 
\] 
However, this approach is flawed as the two maps are never transversal. An alternative method is to try to define an equivariant universal moduli space $\mathcal{M}^{S^1}(D_s,\jsoml)$ and argue that the image under the projection to $\jsoml$ is a Banach submanifold of $\jsoml$. This is the approach we implement in the following section.

\section{Construction of Equivariant moduli spaces}

In this section we construct moduli spaces of $S^1$ invariant $J$-holomorphic maps into $\SSS$ or $\CCC$. Recall that $J_m$ is the standard complex structure on the $m^\text{th}$ Hizerbruch surface $W_m$, where $m=2k$ or $m=2k+1$. Let $D_s$ denote the homology class $B-\frac{s}{2}F$ in $\SSS$ and let it denote the class $B -\frac{s+1}{2}F$ in $\CCC$. As seen in Chapter~2, there is a $\T_m$ invariant, $J_m$-holomorphic curve $\overline{D}$ in $W_m$ in the homology class $D_s$. Consider the $S^1(a,b;m)$ action on $(\SSS,\oml)$ or $(\CCC,\oml)$. From the graph for the circle action $S^1(a,b;m)$ we see that $S^1$ acts on $\overline{D}$ in a non-effective manner with global stabilizer~$\Z_{a}$. 

\begin{lemma}\label{WellDefinedModuli}
Let $S^1(a,b;m)$ be a Hamiltonian circle action on $(\SSS,\oml)$ or $(\CCC,\oml)$. Let $S$ be any $S^1(a,b;m)$-invariant symplectic embedded sphere in the homology class $D_s$ with $s>0$. Then the $S^1$ action on $S$ has global stabilizer isomorphic to~$\Z_a$. 
\end{lemma}

\begin{proof}
Using Lemma~\ref{weight} and Table~\ref{table_weights} we can show that any other $S^1$ invariant curve passes through the same set of fixed points as $\overline{D}$ and, using Lemma~\ref{weight} again, that the global stabilizer is the same. 
\end{proof}

Now, pick any $S^1(a,b;m)$ invariant, embedded, symplectic sphere $C$ in class $D_s$. The curve $C$ is $J$-holomorphic for some $J\in\jsoml$. Then any $J$-holomorphic parametrization $u: (S^2,j_0) \to (M,J)$ defines a holomorphic $S^1$ action on $S^2$ whose stabilizer is isomorphic to $\Z_a$. Since all circle subgroups of the automorphism group $\PSL(2,\C)$ are conjugate, there is a $\Lambda\in\PSL(2,\C)$ such that $u\circ\Lambda$ brings the restricted $S^1$ action on $C$ to the standard action  $S^1(a)$ on $S^2$ with weights $\{a, -a\}$ at the poles. Consequently, $C$ is the image of some equivariant $J$-holomorphic map $\big(S^2, j_0\big)\to \big(\SSS, J\big)$ that intertwines the actions $S^1(a)$ and $S^1(a,b;m)$.\\

We can now define the universal moduli space of all such equivariant maps by setting 
\[
\begin{aligned}
\mathcal{M}^{S^1}(D_s,\, \jsoml): = \{(&u,J) ~|~ u:S^2 \to M\text{~is  equivariant, somewhere injective,} \\
&\text{$J$-holomorphic and represents the class~}D_s, ~J\in\jsoml\}
\end{aligned}
\]

\begin{remark}
As we are only interested in the case when $s>0$, the curves in class $D_s$ have negative self intersection and the adjunction formula implies that all somewhere injective curves in class $D_s$,  $s>0$, are embedded. 
\end{remark}
As in the non-equivariant case we now wish to prove that this moduli space is a smooth Banach manifold. To prove this we recall the non-equivariant set up as in Chapter 3 in \cite{McD} and reformulate it in the equivariant setting.\\

Define the space $B^{S^1}_{q,p}$ by setting
\begin{equation}\label{spaceofmapsModuliSpace}
    B^{S^1}_{q,p}:= \{ u \in \left(W^{q,p}(S^2,M)\right)^{S^1} ~|~ [u]=D_s\}
\end{equation}
where $W^{q,p}(S^2,M))^{S^1}$ denotes the space of equivariant maps of Sobolev regularity $W^{q,p}$ from $S^2$ to $M$. Consider the vector bundle $\mathscr{E}^{S^1}_{q-1,p} \xrightarrow{\pi}{} B^{S^1}_{q,p}\times \J^{S^1}_{\oml,l}$ whose fibre over $(u,J)$ is the space of $S^1$ invariant sections of $\Omega^{0,1}_J(S^2, u^*TM)$ of Sobolev  regularity $W^{q-1,p}$, that is,
\[ \pi^{-1}(u,J) = \Gamma(S^2, \Omega^{0,1}_J(S^2, u^*TM)^{S^1}).\]
We would like to show that the  section $\mathscr{F}^{S^1}(u,J):= (u,\overline \partial_J u) :B_{q,p}^{S^1} \times \J^{S^1}_{\oml,l} \to \mathscr{E}^{S^1}_{q-1,p}$  (where $\overline \partial_J u = \frac{1}{2}(du + J \circ du \circ j_{S^2})$) is transversal to the zero section. Note that   $\left({\mathscr{F}^{S^1}}\right)^{-1}(0) = \mathcal{M}^{S^1}(D_s , \jsoml)$, thus giving it a smooth structure as in the non-equivariant case (See \cite{McD} Lemma 3.2.1).\\

In order to show transversality, we consider the projection  of the derivative map $d\mathscr{F}^{S^1}$ to the vertical tangent bundle and show that this map is surjective at $(u,J)$ when $u$ is a simple equivariant curve. We shall denote this projection by $D\mathscr{F}^{S^1}$. More explicitly, we need to show that the map 
\[
\begin{aligned}
D\mathscr{F}_{u,J}^{S^1}: \mathscr{W}^{q,p}(S^2, u^*TM)^{S^1} \times &S\Omega^{0,1}_{J,l}(M,TM)^{S^1} 
\to &\mathscr{W}^{q,p}(S^2, \Omega^{0,1}_J(S^2, u^*TM))^{S^1}
\end{aligned}
\]
is surjective where $S\Omega^{0,1}_{J,l}(M,TM)^{S^1}$ is the tangent space of $\jsoml$ at $J$ and where $\mathscr{W}^{q,p}(S^2, \Omega^{0,1}_J(S^2, u^*TM))^{S^1}$ denotes the space of equivariant maps  from $S^2$ to $\Omega^{0,1}_J(S^2, u^*TM)$ which are in the Sobolev class $\mathscr{W}^{q,p}$. By Lemma 3.2.1 in \cite{McD}, we know that in the non-equivariant situation, the linearized derivative  
\begin{equation}\label{non-equivariant section}
D\mathscr{F}_{u,J}: \mathscr{W}^{q,p}(S^2, u^*TM) \times S\Omega^{0,1}_{J,l}(M,TM)^{S^1} 
\to \mathscr{W}^{q,p}\left(S^2, \Omega^{0,1}_J\left(S^2, u^*TM\right)\right)
\end{equation}
is surjective. As $J \in \jsom$, the $\delbar_J$ operator commutes with the averaging operator with respect to the $S^1$ action.  Noting that the $D\mathscr{F}_{u,J}$ is equivariant with respect to the $S^1$ actions defined on the domain and codomain, we can  average the above non-equivariant derivative $D\mathscr{F}_{u,J}$ by the $S^1$ action to show that $D\mathscr{F}_{u,J}^{S^1}$ is surjective as well.

\begin{thm}\label{Banach_S^1}
 $\mathcal{M}^{S^1}(D_s , \jsoml)$ is a smooth Banach manifold. \qed
\end{thm}

We now consider the projection map 
\[
\begin{tikzcd}
     &\mathcal{M}^{S^1}(D_s , \jsoml) \arrow[d,"\pi"] \\
     &\jsoml
\end{tikzcd}
\]
To conclude that the image of $\pi$ is a submanifold of $\jsoml$ we need the following theorem whose proof can be found in  \cite{mardsen}.

\begin{thm}\label{mrdsn}(Theorem 3.5.18 in \cite{mardsen})
Let $f:M \to N$ be a smooth map between Banach manifolds such that 
\begin{enumerate}
    \item $\ker Tf$ is a sub-bundle of $TM$
    \item For each $m \in M$, $f_*(T_m M)$ is closed and splits in $T_{f(m)}N$ 
    \item f is open or closed onto it's image
\end{enumerate}    
Then $f(M)$ is a smooth Banach submanifold of N.
\end{thm}

A map that satisfies the above three conditions is called a sub-immersion.\\

We shall show below that the map $\pi: \mathcal{M}^{S^1}(D_s , \jsoml) \to \jsoml$ is a sub-immersion. But before that we need the following lemma.

\begin{lemma}\label{Lemma:Equivariant_Fredholm}Let $G$ be a compact group acting on vector spaces $V$ and $W$ and let $V^G$ and $W^G$ be the fixed point sets of $V$ and $W$ respectively under the action of $G$.
Let  $L: V \rightarrow W$ be a $G$-equivariant linear map with finite dimensional kernel and cokernel.  Let $L^G: V^G \rightarrow W^G$ be the restriction of $L$ to $V^G$. Then $L^G: V^G \rightarrow W^G$  has finite dimensional kernel and cokernel.
\end{lemma}
\begin{proof}
Firstly  the kernel of $L^G: V^G \rightarrow W^G$ is a subspace of $\Ker L$. Hence $\Ker L^G$ is finite dimensional. 
\\

As $\Coker L$ is assumed finite dimensional, it suffices to show that the natural map 
\begin{align*}
    P: \Coker L^G &\rightarrow \Coker L \\
        [w] &\mapsto [w]
\end{align*} 
is injective in order to show that $\Coker L^G$ is also finite dimensional.  Consider $[w] \in \Coker L^G$ such that $P(w)=0$. Then $w + l = l'$ for some $w \in W^G$, $l \in \image L^G$ and $l' \in \image L$. Averaging both sides with respect to the $G$ action and using the fact that $w$ and $l$ are both invariant under the $G$ action, we get that $w$ in fact belongs to $\image L^G$.  Thus $[w] = 0 \in \Coker L^G$ showing that $\Ker P = 0$.

\end{proof}
\begin{lemma}\label{subimmersion}
The projection map $\pi: \mathcal{M}^{S^1}(D_s , \jsoml) \to \jsoml$ is a sub-immersion.
\end{lemma}
\begin{proof}

Note that the $\ker  d\pi$ is of constant rank and is the tangent space to the reparametrization group $\C^*$ is which freely on $\mathcal{M}^{S^1}(D_s , \jsoml)$. Hence $\ker d\pi$ is a sub-bundle of $T\mathcal{M}^{S^1}(D_s , \jsoml)$.
\\

Now we show that the image of $d\pi$ is closed  $T_J\jsoml$. Note first that $T_J\jsoml = S\Omega_J^{0,1}(M,TM)^{S^1}$. Hence, $\pi_* T_{(u,J)} \mathcal{M}^{S^1}\left(D_s , \jsoml\right)$ can be described as
\begin{equation}\label{proj_tangent}
 \pi_*T_{(u,J)} \mathcal{M}^{S^1}\left(D_s , \jsoml\right)=
 \left\{ \alpha \in S\Omega_J^{0,1}(M,TM)^{S^1}~| ~ \alpha \circ du \circ j_{S^2} \in \image D^{S^1}_{u}
\right\}
\end{equation}

This follows from the fact that  $$
D\mathscr{F}_{u,J}^{S^1}(\xi, Y)=D^{S^1}_{u} \xi+\frac{1}{2} Y(u) \circ d u \circ j_{0}
$$
for $\xi \in W^{q, p}\left(S^2,u^{*} T M\right)^{S^1}$ and $Y \in S\Omega^{0,1}_{J,l}(M,TM)^{S^1}$, where $D^{S^1}_{u}$ is a 0th-order perturbation (coming from the possible non-integrability of $J$ ) of the usual Dolbeault $\bar{\partial}$-operator.
\\

Note that the differential of the non-equivariant map in equation \ref{non-equivariant section} is also given $D\mathscr{F}_{u,J}(\xi, Y)=D_{u} \xi+\frac{1}{2} Y(u) \circ d u \circ j_{0}$ and $D_{u}$ is Fredholm. Applying Lemma \ref{Lemma:Equivariant_Fredholm} to $D^{S^1}_{u}$ and $D_{u}$ we get that $D^{S^1}_{u}$ has finite dimensional kernel and cokernel.
\\

Let $\gamma_n \in \pi_* T_{(u,J)} \mathcal{M}^{S^1}\left(D_s , \jsoml\right)$ i.e $\gamma_n \in \Omega^{0,1}(M,TM)^{S^1}= T\jsoml $  and satisfies  $\gamma_n \circ du \circ J_{S^2} \in \image D^{S^1}_{u}$.  Further assume the sequence $\gamma_n$ converges to $\gamma$ in $\Omega^{0,1}(M,TM)^{S^1}$. As  $D^{S^1}_{u}$ has finite cokernel it implies that $\image D^{S^1}_{u}$  is closed we have that  $\gamma \circ du \circ j_{S^2} \in \image D^{S^1}_{u}$  thus showing image of $d\pi$ is closed  $T\jsoml$.
\\

Next to that the image of $d\pi$ splits in $T\jsoml$, we proceed as follows. First we argue that the map 
\[\begin{aligned}
D\mathscr{F}_{u,J}^{S^1}: \mathscr{W}^{q,p}(S^2, u^*TM)^{S^1} \times &S\Omega^{0,1}_{J,l}(M,TM)^{S^1} 
\to &\mathscr{W}^{q,p}(S^2, \Omega^{0,1}_J(S^2, u^*TM))^{S^1}
\end{aligned}
\]
has finite codimension. This follows from applying Lemma\ref{Lemma:Equivariant_Fredholm} to the maps $D\mathscr{F}_{u,J}^{S^1}$ and $D\mathscr{F}_{u,J}$. Note that $D\mathscr{F}_{u,J}$ has finite dimensional kernel and cokernel as $D\mathscr{F}_{u,J}$ is Fredholm. 
Next we show that the codimension of the image of $d\pi$ in $T\jsoml$ is finite and hence it follows that the image of $d\pi$ splits.

We have a natural isomorphism
$$
S\Omega^{0,1}_{J,l}(M,TM)^{S^1} / \operatorname{Im}(d \pi(u, J)) \longrightarrow W^{l, p}\left(\wedge^{0,1} T^{*} S^{2} \otimes_{J} u^{*} T M\right)^{S^1} / \operatorname{Im}\left(D_{u}\right)
$$
given by
$$
Y \mapsto \frac{1}{2} Y(u) \circ d u \circ j_{0}
$$

The image of $d \pi_{(u, J)}$ consists of all $Y$ such that $Y(u) \circ d u \circ j_{S^2} \in \operatorname{Im}\left(D_{u}\right)$. This is a closed subspace of $T_{J} \jsoml$ and, since $D\mathscr{F}^{S^1}(u, J)$ is onto, it has the same finite codimension as the image of $D_{u}$. Hence the differential $d\pi$ has a finite dimensional cokernel in $T_J\jsoml$ and thus splits.\\

Finally to show that $\pi$ is open onto it's image, we consider the factorization 
\[\begin{tikzcd}
      &\mathcal{M}^{S^1}(D_s , \jsoml) \arrow[dr,"\pi"] \arrow[d, "q"] \\
      &\mathcal{M}^{S^1}(D_s , \jsoml)/\C^* \arrow[r,"h"]  & \text{im} ~\pi
\end{tikzcd}
\]
The map $h:\mathcal{M}^{S^1}(D_s , \jsoml)/\C^* \to \im \pi$ is continuous and, by positivity of intersections, is bijective. By Gromov's compactness theorem, its inverse is continuous, so that $h$ is a homeomorphism. Since the map $q$ is the quotient map of the $\C^*$ reparametrization action, $q$ is open. Consequently, $\pi = h \circ q$ is open onto its image and, therefore, $\pi$ is a sub-immersion. 
\end{proof}

\begin{cor}\label{cor:StrataAreSubmanifolds}
$U_{s,l} \cap \jsoml$ is a Banach submanifold of $\jsoml$.
\end{cor}
\begin{proof}
This follows from Lemma~\ref{subimmersion} and from observing that the image of $\pi$ is $U_{s,l} \cap \jsoml$.
\end{proof}

We will now describe the normal bundle of $U_{s,l} \cap \jsoml$ in $\jsoml$ (when $s > 0$) in terms of infinitesimal deformations of complex structures. To this end, we first find a cohomological condition ensuring that the inclusion of complex integrable structures into $\jsoml$ is transverse to the stratum $U_{s,l} \cap \jsoml$, in other words, that we have transverse maps
\[
\begin{tikzcd}
      &\mathcal{M}^{S^1}(D_s , \jsoml) \arrow[d,"\pi"] \\
     I^{S^1}_{\om,l} \arrow[r,"i"]  &\jsoml 
\end{tikzcd}
\\
\]
\begin{lemma} \label{compliment}
Let $(M,\oml,J_s)$  denote any of the Hirzebruch surfaces $(\SSS,\oml,J_s)$ or $(\CCC,\oml,J_s)$, and let $(u,J_s)\in \mathcal{M}^{S^1}(D_s , \jsoml)$. Then the induced map $u^*:H_{J_s}^{0,1}\left(M,TM\right)^{S^1} \rightarrow H^{0,1}_{j_{S^2}}\left(S^2, u^*TM\right)^{S^1}$ is an isomorphism. 
\end{lemma}

\begin{proof}
From  Proposition 3.4 in \cite{AGK} we know that $u^*:H_{J_s}^{0,1}(M,TM) \rightarrow H_{j_{S^2}}^{0,1}(S^2, u^*TM) $ is an isomorphism. As $u$ is equivariant this indeed gives us that \[u^*:H_{J_s}^{0,1}\left(M,TM\right)^{S^1} \longrightarrow H_{j_{S^2}}^{0,1}\left(S^2, u^*TM\right)^{S^1}\]
is also an isomorphism.
\end{proof}

\begin{lemma}\label{trans_strata}
Let $(M,\oml)$ denote either $(\SSS,\oml)$ or $(\CCC,\oml)$. Further let $i: I^{S^1}_{\om,l} \hookrightarrow  \jsoml$ denote the inclusion and  $\pi: \mathcal{M}^{S^1}(D_s , \jsoml) \to \jsoml$ denote the projection. Then,
\begin{itemize}
    \item $i \pitchfork \pi$,
    \item the infinitesimal complement (i.e the fibre of the normal bundle)  of $U_{s,l} \cap \jsoml$ at $J_s \in I^{S^1}_{\om,l}$ can be identified with $H^{0,1}_{J_s}(M,TM)^{S^1}$.
\end{itemize}
\end{lemma}

\begin{proof}Recall by Theorem~\ref{integrable}, that the tangent space of $I^{S^1}_{\om,l}$ was given by 
\begin{align*}
T_{J_s} I^{S^1}_{\om,l} &= cl S\Omega_{J_s}^{0,1}(M,TM)^{S^1}\\
&= \ker\Big( \overline\partial_{J_s}: S\Omega^{0,1}_{{J_s},l}\left(M,TM\right)^{S^1} \to \Omega^{0,2}_{{J_s},l-1}\left(M,TM\right)^{S^1}\Big)
\end{align*}
Let $\gamma \in T_{J_s} \jsoml = \left(S\Omega^{0,1}_{J_s}(M,TM)\right)^{S^1}$ and define
\[
\left[\gamma \circ du \circ j_{S^2}\right] := \eta \in H_{j_{S^2}}^{0,1}(S^2, u^*TM)^{S^1}.
\]
To show that $i \pitchfork \pi$, we need to produce $\beta \in T_{J_s}I^{S^1}_{\om,l} = cl S\Omega_{J_s}^{0,1}(M,TM)^{S^1}$ such that $[\left(\gamma - \beta\right) \circ du \circ j_{S^2}] = 0$. To do so, we consider the following commutative diagram.
\[ 
\stackinset{l}{13ex}{b}{6ex}{%
\scalebox{.8}
{%
\begin{tikzcd}[row sep=9ex, column sep = 13ex, ampersand replacement=\&]
[\alpha] 
    \arrow[mapsto]{r}
    \arrow[mapsto]{d}
\& \left[\alpha \circ du \right]
    \arrow[mapsto]{d}  \\
\left[\alpha \circ J_m\right]
    \arrow[mapsto]{r} 
\& \begin{array}{c}[\alpha \circ du \circ j_{S^2}] =\\ \left[\alpha \circ J_s \circ du\right]\end{array}
\end{tikzcd}
} 
}{%
\begin{tikzcd}[row sep = 20ex, column sep = 20ex, ampersand replacement=\&]
H_{J_s}^{0,1}(M,TM)^{S^1} 
    \arrow{r}{u^*} 
    \arrow[swap]{d}{J_s^*}
\& H_{j_{S^2}}^{0,1}(S^2, u^*TM)^{S^1}
    \arrow{d}{j_{S^2}^*}  \\
H_{J_s}^{0,1}(M,TM)^{S^1} 
    \arrow[swap]{r}{u^*} 
\& H_{j_{S^2}}^{0,1}(S^2, u^*TM)^{S^1}
\end{tikzcd}
}
\]
where all the maps $u^*$,$J^*$ and $j_{S^2}^*$ are isomorphisms. Further we have the equality $\left[\alpha \circ du \circ j_{S^2}\right] = \left[\alpha \circ J \circ du\right]$ as $u$ is $j_{S^2}$-$J_s$ holomorphic. As we know that $H_{J_s}^{0,2}(M)^{S^1} =0$, we see  from the long exact sequence equation~\eqref{les} that the map  $cl S\Omega_{J_s}^{0,1}(M,TM)^{S^1} \to H_{J_s}^{0,1}(M,TM)^{S^1}$  
is surjective. As both $u^*$ and $J^*$ are isomorphisms, there exists $\beta \in cl S\Omega_{J_s}^{0,1}(M,TM)^{S^1} = T_{J_s} I_{\oml,l}^{S^1}$ such that $\left[\beta \circ J \circ du\right]  =  \left[\beta \circ du \circ j_{S^2}\right]= \eta:= \left[\gamma \circ du \circ j_{S^2}\right] $. Hence we indeed have $\left[\left(\gamma -\beta\right) \circ du \circ j_{S^2} \right] = 0 $ as required.
\\

We now show that the fibre of the normal bundle of $U_{s,l} \cap \jsoml$ at $J_s \in I^{S^1}_{\om,l}$ can be identified with $H^{0,1}_{J_s}(M,TM)^{S^1}$. When $J$ is an integrable invariant almost complex structure,  equation~\eqref{proj_tangent} implies that
\begin{multline*}
\pi_*T_{(u,J)} \mathcal{M}^{S^1}\left(D_s , \jsoml\right)\\
=\left\{ \alpha \in S\Omega_J^{0,1}(M,TM)^{S^1}~| ~ \alpha \circ du \circ j_{S^2} =0 \in H_{J_s}^{0,1}(M,TM)^{S^1}
\right\}.
\end{multline*}
Consider the map 
\begin{align*}
    L: S\Omega_{J_s}^{0,1}(M,TM)^{S^1} &\rightarrow H_{J_s}^{0,1}\left(S^2, u^*TM\right)^{S^1} \\
    \alpha &\mapsto [\alpha \circ du \circ j_{S^2}]
\end{align*}

The kernel of L is precisely $\pi_*T_{(u,J)} \mathcal{M}^{S^1}\left(D_s , \jsoml\right)$. As the quotient map $ cl S\Omega_{J_s}^{0,1}(M,TM)^{S^1} \to H_{J_s}^{0,1}(M,TM)^{S^1}$ is surjective and the maps $u^*$ and $j_{S^2}$ are isomorphisms, we have that the map $L$ is surjective. As the kernel of L  is the image of $d\pi$, the cokernel can be identified with $H_{j_{S^2}}^{0,1}\left(S^2, u^*TM\right)^{S^1} \cong H_{J_s}^{0,1}(M,TM)^{S^1}$. Hence the the fibre of the normal bundle of $U_{s,l} \cap \jsoml$ at $J_s \in I^{S^1}_{\om,l}$ can be identified with $H^{0,1}_{J_s}(M,TM)^{S^1}$.
\end{proof}

\section{Isotropy representations}
As shown in the previous section, the codimension of $U_{s,l} \cap \jsoml$ inside $\jsoml$ is equal to the dimension of $H_{J_s}^{0,1}(M,TM)^{S^1}$which can be calculated from Lemma~\ref{trans_strata}. In this section we only perform the calculation for $\SSS$ and we address the case of $\CCC$ in Section~\ref{Isom_CCC}.
\\

\subsection{Even Hirzebruch surfaces and their isometry groups}
Let $2k=m$. The following  Theorem  in \cite{AGK}, tells us how isometry group $K(2k)\simeq S^{1}\times \SO(3)$ acts on the space $H_{J_m}^{0,1}(M,TM)$ of infinitesimal deformations of the complex structure $J_m$. 
\begin{thm}[Theorem 4.2 in \cite{AGK}]\label{Thm_action on deformations}
The isometry group $S^1 \times SO(3)$ acts on $H^{0,1}(S^2 \times S^2,T^{1.0}_{J_m}(S^2 \times S^2))$ via the representation $\Det\otimes\Sym^{m-2}(\C^2)$.
\end{thm}
Here $\Det$ is the standard action of $S^{1}=U(1)$ on $\C^{2}$, and where $\Sym(\C^{2})$ is the representation $\mathscr{W}_{k-1}$ of $\SO(3)$ induced by the $(2k-2)$-fold symmetric product of the standard representation of $\SU(2)$ on $\C^{2}$. We use this fact to calculate the dimension of the $S^1$ invariant subspace of $H_{J_m}^{0,1}(M,TM)$ and thus obtain the codimension $U_{m,l} \cap \jsoml$ inside $\jsoml$. \\

Following \cite{AGK}, we construct the Hirzebruch surface $\mathbb{F}_{2k}$ by K\"ahler reduction of $\C^{4}$ under the action of the torus $T^{2}_{2k}$ defined by
\[(s,t)\cdot z = (s^{2k}tz_{1},tz_{2},sz_{3},sz_{4})\]
The moment map is $\phi(z)=(2k|z_{1}|^{2}+|z_{3}|^{2}+|z_{4}|^{2}, |z_{1}|^{2}+|z_{2}|^{2})$ and the reduced manifold at level $(\lambda+k,1)$ is symplectomorphic to $(S^{2}\times S^{2},\om_{\lambda})$ and biholomorphic to the Hirzebruch surface $\mathbb{F}_{2k}$. In this model, the projection to the base is given by $[(z_{1},\ldots,z_{4})]\mapsto [z_{3}:z_{4}]$, the zero section is $[w_{0}:w_{1}]\mapsto [(w_{0}^{2k},0,w_{0},w_{1})]$, and a fiber is $[w_{0}:w_{1}]\mapsto [(w_{0}w_{1}^{2k},w_{0}w_{1},0,w_{1})]$. The torus $T^{2}(2k)=T^{4}/T^{2}_{2k}$ acts on $\mathbb{F}_{2k}$. This torus is generated by the elements $[(1,e^{it},1,1)]$ and $[(1,1,e^{is},1)]$, and its moment map is $[(z_{1},z_{2},z_{3},z_{4})]\mapsto(|z_{2}|^{2},|z_{3}|^{2})$. In this basis (which is different from the standard basis used in Section~\ref{Section:ToricActionsOnSSSandCCC}), the moment polytope $\Delta(2k)$ is the convex hull of the vertices $(0,0)$, $(1,0)$, $(1,\lambda+k)$, and $(0,\lambda-k)$, as depicted in~Figure~\ref{fig:PolytopesDifferentBases}~(b).
\begin{figure}[H]
\centering
\subcaptionbox{Standard basis}
[.3\linewidth]
{\begin{tikzpicture}
\draw (0,1) -- (2,1) ;
\draw (0,1) -- (0,0) ;
\draw (0,0) -- (3,0) ;
\draw (2,1) -- (3,0) ;
\end{tikzpicture}}
~
\subcaptionbox{Second basis}
[.3\linewidth]
{\begin{tikzpicture}
\draw (0,0) -- (1,0) ;
\draw (0,0) -- (0,2) ;
\draw (1,0) -- (1,3) ;
\draw (0,2) -- (1,3) ;
\node[right]{};
\end{tikzpicture}}
~
\subcaptionbox{Balanced basis}
[.3\linewidth]
{\begin{tikzpicture}
\draw (0,0) -- (1,-1) ;
\draw (0,0) -- (0,1) ;
\draw (0,1) -- (1,2) ;
\draw (1,2) -- (1,-1) ;
\end{tikzpicture}}
\caption{$\Delta(2k)$ in different bases}\label{fig:PolytopesDifferentBases}
\end{figure}
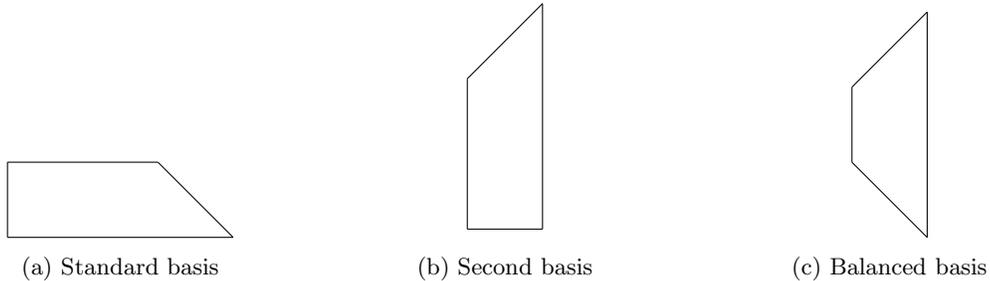

The isometry group of $\mathbb{F}_{2k}$ is 
\[K(2k) = Z_{\U(4)}(T^{2}_{2k})/T^{2}_{2k}=(T^{2}\times \U(2))/T^{2}_{2k}\simeq S^{1}\times\PU(2)\simeq S^{1}\times\SO(3)\]
where the middle isomorphism is given by
\[[(s,t), A]\mapsto (s^{-1}t\det A^{k}, [A])\]
Under this isomorphism, an element $[(1,a,b,1)]$ of the torus $T(2k)$ is taken to
\[\left(ab^{k},\begin{bmatrix}b&0\\0&1\end{bmatrix}\right)=\left(b^{k}a,\begin{bmatrix}b^{1/2}&0\\0&b^{-1/2}\end{bmatrix}\right)\]
Consequently, at the Lie algebra level of the maximal tori, the map identifying the maximal torus of $K(2k)$ whose lie algebra is denoted by $\lie{t}^{2}(2k)$ with the maximal torus $S^1 \times SO(2) \subset S^1 \times SO(3)$ whose lie algebra is denoted by $\lie{t}^{2}$ (where $SO(2)$ is identified with the rotations around the z-axis)  is given by
\[\begin{pmatrix}
1&k\\
0&1
\end{pmatrix}\]
The moment polytope associated to the maximal torus $T^{2}\subset K(2k)$ is thus the balanced polytope obtained from $\Delta(2k)$ by applying the inverse transpose $\begin{pmatrix}1&0\\-k&1\end{pmatrix}$
as shown in Figure~\ref{fig:PolytopesDifferentBases}~(c).

\subsection{Even isotropy representations and codimension calculation}
\label{Even isotropy representations}

Let $J_m$ be the standard $S^1$ invariant integrable almost complex structure in the strata $U_m$, coming from the Hirzebruch surface $W_m$. The action of the isometry group $K(2k)\simeq S^{1}\times \SO(3)$ on the space $H_{J_m}^{0,1}(\SSS,T(\SSS)) \cong \C^{m-1}$ (see \cite{Ko} Example 6.2(b)(4), p.309 for more details about how the isomorphism is obtained) of infinitesimal deformations is isomorphic to $\Det\otimes \Sym^{2k-2}$ , where  $\Det$ is the standard action of $S^{1}=U(1)$ on $\C^{2}$, and where $\Sym(\C^{2})$ is the representation $\mathscr{W}_{k-1}$ of $\SO(3)$ induced by the $(2k-2)$-fold symmetric product of the standard representation  of $\SU(2)$ on $\C^{2}$ (see Theorem 4.2 in\cite{AGK}). We shall denote this $(2k-2)$-fold symmetric product of the standard representation  of $\SU(2)$ on $\C^{2}$ as $\mathscr{V}_{2k-2}$. See \cite{B-tD} for more details about the representation theory of $\SO(3)$ and $\SU(2)$  \\

Let $R(t)$ denote the following  circle of $\SO(3)=\PU(2)=\U(2)/\Delta(S^{1})$
\[R(t)=\begin{pmatrix}1 & 0 & 0\\ 0 & \cos(t) & -\sin(t)\\ 0 & \sin(t) & \cos(t)\end{pmatrix}, ~t\in[0,2\pi)\]
where $\Delta(S^{1})$ denotes the subgroup of matrices of the form $\begin{pmatrix}e^{it} & 0\\ 0 & e^{it}\end{pmatrix}$.
\\

The circle $R(t)$  lifts to 
\[e(t/2):=\begin{pmatrix}e^{it/2} & 0\\ 0 & e^{-it/2}\end{pmatrix}\in\SU(2)\]

As explained above, the action of $K(2k)$ on  $H^{0,1}_{J_m}(S^2 \times S^2,T^{1.0}_{J_m}(S^2 \times S^2)) \cong \C^{m-1}$ is isomorphic to $\Det\otimes\Sym^{2k-2}$. Hence, to calculate the codimension we only need to calculate the dimension of the invariant subspace of $H^{0,1}(S^2 \times S^2,T^{1.0}_{J_m}(S^2 \times S^2)) \cong \C^{m-1}$ under the action of the subcircle $S^1(1,b;2k)\subset K(2k)$. To do so we note that a basis of $\Sym^{2k-2}$ is given by the homogeneous polynomials $P_{j}=z_{1}^{2k-2-j}z_{2}^{j}$ for $j\in\{0,\ldots,2k-2\}$. The action of $R(t)$ on $P_{j}$ is
\[R(t)\cdot P_{j}=e(t/2)\cdot P_{j}=e^{i\big(2k-2-2j\big)t/2}P_{j}=e^{it(k-1-j)}P_{j}\]
so that the action of $(e^{is},R(t))\subset S^{1}\times\SO(3)$ on $P_{j}$ is
\[\left(e^{is},R(t)\right)\cdot P_{j}=e^{i\big(s+t(k-1-j)\big)}P_{j}\]
Each $P_{j}$ generates an eigenspace for the action of the maximal torus $T(2k)$. In particular, the circle $S^{1}(a,b;2k)$ acts trivially on $P_{j}$ if, and only if, 
\[a+b(k-1-j)=(a,b)\cdot(1,k-1-j)=0\]
for $j\in\{0,2k-2\}$. Equivalently, we must have
\[a+bj=(a,b)\cdot(1,j)=0\]
for $j\in\{1-k,\ldots,k-1\}$. Hence, the dimension of the invariant subspace is given by the number of $j \in \{1-k,\ldots,k-1\}$ such that $a+bj=0$.
\\

Note that the above codimension calculation was performed with respect to the basis of the maximal torus in $K(2n)$. To express the codimension in terms of the standard basis, use the matrix $\begin{pmatrix}\frac{m}{2}& -1\\1&0\end{pmatrix}$ which takes the vector $\begin{pmatrix}1\\b\end{pmatrix}$ in the basis for the standard moment polytope of Figure~\ref{fig:PolytopesDifferentBases}~(a) to the vector $\begin{pmatrix}\frac{m}{2}-b\\1\end{pmatrix} $ in the basis for the balanced polytope of Figure~\ref{fig:PolytopesDifferentBases}~(c). In this basis, the codimension of $S^1(1,b;m)$ is given by the number of $j \in \{1-\frac{m}{2}, \cdots, \frac{m}{2}-1\}$ such that $(\frac{m}{2} - b) + j = 0$. Relabelling $j^\prime$ as $\frac{m}{2}+j$, we have that the codimension is given by 
the number of $j^\prime \in \{1, \cdots , m-1\}$ such that $j^\prime=b$.
 
\begin{thm}\label{codimension_calc}
Given the circle action $S^1(1,b,m)$ with $2\lambda > |2b-m|$ and $b \neq \{0, m\}$, the complex codimension of of the stratum $\jsom \cap U_m$ in $\jsom$ in given by the number of $j \in \{1, \cdots , m-1\}$ such that $j=b$. Similarly, for the action $S^1(-1,b,m)$ with $2\lambda > |2b+m|$ and $b \neq \{0, -m\}$, the complex codimension of of the stratum $\jsom \cap U_m$ in $\jsom$ in given by the number of $j \in \{1, \cdots , m-1\}$ such that $j=-b$.
\end{thm}
\begin{cor}
For the circle actions \begin{itemize}
  \item (i) $a=1$, $b \neq \{0, m\}$,  and $2\lambda > |2b-m|$; or
\item (ii) $a=-1$, $b \neq \{0, -m\}$, and $2 \lambda > |2b+m|$.
\end{itemize}The complex codimension of the stratum $\jsom \cap U_m$ in $\jsom$ is either 0 or 1.
\end{cor} 
\begin{proof}
Follows from the calculation and discussion above.
\end{proof}

\begin{remark}
In the beginning of this section, we worked in the $C^l$ category to show that the space $\jsoml \cap U_{2k,l}$ is a Banach submanifold. But in order to investigate the topology of the group $\Symp^{S^1}(\SSS,\oml)$ with $C^\infty$-topology, we require that the space $\jsom \cap U_{2k}$ with the $C^\infty$ topology is a Fr\'echet manifold and that the codimension of $\jsom \cap U_{2k}$ in $\jsom$ is given by the same formula as in Theorem \ref{codimension_calc}. As this discrepancy exists in the literature even in the non-equivariant case, and as a resolution of this issue is well beyond the scope of this document we do not attempt to resolve this here. 
\end{remark}

\chapter{Circle actions on \texorpdfstring{$\CCC$}{the non-trivial bundle}}\label{Chapter-CCC}

In this last section, we complete the determination of the homotopy type of the centralizer groups of $S^1$ actions on $\CCC$. This involves repeating the analysis done in sections~\ref{Section:homotopical descriptionofjsom} through~\ref{Even isotropy representations}, replacing $\SSS$ by $\CCC$ and circle actions on even Hirzebruch surfaces by actions on odd Hirzebruch surfaces. As the arguments are analogous in both cases, we shall only point out the key differences.

\section{Homotopy type of \texorpdfstring{$\Symp^{S^1}(\CCC,\oml)$}{Symp(CP2\#CP2}}
Recall that  the  odd Hirzebruch surface $W_m$, $m=2k+1$, is defined as a complex submanifold of $ \mathbb{C}P^1 \times \mathbb{C}P^2$ defined by setting
\[
W_m:=  \left\{ \left(\left[x_1,x_2\right],\left[y_1,y_2, y_3\right]\right) \in \mathbb{C}P^1 \times \mathbb{C}P^2 ~|~  {x^m}_1y_2 - x_2{y^m}_1 = 0 \right\}
\]
This manifold is diffeomorphic to $\CP^2 \# \overline{\CP^2}$ and is invariant under the toric action defined by 
\[
 \left(u,v\right) \cdot  \left(\left[x_1,x_2\right],\left[y_1,y_2, y_3\right]\right) = \left(\left[ux_1,x_2\right],\left[u^my_1,y_2,vy_3\right]\right)
\]
and whose moment map image is
\[
\scalebox{0.9}{
\begin{tikzpicture}
\node[left] at (0,2) {$Q=(0,1)$};
\node[left] at (0,0) {$P=(0,0)$};
\node[right] at (4,2) {$R= (\lambda - \frac{m+1}{2} ,1)$};
\node[right] at (6,0) {$S=(\lambda + \frac{m-1}{2} ,0)$};
\node[above] at (2,2) {$B-\frac{m+1}{2}F$};
\node[right] at (5,1) {$F$};
\node[left] at (0,1) {$F$};
\node[below] at (3,0) {$B+ \frac{m-1}{2}F$};
\draw (0,2) -- (4,2) ;
\draw (0,0) -- (0,2) ;
\draw (0,0) -- (6,0) ;
\draw (4,2) -- (6,0) ;
\end{tikzpicture}\label{oddHirz}}
\]
The homology class $B=[L]$ is the class of a line $L$ in $\CCC$, $E$ is the class of the exceptional divisor, and $F$ refers to the class $B - E$. There is a canonical form which we also call $\oml$ on $\CP^2 \# \overline{\CP^2}$, which has weight $\lambda$ on B and 1 on $F$. Note that with our convention, $\lambda$ must be strictly greater than 1 for the curve in class $E$ to have positive symplectic area.
As before all symplectic $S^1$ action on $\CP^2 \# \overline{\CP^2}$ extend to toric actions. The graphs for the different circles are given in Figures \ref{fig:GraphsWithFixedSurfaces} and \ref{fig:GraphsIsolatedFixedPoints}. As explained in Theorem~\ref{strata-CCC}, we have a stratification of the space of almost complex structures i.e the space $\jj_{\om_\lambda}$  of all compatible almost complex structures for the form $\om_\lambda$, decomposes into disjoint Fréchet manifolds of finite codimensions
\[
\jj_{\om_\lambda} = U_1 \sqcup U_3 \sqcup U_5\ldots \sqcup U_{2n-1}
\]
where 
\begin{multline*}
U_{2i-1} := \big\{ J \in \jj_{\om_\lambda}~|~ D_{2i-1}:= B-iF \in H_2(S^2 \times S^2,\Z)\\
\text{~is represented by a $J$-holomorphic sphere}\big\}.
\end{multline*}

We shall now use this stratification to construct fibrations for the action of equivariant symplectomorphism group on each equivariant stratum. We follow the same notation as in~\ref{not}. The proofs that the following maps are in fact fibrations is exactly the same as the proof given in Section~\ref{section:ActionOnU_2k} and we do not reproduce them.
\[\Stab^{S^1}(\overline{D}) \longrightarrow \Symp^{S^1}_{h}(\SSS,\oml) \longtwoheadrightarrow {\mathcal{S}^{S^1}_{D_{s}}} \mathbin{\textcolor{blue}{\xrightarrow{\text{~~~$\simeq$~~~}}}} \J^{S^1} \cap U_{2s+1}\rule{0em}{2em}\]

\[\Fix^{S^1}(\overline{D}) \longrightarrow \Stab^{S^1}(\overline{D}) \longtwoheadrightarrow  \Symp^{S^1}(\overline{D}) \mathbin{\textcolor{blue}{\xrightarrow{\text{~~~$\simeq$~~~}}}} S^1 ~\text{or}~ SO(3)\rule{0em}{2em}\]
   
\[\Fix^{S^1} (N(\overline{D})) \longrightarrow \Fix^{S^1}(\overline{D}) \longtwoheadrightarrow  \Aut^{S^1}(N(\overline{D})) \mathbin{\textcolor{blue}{\xrightarrow{\text{~~~$\simeq$~~~}}}} S^1 \rule{0em}{2em}\]
   
\[\Stab^{S^1}(\overline{F}) \cap  \Fix^{S^1}(N(\overline{D})) \longrightarrow \Fix^{S^1}(N(\overline{D})) \longtwoheadrightarrow  \overline{\mathcal{S}^{S^1}_{F,p_0}} \mathbin{\textcolor{blue}{\xrightarrow{\text{~~~$\simeq$~~~}}}} \mathcal{J}^{S^1}(\overline{D})\simeq \{*\}\rule{0em}{2em} \]
    
\[\Fix^{S^1}(\overline{F}) \longrightarrow \Stab^{S^1}(\overline{F}) \cap  \Fix^{S^1}(N(\overline{D})) \longtwoheadrightarrow  \Symp^{S^1}(\overline{F}, N(p_0))  \mathbin{\textcolor{blue}{\xrightarrow{\text{~~~$\simeq$~~~}}}} \left\{*\right\}\rule{0em}{2em}\]
    
\[\left\{*\right\} \mathbin{\textcolor{blue}{\xleftarrow{\text{~~~$\simeq$~~~}}}} \Fix^{S^1}(N(\overline{D} \vee \overline{F})) \longrightarrow \Fix^{S^1}(\overline{F}) \longtwoheadrightarrow  \Aut^{S^1}(N(\overline{D} \vee \overline{F})) \mathbin{\textcolor{blue}{\xrightarrow{\text{~~~$\simeq$~~~}}}} \left\{*\right\}\rule{0em}{2em}\]

When the $S^1(a,b;m) \subset \T_m$ action is of the form $(a,b)= (0,\pm1)$, the second fibration fits into the commutative diagram
\[
\begin{tikzcd}
     &\Fix^{S^1}(\overline{D}) \arrow[r] &\Stab^{S^1}(\overline{D}) \arrow[r,twoheadrightarrow] &\Symp^{S^1}(\overline{D}) \\
     &S^1 \arrow[u,hookrightarrow] \arrow[r] &U(2) \arrow[u,hookrightarrow] \arrow[r] &SO(3) \arrow[u,hookrightarrow]
\end{tikzcd}
\]
For all other actions with $(a,b) \neq (0,\pm1)$ we instead have 
\[
\begin{tikzcd}
     &\Fix^{S^1}(\overline{D}) \arrow[r] &\Stab^{S^1}(\overline{D}) \arrow[r,twoheadrightarrow] &\Symp^{S^1}(\overline{D}) \\
     &S^1 \arrow[u,hookrightarrow] \arrow[r] &\T_{2s+1} \arrow[u,hookrightarrow] \arrow[r] &S^1 \arrow[u,hookrightarrow]
\end{tikzcd}
\]
In both diagrams, the leftmost and the rightmost arrows are homotopy equivalences. As the diagrams above commute, the 5 lemma implies that the middle inclusions $\T \hookrightarrow \Stab^{S^1}(\overline{D})$  or $U(2) \hookrightarrow \Stab^{S^1}(\overline{D})$ are also weak homotopy equivalences.

\begin{thm}\label{homogenous_CCC}
Consider the $S^1(a,b;m)$ action on $(\CCC,\oml)$ with $\lambda >1$. If $ \jsom \cap U_{2s+1}$ is non-empty we have the following homotopy equivalences:
\begin{enumerate}
    \item if $(a,b) \neq (0,\pm1)$, then $\Symp^{S^1}_h(\CCC,\oml)/\T_{2s+1} \simeq \jsom \cap U_{2s+1}$;
    \item if $(a,b) = (0,\pm1)$, then $\Symp^{S^1}_h(\CCC,\oml)/U(2) \simeq \jsom \cap U_{2s+1}$.
\end{enumerate} \qed
\end{thm} 

As for circle actions on $\SSS$, the homotopy type of the symplectic centralizer $\Symp^{S^1}_h(\CCC,\oml)$ depends on whether the space of invariant almost complex structures intersects either one or two strata. Given the circle action $S^1(a,b;m)$ on $(\CCC,\oml)$, Corollaries~\ref{cor:IntersectingOnlyOneStratum} and \ref{cor:IntersectingTwoStrata} imply that
\begin{enumerate}
    \item if $a=1$, $b \neq\{0,m\}$ and $2\lambda > |2b-m|+1$, then the space of $S^1(1,b;m)$-equivariant almost complex structures $\jsom$ intersects the two strata $U_m$ and $U_{|2b-m|}$.
    \item If $a=-1$,$b \neq\{0,-m\}$ and $2\lambda > |2b+m|+1$, then the space of $S^1(-1,b;m)$-equivariant almost complex structures $\jsom$ intersects the two strata $U_m$ and $U_{|2b+m|}$.
    \item for all other cases $\jsom$ intersects only the stratum $U_m$.
\end{enumerate}  

As before, the homotopy type of the equivariant symplectomorphism group can be easily described whenever $\jsom$ intersects only one stratum. These actions are listed in Table~\ref{table:CentralizersCircleActionsOnCCCSingleStratum}.

\begin{thm}\label{CCC1strata}
Consider a circle action $S^1(a,b;m)$ on $(\CCC,\oml)$ for which there is a single invariant stratum. The homotopy type of the symplectic centralizer $\Symp^{S^1}(\CCC,\oml)$ is given in Table~\ref{table:CentralizersCircleActionsOnCCCSingleStratum}.\qed
\end{thm}

{\centering
\begin{table}[h!]
\begin{tabular}{|p{39mm}|p{47.35mm}|p{29.95mm}|}
\hline
Values of $(a,b ;m)$, $m$ odd & Conditions on $\lambda>1$ &$\!\Symp^{S^1}\!(\CCC)\!$\\
\hline
\hline
{$(0,\pm 1;m)$}  &$2\lambda > m+1$ & $\simeq\U(2)$ \\
\hline
$(\pm 1,0;m)$ &$ 2\lambda >m+1$ &$\simeq\T$\\
\hline
$(\pm 1,\pm m;m)$ &$ 2\lambda > m+1$ & $\simeq\T$\\
\hline
{$(1,b;m), b \neq \{m,0\}$} 
& $2\lambda > m+1$ and $|2b-m|+1 \geq2 \lambda$ &$\simeq\T$ \\
\hline
{$(-1,b;m), b \neq \{ -m,0\}$} 
&$|2b+m|+1 \geq2 \lambda$ &$\simeq\T$ \\
\hline
All other values of\newline $(a,b;m)$ except $(\pm 1,b;m)$&$ 2\lambda > m+1$  &$\simeq\T$ \\
\hline
\end{tabular}
\caption{Centralizers of circle actions on $(\CCC,\oml)$ with a single invariant stratum}
\label{table:CentralizersCircleActionsOnCCCSingleStratum}
\end{table}
}

Theorem~\ref{CCC1strata} gives the homotopy type of the group of equivariant symplectomorphisms for all circle actions apart from the following two exceptional families for which $\jsom$ intersects exactly two strata:
\begin{itemize}
\item (i) $a=1$, $b \neq \{0, m\}$,  and $2\lambda > |2b-m|+1$; or
\item (ii) $a=-1$, $b \neq \{0, -m\}$, and $2 \lambda > |2b+m|+1$.
\end{itemize}
For these two families of actions, we now proceed as in Chapter~4 and show that the centralizers are homotopy equivalent to homotopy pushouts of tori.\\

Firstly, we show that the map $\T_m \hookrightarrow \Symp^{S^1}_h(\CCC,\oml)$ induces a map that is injective in homology. Fix a curve $\overline{F}$ in the homology class $F$ and passing through the fixed points $Q$ and $P$ in figure~\ref{oddHirz}. Let $\mathcal{S}^{S^1}_{F,Q}$ denote the space of $S^1$-invariant curves in the class $F$ passing through $Q$ (defined in figure~\ref{oddHirz}) and let $\Symp^{S^1}_h(\CCC,\overline{F},\oml)$ denote the space of $S^1$-equivariant symplectomorphisms of $(\CCC,\oml)$ that pointwise fix the curve $\overline{F}$.
\\

Without loss of generality, assume that $\jsom \cap U_{|2b-m|}$ is the strata of positive codimension in $\jsom$. As $\jsom$ is contractible, $\jsom \cap U_m = \jsom \setminus \left(\jsom \cap U_{|2b-m|}\right)$, and the real codimension of $\jsom \cap U_{|2b-m|}$ in $\jsom$ is 2 (See Corollary~\ref{odd_codim}), it follows that $\jsom \cap U_m \simeq \Symp^{S^1}_h(\CCC,\oml)/T_m$ is connected. As a consequence we have that $\Symp^{S^1}_h(\CCC,\oml)$ is connected. As the fixed points for the $S^1(\pm 1,b,m)$ actions are isolated and as $\Symp^{S^1}_h(\CCC,\oml)$ is connected, any element $\phi \in \Symp^{S^1}_h(\CCC,\oml)$ takes a fixed point for the $S^1$ action to itself. Thus the action of $\Symp^{S^1}_h(\CCC,\oml)$ on $\mathcal{S}^{S^1}_{F,Q}$ is well defined.

\begin{lemma}\label{odd_inj}
The inclusion
\[i: \Symp^{S^1}_h(\CCC,\overline{F},\oml) \hookrightarrow \Symp^{S^1}_h(\CCC,\oml)\]
is a homotopy equivalence. 
\end{lemma}
\begin{proof}
Consider the fibration
\begin{equation*}
    \Symp^{S^1}_h(\CCC,\overline{F},\oml) \hookrightarrow \Symp^{S^1}_h(\CCC,\oml) \longtwoheadrightarrow \mathcal{S}^{S^1}_{F,Q}
\end{equation*}
To show that the action $\Symp^{S^1}_h(\CCC,\oml)$ on $\mathcal{S}^{S^1}_{F,Q}$ is transitive we note that given $F^\prime \in \mathcal{S}^{S^1}_{F,Q}$, there exists a $J^\prime \in \jsom$ such that $F^\prime$ is $J^\prime$-holomorphic. As $\jsom$ is connected, consider a path $J_t$ such that $J_0=J^\prime$ and $J_1=J_m$ where $J_m$ is the standard complex structure on the $m^\text{th}$-Hirzebruch surface for which the curve $\overline{F}$ is holomorphic.
\\

By Theorem~\ref{prop_FIndecomposable}, for every $J_t$ we have a family of curves $F_t$ (with $F_0= F^\prime$ and $F_1 =\overline{F}$) in class $F$ passing through $Q$  and this curve is $S^1$ invariant as $J_t$ are $S^1$ invariant. By Lemma~\ref{Au} we have a one parameter family of Hamiltonian symplectomorphisms $\phi_t \in \Symp^{S^1}_h(\CCC,\oml)$  such that $\phi_t(F_0) = F_t$ for all $t$.

Thus it suffices to show that $\mathcal{S}^{S^1}_{F,Q}$ is contractible to complete the proof. To do this note that,
\begin{equation*}
  \jsom(\overline{F})  \to \jsom \to \mathcal{S}^{S^1}_{F,Q}
\end{equation*}
where $\jsom(\overline{F})$ denotes the space of $S^1$ invariant almost complex structures for which the curve $\overline{F}$ is $J$-holomorphic. As both $\jsom(\overline{F})$ and $\jsom$ are contractible, $\mathcal{S}^{S^1}_{F,Q}$ is contractible as well completing the proof.
\end{proof}

Define the following projections just as in the exposition above Theorem~\ref{inj}. In our case we take the point $\{*\}$ to be the point $Q$
in Figure~\ref{oddHirz}.

\begin{align*}
    s_{0}: S^2 &\to W_m \\
    \left[z_{0}, z_{1}\right] &\mapsto\left(\left[z_{0}, z_{1}\right],[0,0,1]\right)
\end{align*} 
and the projection to the first factor of $\CP^1 \times \CP^2$ is
\begin{align*}
    \pi_1: W_m &\to S^2 \\
    \left(\left[z_{0}, z_{1}\right],\left[w_{0}, w_{1}, w_{2}\right]\right) &\mapsto\left[z_{0}, z_{1}\right]
\end{align*}
We define a continuous map $h_1:\Symp_h^{S^1} (\CCC,\oml) \to \mathcal{E}\left(S^{2}, *\right)$ by setting
\begin{equation*}
  \begin{aligned}
h_1:\Symp_h^{S^1} (\CCC,\oml) &\to \mathcal{E}\left(S^{2}, *\right) \\
\psi &\mapsto \psi_1:= \pi_{1} \circ \psi \circ s_{0}
\end{aligned}
\end{equation*}
Further define the restriction map $r: \Symp_h^{S^1} (\CCC,\overline{F},\oml) \to \mathcal{E}\left(S^{2}, *\right)$ by just restricting $\phi \in \Symp_h^{S^1} (\CCC,\overline{F},\oml)$ to the fibre $\overline{F}$. We obtain a well defined map 
\begin{align*}
    h: \Symp_h^{S^1} (\CCC,\overline{F},\oml) &\to \mathcal{E}\left(S^{2}, *\right) \times \mathcal{E}\left(S^{2}, *\right) \\
    \phi &\mapsto (h_1(\phi), r(\phi))
\end{align*}

\begin{thm}
The inclusion map $i:\T_m \hookrightarrow \Symp^{S^1}_h(\CCC,\oml)$ induces a map that is injective in homology.
\end{thm}

\begin{proof} By Lemma~\ref{odd_inj}, it suffices to prove that the inclusion $i:\T_m \hookrightarrow \Symp^{S^1}_h(\CCC,\overline{F},\oml) $ induces a map that is injective in homology. Composing with $h$ we have an inclusion of $h \circ i:\T_m \hookrightarrow \mathcal{E}\left(S^{2}, *\right) \times \mathcal{E}\left(S^{2}, *\right)$ and it suffices to show that this map induces a map that is injective in homology. The proof of this claim in analogous to the proof of Theorem~\ref{inj}.
\end{proof}
\begin{remark}
As the argument above doesn't depend on $m$, the same proof shows that for the family of circle actions given by $S^1(1,b,m)$ with $2\lambda > |2b-m|+1$, the inclusion $\T_{|2b-m|}$ into $\Symp^{S^1}_h(\CCC,\oml)$ is injective in homology. Similarly, the $S^1(-1,b;m)$ actions with $2\lambda > |2b+m|+1$, the inclusion $\T_{|2b+m|}$ and $T_m$ into $\Symp^{S^1}_h(\CCC,\oml)$ also induces a map that is injective in homology.\\
\end{remark}

From our discussion above we know that 
\begin{align*}
U^{S^1}_m:=\jsom \cap U_m&\simeq \Symp^{S^1}_h(\CCC,\oml)/\T_m\\
U^{S^1}_{|2b-m|}:=\jsom \cap U_{|2b-m|}&\simeq \Symp^{S^1}_h(\CCC,\oml)/\T_{|2b-m|}.
\end{align*}
and the Leray-Hirsch theorem implies that, over any field $k$, we have isomorphisms of $k$-modules
\begin{align*}
H^*(\Symp^{S^1}_h(\CCC,\oml)) &\cong H^*(U^{S^1}_{m}) \bigotimes H^*(\T) \\
H^*(\Symp^{S^1}_h(\CCC,\oml)) &\cong H^*(U^{S^1}_{|2b-m|}) \bigotimes H^*(\T)
\end{align*}
The total space $\jsom = U^{S^1}_{m}\sqcup U^{S^1}_{|2b-m|}$ being contractible, there is an Alexander-Eells  isomorphism $\lambda_{*}:H_{p}(U_{m'}^{S^1};k)\to H_{p+d}(U_{m}^{S^1};k)$ 
where $d$ is the codimension of the stratum $U^{S^1}_{|2b-m|}$ in $\jsom$. As we will show in Corollary~\ref{odd_codim} below, this codimension is again two. Proceeding as in Section~\ref{subsection:CohomologyModule}, we can then compute the ranks of the homology groups of the space of equivariant symplectomorphisms.

\begin{thm}\label{thm:CohomologyModuleCCC} Consider the following circle actions on  $\CCC$:
\begin{itemize}
\item (i) $a=1$, $b \neq \{0, m\}$,  and $2\lambda > |2b-m|+1$; or
\item (ii) $a=-1$, $b \neq \{0, -m\}$, and $2 \lambda > |2b+m|+1$.
\end{itemize} 
For any field $k$,  
\[
H^p\left(\Symp^{S^1}(\CCC,\oml), k\right) = 
\begin{cases}
k &p=0\\
k^3 &p =1 \\
k^4 &p \geq 2
\end{cases}
\]
\end{thm}

Finally, the analysis of the homotopy pushout $(T_m\from S^1\to T_{m'})$ leading to Theorem~\ref{full_homo} holds verbatim for the $S^1(a,b;m)$ actions on $\CCC$ satisfying the conditions of Theorem~\ref{thm:CohomologyModuleCCC} and shows that $\Symp^{S^1}(\CCC,\oml)\simeq \Omega S^3 \times S^1 \times S^1 \times S^1$.\\

\begin{thm}\label{full_homo_CCC}
Consider a Hamiltonian circle action on $(\CCC, \oml)$. The symplectic centralizer $\Symp^{S^1}(\CCC,\oml)$ is given in Table~\ref{table:CentralizersAllCircleAvtionsOnCCC}.
\qed
\end{thm}
{\centering
\begin{table}[H]
\begin{tabular}{|p{39mm}|p{48.75mm}|p{28.5mm}|}
 \hline
 Values of $(a,b ;m)$, $m$ odd & Conditions on $\lambda >1$ &\!$\Symp^{S^1}\!\!(\CCC)$\!\!\\
\hline
\hline
{$(0,\pm 1;m)$}  &$2\lambda > m+1$ & $\simeq \U(2)$ \\
\hline
 $(\pm 1,0;m)$ &$2\lambda > m+1$ &$\simeq \T$\\
 \hline
$(\pm 1,\pm m;m)$ &$2\lambda > m+1$ & $\simeq \T$\\
\hline
\multirow{2}{10em}{$(1,b;m)$, $b \neq \{m,0\}$} 
&$2\lambda > m+1$ and $|2b-m|+1 \geq2 \lambda$ &$\simeq \T$ \\\cline{2-3}
&$2\lambda > m+1$ and $2 \lambda >|2b-m|+1 $  &$\simeq \Omega S^3 \times \mathbb{T}^3$ \\
 \hline
 \multirow{2}{10em}{$(-1,b;m)$,\newline $b \neq \{ -m,0\}$} 
&$2\lambda > m+1$ and $|2b+m|+1 \geq2 \lambda $ &$\simeq \T$ \\\cline{2-3}
&$2\lambda > m+1$ and $2 \lambda >|2b+m|+1$ &\!$\simeq \Omega S^3 \times \mathbb{T}^3$\\
\hline
All other values $(a,b;m)$ &$ 2\lambda > m+1$ &$\simeq \T$ \\
\hline
\end{tabular}
\caption{Centralizers of circle actions on $(\CCC,\oml)$}
\label{table:CentralizersAllCircleAvtionsOnCCC}
\end{table}
}

\begin{remark}
Recall that in the above table for the  $S^1$ action $S^1(a,b;m)$ for $m \neq 0$ on $(\CCC,\oml)$, to be Hamiltonian we need the condition $\lambda > \frac{m+1}{2}$.
\end{remark}

\section{Isometry groups of odd Hirzebruch surfaces}\label{Isom_CCC}

In this section we calculate the codimension of the smaller strata in $\jsom$. Let $M$ denote the manifold $\CCC$. By Theorem~\ref{trans_strata}, we know that the normal bundle to the strata $U_{2s+1} \cap \jsom$ can be identified with $H^{0,1}_{J_{2s+1}}(M,T^{1,0}_{J_{2s+1}}M)^{S^1}$. Thus to calculate the codimension of $U_{2s+1} \cap \jsom$ in $\jsom$, we need to understand how $S^1$ acts on $H^{0,1}_{J_{2s+1}}(M,T^{1,0}_{J_{2s+1}}M)$ and calculate the dimension of the invariant subspace. We use the convention $m = 2k+1$.
\\

The Hirzebruch surface $\mathbb{F}_{2k+1}$ is obtained by K\"ahler reduction of $\C^{4}$ under the action of the torus $T^{2}_{2k+1}$ defined by
\[(s,t)\cdot z = (s^{2k+1}tz_{1},tz_{2},sz_{3},sz_{4})\]
The moment map is $\phi(z)=((2k+1)|z_{1}|^{2}+|z_{3}|^{2}+|z_{4}|^{2}, |z_{1}|^{2}+|z_{2}|^{2})$ and the reduced manifold at level $(\lambda +k,1)$ is symplectomorphic to $(\CP^2 \# \overline{\C}P^2,\om_{\lambda})$ and biholomorphic to the Hirzebruch surface $\mathbb{F}_{2k+1}$. In this model, the projection to the base is given by $[(z_{1},\ldots,z_{4})]\mapsto [z_{3}:z_{4}]$, the zero section is $[w_{0}:w_{1}]\mapsto [(w_{0}^{2k+1},0,w_{0},w_{1})]$, and a fiber is $[w_{0}:w_{1}]\mapsto [(w_{0}w_{1}^{2k+1},w_{0}w_{1},0,w_{1})]$. The torus $T^{2}(2k+1)=T^{4}/T^{2}_{2k+1}$ acts on $\mathbb{F}_{2k+1}$. This torus is generated by the elements $[(1,e^{it},1,1)]$ and $[(1,1,e^{is},1)]$, and its moment map is $[(z_{1},z_{2},z_{3},z_{4})]\mapsto(|z_{2}|^{2},|z_{3}|^{2})$. The moment polytope $\Delta(2k+1)$ is the convex hull of the vertices $(0,0)$, $(1,0)$, $(1,\lambda+k)$, and~$(0,\lambda-k -1)$.
\\

The isometry group of $\mathbb{F}_{2k+1}$ is 
\[K(2k+1) = Z_{\U(4)}(T^{2}_{2k+1})/T^{2}_{2k+1}=(T^{2}\times \U(2))/T^{2}_{2k+1}\simeq \U(2)\]
where the last isomorphism is given by
\[[(s,t), A]\mapsto (s^{-1}t\det A^{k}) A\]
Under this isomorphism, an element $[(1,a,b,1)]$ of the torus $T(2k+1)$ is taken to
\[ab^{k}\begin{bmatrix}b&0\\0&1\end{bmatrix}\]
Consequently, at the Lie algebra level of the maximal tori, the map $\lie{t}^{2}(2k+1)\to \lie{t}^{2}$ is given by
\[\begin{pmatrix}
1&k+1\\
1&k
\end{pmatrix}\]
The moment polytope associated to the maximal torus $T^{2}\subset K(2k+1)$ is thus  the balanced polytope obtained from $\Delta(2k+1)$ by applying the inverse transpose $\begin{pmatrix}-k&1\\k+1&-1\end{pmatrix}$.

\subsection{Odd isotropy representations and codimension calculation}\phantom{Th}
The action of the isometry group $K(2k+1)\simeq \U(2)$ on the space $H_{J}^{0,1}(M,TM)$ of infinitesimal deformations is isomorphic to $\Det^{1-k}\otimes \Sym^{2k-1}$, where $\Det$ is the determinant representation of $\U(2)$ on $\C$, and where $\Sym^{k}(\C^{2})$ is the $k$-fold symmetric product of the standard representation of $\U(2)$ on $\C^{2}$. Using the double covering $S^{1}\times\SU(2)\to\U(2)$, we see that irreducible representations of $\U(2)$ correspond to irreducible representations of $S^{1}\times\SU(2)$ for which $(-1,-\id)$ acts trivially. If $A_{n}$ denotes the representation $t\cdot z=t^{n}z$ of $S^{1}$ on $\C$, and if $V_{k}$ is the $k$-fold symmetric product of the defining representation of $\SU(2)$ on $\C^{2}$, then the irreducible representations of $\U(2)$ are $A_{n}\otimes V_{k}$ with $n+k$ even. In this notation, we have the identifications $\Det = A_{2}$, while $\Sym=A_{1}\otimes V_{1}$. Consequently, $\Det^{1-k}\otimes \Sym^{2k-1} = A_{1}\otimes V_{2k-1}$.

With respect to the double covering $S^{1}\times\SU(2)\to\U(2)$, the maximal torus $T^{2}\subset\U(2)$ of diagonal matrices $D_{s,t}:=\begin{pmatrix}e^{is} & 0\\ 0 & e^{it}\end{pmatrix}$ lifts to
\[\left(|D_{s,t}|^{1/2},\frac{D}{|D|^{1/2}}\right)
=\left(e^{i(s+t)/2},\begin{pmatrix}e^{i(s-t)/2} & 0\\ 0 & e^{i(t-s)/2}\end{pmatrix}\right)\]
As explained above, the induced action of the K\"ahler isometry group $K(2k+1)$ on $H^{0,1}(\CCC,T^{1.0}_{J_m}(\CCC)) \cong \C^{m-1}$ is isomorphic to  $\Det^{1-k}\otimes\Sym^{2k-1}$. Hence to calculate the the codimension we only need to calculate the dimension of the invariant subspace of the vector space $ H^{0,1}(\CCC,T^{1.0}_{J_m}(\CCC)) \cong \C^{m-1}$ under the $S^1(1,b;m)$ action. To do so we note that a basis of $\Sym^{2k-1}$ is given by the homogeneous polynomials $P_{j}=z_{1}^{2k-1-j}z_{2}^{j}$ for $j\in\{0,\ldots,2k-1\}$. The action of $D_{s,t}$ on $P_{j}$ is
\[D_{s,t}\cdot P_{j}=e^{i\big((s+t)(1-k)+s(2k-1-j)+tj\big)}P_{j}\]
so that each $P_{j}$ generates an eigenspace for the action of the maximal torus $T(2k+1)$ generated by $D_{s,t}$. In particular, the circle $S^{1}(a,b;2k+1)$ acts trivially on $P_{j}$ if, and only if, 
\[(a-b)(k-j)+b=(a,b)\cdot(k-j,j-k+1)=0\]
Thus the codimension (in the balanced basis of the maximal torus of K(2k+1) is given by the number of $j \in \{0,\ldots,2k-1\}$ such that \begin{equation}\label{formula_CCC}
    (a-b)(k-j)+b=(a,b)\cdot(k-j,j-k+1)=0.
\end{equation}
Note that just as in the $\SSS$ case, the above codimension calculation was with respect to the basis of the maximal torus in $K(2k+1)$. Hence to calculate the codimension for the $S^1(1,b,m) \subset \T_m$ as in our case, we need to transform the basis by multiplication by the matrix $\begin{pmatrix}\frac{m-1}{2}+1& -1\\\frac{m-1}{2}&-1\end{pmatrix}$. Thus it takes the vector $\begin{pmatrix}1\\b\end{pmatrix}$ in the basis for the standard moment polytope 

to the vector $\begin{pmatrix}\frac{m+1}{2}-b\\\frac{m-1}{2}-b\end{pmatrix}$ in the basis for the balanced polytope. Hence $a$ and $b$ in equation~\ref{formula_CCC} above need to be replaced by $\frac{m+1}{2}-b$ and $\frac{m-1}{2}-b$ respectively to get the correct codimension for the $S^1(1,b,m)$ action. 

\begin{thm}\label{codimension_calc2}
Given the circle action $S^1(1,b,m)$ on $(\CCC,\oml)$ with $2\lambda > |2b-m|+1$ and $b \neq \{0, m\}$, the complex codimension of of the strata $\jsom \cap U_m$ in $\jsom$ in given by the number of $j \in \{1, \cdots , m-1\}$ such that $j=b$.
\\

Similarly for the action $S^1(-1,b,m)$ with $2\lambda > |2b+m|+1$ and $b \neq \{0, -m\}$, the complex codimension of of the strata $\jsom \cap U_m$ in $\jsom$ in given by the number of $j \in \{1, \cdots , m-1\}$ such that $j=-b$.
\qed
\end{thm}

\begin{cor}\label{odd_codim}
For the circle actions \begin{itemize}
  \item (i) $a=1$, $b \neq \{0, m\}$,  and $2\lambda > |2b-m|+1$; or
\item (ii) $a=-1$, $b \neq \{0, -m\}$, and $2 \lambda > |2b+m|+1$.
\end{itemize}The complex codimension of the stratum $\jsom \cap U_m$ in $\jsom$ is either 0 or 1.
\qed
\end{cor}

\appendix

\chapter[Equivariant Differential Topology]{Equivariant versions of classical results from Differential Topology}\label{ElementaryEquivariantTopology}

In this appendix we present equivariant versions of standard results in differential and symplectic topology. We either provide proofs or give precise references. For a classical exposition of smooth equivariant geometry, one may consult Bredon~\cite{Bredon}. For a discussion of normal forms theorems along symplectic submanifolds that is relevant to our purpose, we recommand Guadagni~\cite{Guadagni}.\\

Throughout this section, we consider a symplectic manifold $(M,\om)$ on which a compact Lie group $G$, possibly finite, is acting symplectically, leaving invariant a submanifold $S$. 

\subsection*{Invariant almost complex structures}
\begin{lemma}\label{contractibilityofinvariantacs}
Let $G$ be a compact group acting symplectically on $(M,\omega)$. Then the space $\mathcal{J}^G_\om$ of  $G$-invariant compatible almost complex structures on $(M,\om)$ is contractible.
\end{lemma}

\begin{proof}
Let $\Met^G$ denote the space of $G$-invariant metrics on $M$. As $\Met^G$ is an affine space, it is contractible. We construct a homotopy equivalence $\alpha:\Met^G\to \mathcal{J}^G_\om$ using the polar decomposition. Choose an $G$-invariant metric $g$ on $M$. Then there exists a fiberwise automorphism $A$ of the tangent bundle such that $g(u,v) = \om(u,Av)$ for all $u,v \in TM$. The group $G$ acts on bundle automorphisms by conjugation, and both $A$ and its adjoint $A^\dagger$ are invariant. By construction, the symmetric operator $A^\dagger A$ is positive definite. Define $P=(A^\dagger A)^{1/2}$ as the unique positive definite square root of $A^\dagger A$. Define $\alpha(g) = AP^{-1}$. One can check that $\alpha(g)^2 = - \id$ and that $\alpha(g)$ is compatible with $\om$. In order to show that $\alpha(g)$ is $G$ invariant, we observe that for all $h\in G$, $(h\cdot P)^2=A^\dagger A$, that is, the action of $G$ permutes the square roots of $A^\dagger A$. Since $h \cdot P$ is also positive definite, we conclude that $h\cdot P=P$ for all $h\in G$. It follows that $\alpha(g)$ is a $G$-invariant compatible almost complex structure. 
\\

We now claim that the map $ \beta: \mathcal{J}^G_\om\to \Met^G$ that sends $J$ to the compatible metric $\beta(J) := \om(\cdot,J\cdot)$ is a homotopy inverse of the map $\alpha:\Met^G\to \mathcal{J}^G_\om$ constructed above. Indeed, the composition $\alpha \circ \beta: \mathcal{J}^G_\om \rightarrow \mathcal{J}^G_\om$ is equal to the identity, while the map $\beta \circ \alpha: \Met^G \rightarrow \Met^G $ is homotopy equivalent to the identity as $\Met^G$ is contractible.
\end{proof}

\begin{lemma}\label{Lemma:Contractibilityofacs}
Let $G$ be a compact group acting symplectically on $(M,\omega)$. Let $S$ be a $G$ invariant symplectic submanifold. Then the space $\mathcal{J}^G_\om(S)$ of $G$-invariant compatible almost complex structures on $(M,\om)$ for which $S$ is $J$-holomorphic is contractible.
\end{lemma}
\begin{proof}
Let $\mathfrak{J}^G_{\om|_S}(S)$ denote the space of compatible almost complex structures on $S$ compatible with respect to the restricted symplectic form $\om|_S$ on $S$.  Then the restriction map $r: \mathcal{J}^G_\om(S)\to\mathfrak{J}^G_{\om|_S}(S)$ is a fibration. Fix a $J_0 \in \mathfrak{J}^G_{\om|_S}(S)$. Let $g_0$ denote the compatible metric on $S$ given by $\om|_S$ and $J_0$. The fibre of $r: \mathcal{J}^G_\om(S)\to\mathfrak{J}^G_{\om|_S}(S)$ at $J_0$ is given by the set of $J\in\mathcal{J}^G_\om(S)$ such that $J|_S = J_0$. This space of almost complex structures is homeomorphic to the space of $\om$ compatible metrics on $(M,\om)$ such that the restriction of the metric to $S$ is $g_0$. The space of such metrics is an affine space and hence contractible. By Lemma~\ref{contractibilityofinvariantacs} $\mathfrak{J}^G_{\om|_S}(S)$ is also contractible this completes the proof.
\end{proof}

\subsection*{Invariant neighborhoods and equivariant isotopies}

\begin{thm}[G-equivariant tubular neighborhood]\label{thm:EquivariantNhoods} Let $G$ be a compact Lie group acting smoothly on $M$ and let $S\subset M$ be a closed invariant submanifold.
\begin{enumerate}
\item Any two (open or closed) invariant tubular neighborhoods of $S$ are equivariantly isotopic.
\item If the invariant submanifold $S$ is compact, then any two closed invariant tubular neighborhoods of $S$ are ambient isotopic through an equivariant isotopy that is the identity outside a compact neighborhood of $S$.
\end{enumerate}
\end{thm}
\begin{proof}
The first part is Theorem~VI.2.6 in~Bredon~\cite{Bredon}. As explained in~\cite{Bredon} Remark p. 312, the second part follows by combining Part (1) with the results of Section~VI.3 in~\cite{Bredon}. 
\end{proof}

Various $G$-equivariant neighborhood and isotopy theorems in the symplectic category are based on the following two lemmata.

\begin{lemma}[Relative Poincaré lemma with parameters]\label{lemma:RelativePoincare}
Let $M$ be a smooth manifold and let $\theta_t$ be a $1$-parameter family of $p$-forms that are identically zero in a neighborhood of a closed set $F$, and whose relative cohomology classes $[\theta_t]$ vanish in $H^p(M,F)$. Then there exists a $1$-parameter family of $(p-1)$-forms $\beta_t$, vanishing in a neighborhood of $F$, such that $d\beta_t = \theta_t$. Moreover, if the forms $\theta_t$ are invariant under the action of a compact group $G$, and if the set $F$ is invariant, then the primitives $\beta_t$ can be chosen to be $G$-invariant.
\end{lemma}
\begin{proof}
The non-equivariant case is Lemma II.2.2 in~\cite{Banyaga78}. The equivariant version follows from a simple averaging argument.
\end{proof}

\begin{lemma}[Equivariant Moser's argument]\label{lemma:EquivariantMoser}
Let $\omega_t$, $t\in[0,1]$, be a smooth family of cohomologous symplectic forms on $M$. Suppose a compact group $G$ acts on $M$ and that the action is hamiltonian with respect to each form $\om_t$. Assume that there exists a smooth family of $G$-invariant $1$-forms $\{\sigma_t\}_{t\in [0,1]}$ such that $\dd{t}\omega_t=d\sigma_t$. Then there exists a $G$-equivariant diffeotopy $\psi_t:M\rightarrow M$ such that $\psi_t^*(\omega_t)=\omega_0$, $\psi_0=\id$, and $\psi_t=\id$ on the zero set $Z=\{m\in M~|~\sigma_t(m)=0~\forall t\in[0,1]\}$.
\end{lemma}

\begin{proof}
Since all the forms involved are invariant, the standard proof of Moser's argument gives an equivariant diffeotopy, see Section 3.2 in~\cite{MS}.
\end{proof}

\begin{lemma}[Equivariant symplectic neighborhoods theorem]\label{EqSymN}
Let $G$ be a compact group, and let $(M_{i},\om_{i})$, $i=0,1$, be two symplectic $G$-manifolds. Let $S_{i}\subset M_{i}$ be two invariant symplectic submanifolds with invariant symplectic normal bundles $N_{i}$. Suppose that there is an equivariant isomorphism $A:N_{0}\to N_{1}$ covering an equivariant symplectomorphism $\phi:S_{0}\to S_{1}$. Then $\phi$ extends to a equivariant symplectomorphism of neighborhoods $\Phi:U_{0}\to U_{1}$ whose derivative along $S_{0}$ is equal to $A$.
\end{lemma}
\begin{proof}
This is Theorem~3.30 in~\cite{MS}. It can be made equivariant using Lemma~\ref{lemma:RelativePoincare} and Lemma~\ref{lemma:EquivariantMoser}.
\end{proof}

Given a principal $G$-bundle $P$ over a symplectic base $(S,\om_S)$, a principal connection $1$-form $\alpha$ on $P$, and a Hamiltonian action of $G$ on a symplectic manifold $(F,\om_F)$, the minimal coupling construction of Sternberg and Weinstein defines a symplectic form $\om_\alpha$ on the associated bundle $P\times_G F$. We recall this construction in the special case of normal bundles of symplectic submanifolds, in a form we borrow from Guadagni~\cite{Guadagni}. The general case can be found in Weinstein~\cite{Weinstein78}.

\begin{prop}[Minimal coupling construction on normal bundles]\label{prop:CouplingConstruction}
Let $(S,\om_S)$ be a symplectic submanifold of codimension $2k$ in $(M,\om)$. Let $N$ be its symplectic normal bundle, and let $P$ be the associated principal $U(k)$ bundle. Let $\alpha$ be a connection form on $P$, and let $F_\alpha$ be its curvature $2$-form. Let $\mu$ be the moment map of the standard $\U(k)$ action on $\C^k$. The $2$-form
\[\omega_S+d\langle F_\alpha, \mu\rangle +\omega_{\C^k}\]
on $P\times\C^k$ descends to a symplectic form $\om_\alpha$ on the normal bundle $N=P\times_{U(k)}\C^k$. Moreover, the form $\om_\alpha$ extends the fiberwise symplectic form on $N$ and is equal to $\om_M$ along the zero section.
\end{prop}
\begin{proof}
This is Lemma 2 and Corollary 2.2 in~\cite{Guadagni}.
\end{proof}
Together with the Equivariant symplectic neighborhoods Theorem~\ref{EqSymN}, the the minimal coupling construction provides a normal form theorem.

\begin{thm}[Equivariant symplectic tubular neighborhood]\label{thm:EquivariantSymplecticTubularNhood} Let $G$ be a compact group acting symplectically on $(M,\om)$. Let $S$ be an invariant symplectic submanifold. Consider the linearized action $G$ on the normal bundle $N$ endowed with the minimal coupling form $\om_\alpha$. Then there is a $G$-equivariant symplectomorphism $\phi:U\rightarrow M$, defined in a neighborhood $U\subset N$ of the zero section, whose derivative $d\phi$ along $S$ is the identity.
\end{thm}
\begin{proof}
Choose a $G$-invariant and compatible metric on $M$. Then the exponential map $\exp: N\to M$ is equivariant and $\exp^*(\omega)=\omega$ along $S$. Choose any $G$-invariant connection $\alpha$ on $N$ and consider the induced symplectic form $\omega_\alpha$. Since $\exp^*(\omega)=\omega_\alpha$ along $S$, Lemma~\ref{EqSymN} implies the existence of a $G$-equivariant embedding $\phi:U\to N$ defined in a neighborhood of the zero section of $N$ and  such that $\phi^*(\exp^*(\omega))=\omega_\alpha$. By construction, the composition $\exp\circ\phi: U\to M$ is the desired symplectomorphism.
\end{proof}

Although it is always possible to extend a smooth isotopy defined in a neighborhood of a closed set $F\subset M$, there are obstructions to extend symplectic isotopies. 

\begin{thm}[Banyaga's extension theorem]\label{thm:BanyagaExtension}
Let $(M,\om)$ be a symplectic manifold, let $F\subset M$ be a closed subset, and let $h_t:M\to M$ be a smooth isotopy with compact support which is symplectic in a neighborhood of $F$. Suppose that the relative cohomology classes $[h_t^*\om-\om]$ are zero in $H^2(M,F)$. Then there exists a symplectic isotopy $g_t:M\to M$ with compact support which coincides with $h_t$ in a neighborhood of $F$. Moreover, if a compact group $G$ is acting symplectically on $M$, if $F$ is invariant, and if the isotopy $h_t$ is equivariant, then $g_t$ is also equivariant.
\end{thm}
\begin{proof} We prove the equivariant statement. By assumption, the closed forms $\dot{\omega}_{t}:=\frac{\del}{\del t}h_{t}^{*}\omega$ satisfy the conditions of the relative Poincaré lemma~\ref{lemma:RelativePoincare} so that we can find a $1$-parameter family of invariant $1$-forms $\sigma_t$ such that $d\sigma_t=\dot{\omega}_{t}$ and which restricts to zero on $F$. Using Moser's argument~\ref{lemma:EquivariantMoser}, there is an equivariant isotopy $g_t$ with compact support, which is the identity on $F$, and such that $g_t^*\om_t=\om_0$. The composition $h_t\circ g_t$ is then an equivariant symplectic isotopy that coincide with $h_t$ on $F$.
\end{proof}

In the case one is only interested in realising a displacement of unparametrised symplectic submanifolds by an ambient symplectic isotopy, then there are no obstructions.

\begin{lemma}[Displacement of unparametrised symplectic submanifolds]\label{Au} Let $G$ be a compact group acting symplectically on a compact manifold $(M,\om)$. Let $W_t$, $t \in [0,1]^k$, be a smooth $k$-parameter family of symplectic submanifolds which are invariant under the $G$ action. Then there exists a  $k$-parameter family of equivariant Hamiltonian symplectomorphisms $\phi_{t}: M \rightarrow M$ such that $\phi_t(\tilde{W}_0) = \tilde{W}_t$. Here $\tilde{W}_t$ is the image of the symplectic submanifold $W_t$. 
\end{lemma}
\begin{proof}
As the proof in the non-equivariant case only depends on the relative Poincaré lemma and on the symplectic neighborhoods theorem, it can be made equivariant. For more details, see Proposition~4 in~\cite{Auroux}.
\end{proof}

Let $\Symp^G(M)$ be the group of symplectomorphisms that commute with the action of $G$. Let $S$ be a compact, invariant, symplectic submanifold. Define the two subgroups
\[\Symp^G_{\id,N}(M,S)=\{\phi\in\Symp^G(M)~|~\phi|_{S}=\id,~d\phi|_{T_{S}M}=\id\}\]
\[\Symp^G_{\id,\near}(M,S)=\{\phi\in\Symp^G(M)~|~\phi=\id \text{~near~} S\}\]

\begin{prop}\label{prop:HomotopyEquivalenceSympN-Near}
Let $G$ be a compact group acting symplectically on a symplectic manifold $(M,\om)$ with finitely many orbit types. Let $S$ be a compact, invariant, symplectic submanifold. The inclusion $\Symp^G_{\id,\near}(M,S)\into\Symp^G_{\id,N}(M,S)$ is a homotopy equivalence.
\end{prop}
\begin{proof}
Since both spaces are homotopy equivalent to CW-complexes, it is enough to show that the inclusion $\Symp^G_{\id,\near}(M,S)\into\Symp^G_{\id,N}(M,S)$ is a weak homotopy equivalence.\\

We first prove a smooth version of the statement. By the Mostow–Palais theorem~\cite[Chapter 2 Theorem 10.1]{Bredon}, we can equivariantly embed $M$ into a finite-dimensional orthogonal representation $\R^q$. Let $U\subset\R^q$ be an invariant tubular neighborhood of $S$ and let $\rho:U\to S$ be the projection. Choose a smooth invariant function $\lambda:M\to[0,1]$ such that $S=\lambda^{-1}(0)$ and whose derivative is nowhere zero in directions normal to $S$ ($\lambda$ can be chosen to be Morse-Smale, in which case $S$ is a non-degenerate critical submanifold). Let $\alpha:\R\to[0,1]$ be a smooth increasing function such that $\alpha(u)=0$ for $u\leq 1$ and $\alpha(u)=1$ for $u\geq 2$.\\

Given $0<\epsilon<1$ let $\phi_\epsilon(x) = \alpha(\log_\epsilon(\lambda(x)))$ for $\lambda(x)>0$ and set $\phi_\epsilon(x)=1$ for $\lambda(x)=0$. By construction, $\phi_\epsilon$ is an invariant function such that $\phi_\epsilon(x)=1$ for $\lambda(x)\leq\epsilon^2$ and $\phi_\epsilon(x)=0$ for $\lambda(x)\geq \epsilon$.\\

Given $g\in \Diff^G_{\id,N}(M,S)$, define the homotopy
\[
g^t_\epsilon(x) = \rho\big(g(x) + t\phi_\epsilon(x)(x-g(x))\big)
\]
Because $g$ is $C^1$ close to the identity near $S$, an argument similar to the one presented in~\cite[Proposition 1.3]{Igusa-PseudoIsotopies} shows that we can find $\epsilon$ small enough so that $dg^t_\epsilon$ is invertible for all $t\in[0,1]$ and for all $\lambda(x)\leq\epsilon$. It follows that $g_\epsilon^t$ is a diffeomorphism and that this homotopy takes $\Diff^G_{\id,N}(M,S)$ into $\Diff^G_{\id,\near N}(M,S)$. Note that $g^t_\epsilon$ is equivariant for all $t$, as the $G$ action is linear on $\R^q$. Moreover, if $g$ belongs to a compact family $K\subset \Diff^G_{\id,N}(M,S)$, we can choose $\epsilon$ uniformly so that the family retracts into $\Diff^G_{\id,\near N}(M,S)$. Note that this equivariant retraction of $K$ is supported in an arbitrarily small (and fixed) tubular neighborhood $U$ of $S$.\\

We can further arrange the homotopy to that the resulting diffeomorphisms $g_\epsilon^t$ are the identity in nested increasing neighborhoods $V_\epsilon\subset V_{\epsilon+\delta}\subset U$. To see this, view $g_\epsilon^t$ as a 2-parameter family depending on $t\in[0,1]$ and on $\epsilon\in(0,1]$. Let $g_\epsilon:=g_\epsilon^1$ be the family obtained by taking $t=1$. We claim that $lim_{\epsilon\to0^+} g_\epsilon = g$ in the $C^1$ topology. Indeed, because $g$ is $C^1$ close to the identity near $S$, we can choose $U$ so that the norm $||g(x)-x||$ is bounded by $\lambda(x)A$ for some constant $A>0$. Similarly, we can bound $||dg_x-\id||$ by $\lambda(x)B$ for some $B>0$. It follows that we can bound $||dg-dg_\epsilon||$ by
\[
||dg-dg_\epsilon||\leq C\Big(||\id-dg||\phi_\epsilon + ||(\id-g)d\phi_\epsilon||\leq \lambda A\phi_\epsilon+\lambda B||d\phi_\epsilon ||  \Big).
\]
As $0\leq \phi_\epsilon\leq 1$, for all $x$, $\lim_{\epsilon\to0^+}\lambda(x) A\phi_\epsilon(x) = 0$. Similarly, as
\[
\left\|(d\phi_\epsilon)_x\right\| = \left\|\frac{d\alpha_x d\lambda_x}{\lambda(x)\log(\epsilon)}\right\| \leq \frac{D}{\lambda(x)|\log(\epsilon)|}
\]
for some $D>0$, we have
\[
\lim_{\epsilon\to0^+} \lambda(x) B||d\phi_\epsilon || 
\leq \lim_{\epsilon\to0^+} \lambda(x) B\frac{D}{\lambda(x)|\log(\epsilon)|} = 0,\quad \forall x.
\]
Consequently, the family $g_\epsilon$ extends $C^1$ continuously to $\epsilon=0$ by setting $g_0=g$. In particular, for $\epsilon>0$, $g_\epsilon=\id$ on the open neighbourhood $V_{\epsilon}:= \left\{x \in M ~|~ \lambda(x) < \epsilon^2 \right\}$. This gives us a homotopy $g_t$ such that $g_0=g$ and $g_1 \in\Diff^G_{\id,\near N}(M,S) $ which is smooth in the $C^1$-topology on $\Diff^G_{\id,N}(M,S)$. Finally, since the statement holds in the $C^1$ topology, by a standard approximation argument (as in \cite[Chapter 6, Theorem 4.2]{Bredon}) it holds in all $C^k$, $2\leq k\leq \infty$ topologies.\\

We now prove the statement in the symplectic category. Pick a symplectomorphism $g\in \Symp^G(M, T_N M)$ and choose $0<\epsilon<1$ small enough so that we have a smooth deformation retraction $g_\epsilon$ as above. For convenience, we reparametrize this homotopy by taking a parameter $t\in[0,1]$. Then $g_0=g$, $g_t=\id$ on $V_t$, and $g_t=g$ outside $U$. Let $\om_t=g_t^*\om$ be a 1-parameter family of $G$-invariant symplectic forms which equal $\om$ on $V_t$ by construction. 

Let $\sigma_t$ be a 1-form such that $\sigma_t=0$ on $V$ and $d\sigma_t=\om_t-\om_0$. This is possible by the equivariant Poincar\'e (Lemma \ref{lemma:RelativePoincare}). Consider the convex combination
\[\om_{t,s}=\om_0+s d\sigma_t.\]
There is some $t_0$ such that, for all $t<t_0$ and all $0\leq s\leq 1$, the form $\om_{t,s}$ is non-degenerate. Use equivariant Moser isotopy to get a rectifying equivariant isotopy $\psi_{t,s}$ such that $\psi_{t,s}$ is the identity on $V_t$ and $g_t^* \psi_{t,s}^*\om_0=\om_0$. Furthermore, as $\om_{t,s} = \om$ on $V_t$, we can choose $\psi_{t,s}$ such that $\psi_{s,t}=\id$ on some possibly smaller neighbourhood $V'_t \subset V_t$. Hence the homotopy $h_t$ is defined as $ h_t:= \psi_{1,t} \circ g_t$ for $t \in [0,t_0]$ is the required homotopy as at $t=0$, $h_0 = g$ and $h_{t_0}$ is a symplectomorphism that is the identity on a small neighbourhood $V_{t_0}:= \left\{x \in M ~|~ \lambda(x) < \epsilon_{t_0}^2 \right\}$. \\

This process can be done continuously along any compact family $K$ in the group $\Symp^G(M, T_N M)$, showing that the inclusion $\Symp^G(M, N)\into \Symp^G(M, T_N M)$ is a weak equivalence.
\end{proof}

Finally, we present a theorem of R. Palais that we repeatedly use in Chapters 3 and 6 to justify that some maps between infinite dimensional spaces are fibrations.\\

Let $X$ be a topological space with an action of a topological group $G$. We say $X$ admits local cross sections at $x_0 \in X$ if for there is a neighbourhood $U$ containing $x_0$ and a map $\chi: X \rightarrow G$ such that $\chi(u) \cdot x_0 = u$ for all $u \in U$.  We say $X$ admits local cross sections if this is true for all $x_0 \in X$. 

\begin{thm}[Local triviality of $G$-maps]\label{palais}
Let $X$, $Y$ be  topological spaces with a action of a topological group $G$. Let the $G$ action on $X$ admit local cross sections. Then any equivariant map $f$ from $Y$ to $X$ is locally trivial.
\end{thm}
\begin{proof}
This is Theorem A in \cite{Palais1960}.
\end{proof}

\chapter{Equivariant Automorphism Groups}\label{AppendixAutomorphisms}

In this section we determine the homotopy type of the equivariant automorphism groups that arise in lemma~\ref{gauge} and~lemma \ref{ngauge}.\\

Let $(M,\om)$ be a symplectic $4$-manifold equipped with an effective Hamiltonian action of the circle $S^1$. Let $S$ be an embedded symplectic sphere $S^1$ left invariant by the action. Further let $N(S)$ denote the $S^1$-invariant symplectic normal bundle of $S$ and let $\Aut^{S^1}(N(S))$ denote the space of $S^1$ equivariant fiberwise linear symplectic automorphisms of $N(S)$.
\begin{lemma}\label{Gauge(D)}
The group $\Aut^{S^1}(N(S))$ is homotopy equivalent to $S^1$.
\end{lemma}
\begin{proof}
If the self-intersection of $S$ is $m$, then the normal bundle $N(S)$ is equivariantly symplectomorphic to the standard K\"ahler bundle $\Oo(m)\to \CP^1$ where the $S^1$ action on $\Oo(m)$ lifts a standard circle action on $S^2$ fixing the two poles $\{P, Q\}$, and where the symplectic form has been appropriately rescaled. We trivialize $\Oo(m)$ on the upper and lower discs $D_P^2 = S^2-\{Q\}$, $D_Q^2 = S^2-\{P\}$ so that, over the cylinder $S^2-\{P,Q\}\simeq S^1\times (-1,1)$, the transition map is given by $(s,r,z_2)\mapsto (s,r,s^m z_2)$. On the trivializing discs $D_P^2$, $D_Q^2$, the circle action is then given by 
\[t\cdot(z_1,z_2) = (t^{-k} z_1,t^{p} z_2)\text{~and~} t\cdot(z_1,z_2) = (t^{k} z_1,t^{q} z_2)\]
where $\{-k,p\}$ and $\{k,q\}$ are, respectively, the weights at $P$ and $Q$ in the tangent and normal directions.\\

The restriction to $D_P^2$ of any $S^1$ fiberwise linear symplectic automorphisms $\phi\in \Aut(N(S))$ can be written
\[\phi((z_1,z_2))=(z_1,A_P(z_1)z_2)\]
for a unique map $A_P:D_P^2\to\Sp(2)$. The automorphism $\phi$ is $S^1$ invariant if, and only if, 
\begin{equation}\label{eq:ConditionForA}
t^pA_P(t^{-k}z_1)t^{-p} = A_P(z_1)
\end{equation}
for all $z_1\in D_P^2$ and all $t\in S^1$. In particular, $A_P(z_1)$ must commute with the action of any $t\in\Stab(z_1)$. Moreover, the values of $A_P$ along the orbit of $z_1$ is determined solely by its value at $z_1$ via the formula~\eqref{eq:ConditionForA}. Consequently, $A_P$ factors through a map $\bar A_P:D^2_P/S^1\to \Sp(2)$ verifying $\bar A_P([z_1])\in C(\Stab(z_1))\subset\Sp(2)$, where $C(G)$ denotes the centralizer of $G$ in $\Sp(2)$. As the pole $P$ is a fixed point of the circle action, $\bar A_P([0])\in C(S^1)=\U(1)$. Of course, a similar description holds for the restriction of the automorphism $\phi$ to the other chart $D_Q^2$. Over the cylinder $S^1\times(-1,1)$, the two maps $\bar A_P$ and $\bar A_Q$ are related by the formula $\bar A_Q([s,r])=s^m \bar A_P([s,r]) s^{-m}$, so that, when $z_1$ approaches the boundary of $D_P^2$, we must have
\[\lim_{|z_1|\to1} \bar A_P([z_1]) = \bar A_Q([0])\in\U(1).\]
We conclude that $\phi$ is uniquely determined by a map $\bar A:S^2/S^1\to\Sp(2)$ which takes one of the following three possible forms:
\begin{enumerate}
\item If $k=0$, that is, if $S$ is pointwise fixed, then $\bar A:S^2\to U(1)\subset\Sp(2)$.
\item If $k=1$, then $\bar A:[-1,1]\to \Sp(2)$ with $\bar A([\pm1])\in\U(1)$.
\item If $k\geq 2$, then $\bar A:[-1,1]\to C(\Z_k)\subset\Sp(2)$ with $\bar A([\pm1])\in\U(1)$.
\end{enumerate}
Conversely, any such map $\bar A$ defines a unique equivariant automorphism of $N(S)$, so that $\Aut^{S^1}(N(S))$ is identified with the corresponding space $\Maps_k$ of all such maps. Clearly, for $k=0$. the space $\Maps_0$ retracts onto $U(1)$. For $k=1$, the deformation retraction of $\Sp(2)$ onto $\U(1)$ given by the polar decomposition induces a homotopy equivalence $\Maps_1\simeq \U(1)$. Finally, because the polar decomposition is equivariant with respect to the adjoint action of any compact subgroup $G\subset\Sp(2)$, the centralizer $C(\Z_k)\subset\Sp(2)$, $k\geq 2$, deformation retracts onto $\U(1)$, showing  that $\Maps_k\simeq \U(1)$ for all $k\geq 2$.
\end{proof}

Let $C \vee F\subset (M,\om)$ be two embedded symplectic spheres which are invariant under the circle action and which intersectect $\om$-orthogonally at $P \in M$. Let $TM|_{C\vee F}$ be the tangent bundle of $M$ restricted to $C\vee F$, and let $N(C\vee F)\subset TM|_{C\vee F}$ be union of the symplectic normal bundles $N(C)$ and $N(F)$.
\begin{lemma} \label{Gauge(N(D))}
The group $\Aut^{S^1}(N(C \vee \overline{F}))$ of $S^1$-equivariant fiberwise symplectic automorphisms of the the symplectic normal bundle which are equal to the identity in a neighbourhood of the wedge point is contractible. 
\end{lemma}

\begin{proof}
Restricting any automorphisms $\phi\in\Aut^{S^1}(N(C \vee \overline{F}))$ to the normal bundle $N(C)$, the same argument as in the proof of lemma~\ref{Gauge(D)} show that $\phi$ is given by a map $\bar A\in\Maps_k$ which must satisfy the extra condition $\bar A([z_1])=\id$ near $[P]$. The subspace of all these maps is contractible. As the same holds for the restriction of $\phi$ to $F$, the conclusion follows.
\end{proof}

\chapter[Equivariant Gompf argument]{\texorpdfstring{$J$}{J}-Holomorphic configurations and equivariant Gompf argument}

\begin{thm}\label{transverse}
Let $G$ be a compact group. Let A and B be two $G$-invariant symplectic spheres in a 4-dimensional symplectic manifold $(M,\omega)$ intersecting $\omega$-orthogonally at a unique fixed point $p$ for the $G$ action. Then there exists an invariant $J \in \mathcal{J}_\om^{G}$ for which both A and B are $J$-holomorphic. Here $\mathcal{J}_\om^{G}$ denotes the space of $G$ invariant compatible almost complex structures on $M$. 
\end{thm}

\begin{proof}
The proof of Lemma~A.1 in~\cite{Evans} can be made equivariant by standard averaging arguments.
\end{proof}

\begin{thm}{(Equivariant Gompf Argument)} \label{gmpf}
Let $G$ be a compact group. Let $p$ be a fixed point for the action. Let $A$,$B$ and $\overline{A}$ be $G$-invariant symplectic spheres in a 4-dimensional symplectic manifold $(M,\omega)$ such that 
\begin{itemize}
    \item $A \cap B = \{p\}$ and the intersection at $p$ is transverse
    \item $\overline{A} \cap B = \{p\}$ and the intersection at $p$ is $\om$-orthogonal.
\end{itemize} 
Then there exists an $S^1$-equivariant isotopy $A_t$ of $A$  such that $A_t$ intersects $B$ transversely at $p$ for all $t$,  $A_1$ = $\overline{A}$ in a small neighbourhood of $p$ and the curve $A_1$ agrees with $A$ outside some neighbourhood of $p$.
\end{thm}
\begin{proof}Since this is a local problem, we can work in a trivialising chart in $\R^4$ in which the action is linear. Let $B^{\bot_\om}$ be the symplectic orthogonal to $B$. We can assume the image of  $B$ to be the two plane in $\R^4$ given by $(0,0,x,y)$, and $B^{\bot_\om}$ to the given by the plane $(x,y,0,0)$. As $\overline{A}$ is $\om$-orthogonal to $B$, we can choose the neighbourhood such that $\overline{A} = B^{\bot_\om}$ near $p$. As $A$ is transverse to $B$ at $p$, we can assume its image is given by the graph of function (which we also call A) $A:(f,g): \R^2 \rightarrow \R^2$.
\\

Next we observe that given a function $A:=(f,g): \R^2 \rightarrow \R^2$, the graph of $A$ is a symplectic (for the standard form) submanifold of $\R^4$ iff $\left\{f,g\right\} > -1$. This can be proven from  a direct computation.  We will construct an isotopy of graphs of function of the form $A_t:= \alpha_t(r^2) A$ where $\alpha_t$ is a bump function depending only on the radius squared (for a fixed $G$ invariant  metric) in $\R^2$, and such that \begin{itemize}
    \item $A_0 = A$,
    \item $A_1 = 0$  near (0,0),
    \item $A_t = A$ outside of some neighbourhood of the origin,
    \item $A_t$ is symplectic for all $t$.
\end{itemize}  
Note that as $\alpha_t$ is depends on the radius for a fixed $G$ invariant metric, $A_t$ is also $G$ invariant.
\\

Define $E = g\left\{f,r^2\right\} + f\left\{r^2,g\right\}$. Using the fact that $r^2(0,0) = 0$ and that $(r^2)^\prime (0,0) = 0$ we see that $E(0,0) = 0$ and $\frac{\partial}{\partial r}E(0,0) = 0$. By the intermediate value theorem, there exists $c > 0$, $\epsilon > 0$ and $u > 0$ such that on the ball of radius $u + \epsilon$ around the origin $B(0,u+\epsilon)$  we have $E(x) \geq -c r^2(x)$. Choose $\delta$ such that  on $B(0,u+\epsilon)$, $1 + \left\{f,g\right\} > \delta > 0$
\\

Pick $\alpha: \R \rightarrow \R$ satisfying the following properties:

\begin{itemize}
\item  $\alpha(r^2) = 1$ for $r^2 \geq u$,
\item $\alpha(r) = 0$ for r near $0$,
\item $\alpha^\prime(r^2) \leq \frac{\delta}{2cr^2} < \frac{1 + \{f,g\}}{2cr^2}$.
\end{itemize} 
Define $\alpha_t:= (1-t) + t \alpha(r^2)$ and $A_t := \alpha_t A$. To show that $A_t$ is symplectic for all $0 \leq t \leq 1$ we need to check that  $\left\{\alpha_tf,\alpha_tg\right\} > -1$ for all $0 \leq t \leq 1$. 
In the neighbourhood $B(u)$ we have
\begin{equation*}
1 + \left\{\alpha_t f,\alpha_t g\right\} = \underbrace{1 + \alpha_t^2 \{f,g\}}_{\geq \delta} + \underbrace{\alpha_t\alpha_t^\prime E}_{\geq \frac{-\delta}{2}} \geq 0
\end{equation*}
The inequality $1 + \alpha_t^2 \{f,g\} \geq \delta$ follows from the definition of $\delta$ and from noting that $0 \geq \alpha_t \geq 1$.  $\alpha_t\alpha_t^\prime E \geq \frac{-\delta}{2}$ follows from the  inequality 
\begin{align*}
    \alpha_t\alpha_t^\prime E &\geq \alpha_t\alpha_t^\prime (-cr^2) \\
    &\geq -\alpha_t\frac{\delta}{2cr^2} (cr^2) \\
    &\geq -\alpha_t \frac{\delta}{2} \\
    &\geq \frac{-\delta}{2}
\end{align*}
Thus in the neighbourhood $B(u)$ we have the inequality $1 + \left\{\alpha_t f,\alpha_t g\right\} > 0$ for all  t. Outside of $B(u)$, the derivative $\alpha_t^\prime$ is identically 0 and $\alpha_t \equiv 1$. Hence $\alpha_t\alpha_t^\prime E \equiv 0$ outside $B(u)$ and $1 + \left\{\alpha_t f,\alpha_t g\right\} = 1 + \alpha_t^2 \{f,g\} + \cancelto{0}{\alpha_t\alpha_t^\prime E} = 1 + \{f,g\} > 0$ outside of $B(u)$.\\

Finally we note that $A_1 =0$ in a neighbourhood of $(0,0)$ and it equals $A$ outside the ball of radius $u$ around the origin, thus proving the claim.
\end{proof}

\chapter{Alexander-Eells isomorphism}\label{Appendix-Alexander-Eells}

In this appendix, we prove the version of the Alexander-Eells isomorphism used in Section~\ref{subsection:CohomologyModule}. We first recall an isomorphism between the homology of a submanifold $Y\subset X$ and the homology of its complement $X-Y$ that is reminiscent of the Alexander-Pontryagin duality in the category of oriented, finite dimensional manifolds. This isomorphism, due to J. Eells, exists whenever the submanifold $Y$ is co-oriented and holds, in particular, for infinite dimensional Fréchet manifolds. We then give a geometric realization of this isomorphism in the special case $Y$ and $X-Y$ are orbits of a continuous action $G\times X\to X$ satisfying some mild assumptions. We closely follow the exposition in Eells~\cite{Eells}.\\

Let $X$ be a manifold, possibly infinite dimensional, and let $Y$ be a co-oriented submanifold of positive and finite codimension $p$. As explained in~\cite{Eells}, there exists an isomorphism of singular cohomology groups
\[\phi:H^{i}(Y)\to H^{i+p}(X, X-Y)\]
called the Alexander-Eells isomorphism. We define the fundamental class (Thom class) of the pair $(X,Y)$ as $u=\phi(1)\in H^{p}(X, X-Y)$.
\begin{prop}[Eells~\cite{Eells}, p. 113]
The pairing
\[\begin{aligned}
H^{*}(Y)\otimes H^{*}(X,X-Y)&\to H^{*}(X,X-Y)\\
y\otimes x &\mapsto y\cup x
\end{aligned}\]
makes $H^{*}(X,X-Y)$ into a free $H^{*}(Y)$-module of rank one, generated by $u$.
\end{prop}

Let $\phi_{*}:H_{i+p}(X,X-Y)\to H_{i}(Y)$ be the dual of the Alexander-Eells isomorphism $\phi^{*}=\phi$. By definition, we have
\[\phi_{*}(a)=u\cap a\]

Suppose a topological group $G$ acts continuously on $X$ (on the left), leaving $Y$ invariant, and in such a way that both $X-Y$ and $Y$ are homotopy equivalent to orbits. We have continuous maps $\mu:G\times (X,X-Y)\to (X,X-Y)$, ~$\mu:G\times (X-Y)\to (X-Y)$, and $\mu:G\times Y\to Y$ inducing $H_{*}(G)$-module structures on $H_{*}(X,X-Y)$, $H_{*}(X-Y)$ and $H_{*}(Y)$. We write $\mu_{*}(c\otimes a)=c*a$ for the action of $c\in H_{i}(G)$.

\begin{lemma}\label{lemma:AlexanderEellsGmodule}
In this situation, the Alexander-Eells isomorphism preserves the $H_{*}(G)$-module structure, that is, the following diagram is commutative:
\[
\begin{tikzcd}
H_{*}(G)\otimes H_{*}(X,X-Y) \arrow{r}{\mu_{*}} \arrow[swap]{d}{1\otimes\phi_{*}} & H_{*}(X,X-Y) \arrow{d}{\phi_{*}} \\
H_{*}(G)\otimes H_{*}(Y) \arrow{r}{\mu_{*}} & H_{*}(Y)
\end{tikzcd}
\]
Thus for any $a\in H_{i+p}(X,X-Y)$, $c\in H_{i}(G)$, we have $\phi_{*}(c*a) = c*\phi_{*}(a)$.
\end{lemma}
\begin{proof}
We first note that if $u$ is the fundamental class of the pair $(X,Y)$, then $\mu^{*}(u)=1\otimes u\in H^{0}(G)\otimes H^{p}(X,X-Y)$, because $H^{i}(X,X-Y)=0$ for all $i<p$. The cap product is a bilinear pairing
\[
\cap: H^{*}\big(G\times (X,X-Y)\big)\otimes H_{*}\big(G\times (X,X-Y)\big) \to H_{*}\big(G\times X\big)
\]
that is adjoint to the cup product. Using the relative form of K\"unneth theorem, this bilinear map defines a pairing
\[
\cap:\big(H^{*}(G)\otimes H^{*}(X,X-Y)\big)\otimes\big(H_{*}(G)\otimes H_{*}(X,X-Y)\big)\to H_{*}(G)\otimes H_{*}(X)
\]
Let's write $j^{*}:H^{*}(X,X-Y)\to H^{*}(X)$.  Then, for any $c\otimes x\in H^{*}(G)\otimes H^{*}(X,X-Y)$, and any $b\otimes a\in H_{*}(G)\otimes H_{*}(X,X-Y)$ we have
\[\begin{aligned}
\big\langle c\otimes j^{*}x, \mu^{*}(u)\cap (b\otimes a)\big\rangle 
&= \big\langle (c\otimes x)\cup\mu^{*}(u),b\otimes a\big\rangle\\
&= \big\langle (c\otimes x)\cup(1\otimes u),b\otimes a\big\rangle\\
&= \big\langle (c\cup1)\otimes (x\cup u),b\otimes a\big\rangle\\
&= \big\langle c, b\big\rangle \big\langle (x\cup u),a\big\rangle\\
&= \big\langle c,b\big\rangle \big\langle j^{*}x, u\cap a\big\rangle\\
&=\big\langle c\otimes j^{*}x, b\otimes (u\cap a)\big\rangle
\end{aligned}\]
It follows that $\mu^{*}(u)\cap (b\otimes a) = b\otimes (u\cap a) = b\otimes \phi_{*}(a)$. We then compute
\begin{multline*}
\phi_{*}(\mu_{*}(c\otimes x)) 
=u\cap\mu_{*}(c\otimes x) 
=\mu_{*}\big(\mu^{*}(u)\cap(c\otimes x)\big)\\
=\mu_{*}\big(c\otimes\phi_{*}(x)\big)
=\mu_{*}\big((1\otimes \phi_{*})(c\otimes x)\big)
\end{multline*}
which is the desired relation.
\end{proof}
Since $H_{i}(X,X-Y)=0$ for $i<p$, the Universal Coefficient Theorem yields a canonical isomorphism $\beta:H^{p}(X,X-Y)\to \Hom(H_{p}(X,X-Y);\Z)$. If $Y$ is connected, $H_{p}(X,X-Y)$ is of rank one. In this case, define $a_{u}\in H_{p}(X,X-Y)$ as the unique class such that $\beta(u)a_{u}=1$. Suppose the Leray-Hirsch theorem applies to the evaluation fibration $G\to Y$. Then, $H_{*}(X,X-Y)$ becomes a $H_{*}(Y)$-module by identifying $H_{*}(Y)$ with $1\otimes H_{*}(Y)\subset H_{*}(G)$ and by setting
\[b*x := [1\otimes b]*x\]
\begin{thm}\label{thm:AlexanderEells}
Under the above assumptions, the isomorphism $\psi_{*}=\phi_{*}^{-1}:H_{i}(Y)\to H_{i+p}(X,X-Y)$ is given by $\psi_{*}(y)=y*a_{u}$. Thus $H_{*}(X,X-Y)$ is generated by $a_{u}$ as a $H_{*}(Y)$-module.
\end{thm}
\begin{proof}
Since
\[\langle 1, \phi_{*}(a_{u})\rangle
=\langle 1, u\cap a_{u}\rangle 
=\langle 1\cup u, a_{u}\rangle
=\langle u,a_{u}\rangle=1\]
it follows that $\phi_{*}(a_{u})=1$, which is equivalent to $\psi_{*}(1)=a_{u}$. Since $\phi_{*}$ is an isomorphism of $H_{*}(G)$-modules, its inverse is also an isomorphism of $H_{*}(G)$-modules. We can then write
\[\psi_{*}(y) = \psi_{*}(y*1) = y*\psi_{*}(1) = y*a_{u}\]
\end{proof}
From the naturality of the connecting homomorphism $\partial$ in the long exact sequence of the pair $(X, X-Y)$, we get
\begin{cor}\label{cor:AlexanderEells}
Suppose in addition to the above hypotheses that $X$ is contractible. Then the isomorphism 
\[\lambda_{*}=\partial\circ\psi_{*}: H_{i}(Y)\to H_{i+p-1}(X-Y)\]
is given by $\lambda_{*}(y) = y*x_{u}$, where $x_{u}=\partial a_{u}$.
\end{cor}

We now apply the Alexander-Eells isomorphism to the situation of Section~\ref{section:TwoOrbits}. Let $G$ be the centralizer $\Symp_{h}^{S^1}(\SSS,\oml)$, let $X=\jsom$ be the contractible space of invariant, compatible,  almost-complex structures, and let $Y$ be the codimension 2 stratum $\jsom\cap U_{m'}$. Let $p_m$ denote the map 
\[p_m: \Symp^{S^1}_h(\SSS,\oml)\rightarrow \Symp^{S^1}_h(\SSS,\oml)/\T_m \simeq \jsom \cap U_{m}\]
Further, let $y_1$ and $y_2$ denote the generators of the Pontryagin algebra $H_{*}(T^{2}_{m'};k)$.  Note that under the inclusion $i_m: \T_m \hookrightarrow \Symp_{h}^{S^1}(\SSS,\oml)$, the elements $y_1$ and $y_2$ correspond to the circle actions $(1,0;m)$ and $(0,1;m)$ respectively. 
The connecting isomorphism $\partial:H_{2}(\jsom, \jsom\cap U_{m})\to H_{1}(\jsom\cap U_{m})$ maps the generator $a_{u}$ to the link of $\jsom\cap U_{m'}$ in $\jsom\cap U_{m}$, that is, to the loop $p_{m}(y_{2})$ generated by the action of $y_{2}\subset G$. Consequently, Corollary~\ref{cor:AlexanderEells} immediately implies the following geometric description of the Alexander-Eells isomorphism.
\begin{prop}\label{prop:AlexanderEellsGeometric}
The Alexander-Eells isomorphism
\[\lambda_{*}:H_{p}(\jsom\cap U_{m'})\to H_{p+1}(\jsom\cap U_{m})\]
is given by
\[\lambda_{*}(y)= y*p_{m}(y_{2})=y*y_{2}*1=[y_{2}\otimes y]*1=\mu_{m}\big([y_{2}\otimes y]\otimes 1\big)=p_{m}\big([y_{2}\otimes y]\big)\]
In particular, if $p_{m'}(\tilde y) = y\in H_{*}(Y)$, then $\lambda_{*}(y)=\tilde y*y_{2}*1=p_{m}(\tilde y* y_{2})$. 
\end{prop}


\bibliographystyle{amsalpha}
\bibliography{Bibliography}

\providecommand{\bysame}{\leavevmode\hbox to3em{\hrulefill}\thinspace}
\providecommand{\MR}{\relax\ifhmode\unskip\space\fi MR }
\providecommand{\MRhref}[2]{%
  \href{http://www.ams.org/mathscinet-getitem?mr=#1}{#2}
}
\providecommand{\href}[2]{#2}
\begin{thebibliography}{AMR88}

\bibitem[Abr98]{abreu}
Miguel Abreu, \emph{Topology of symplectomorphism groups of {$S^2\times S^2$}},
  Invent. Math. \textbf{131} (1998), no.~1, 1--23. \MR{1489893}

\bibitem[AE19]{AE}
S\'{\i}lvia Anjos and Sinan Eden, \emph{The homotopy {L}ie algebra of
  symplectomorphism groups of 3-fold blowups of {$(S^2\times S^2,\sigma_{\rm
  std}\oplus\sigma_{\rm std})$}}, Michigan Math. J. \textbf{68} (2019), no.~1,
  71--126. \MR{3934605}

\bibitem[AG04]{AG}
S\'{\i}lvia Anjos and Gustavo Granja, \emph{Homotopy decomposition of a group
  of symplectomorphisms of {$S^2\times S^2$}}, Topology \textbf{43} (2004),
  no.~3, 599--618. \MR{2041632}

\bibitem[AGK09]{AGK}
Miguel Abreu, Gustavo Granja, and Nitu Kitchloo, \emph{Compatible complex
  structures on symplectic rational ruled surfaces}, Duke Math. J. \textbf{148}
  (2009), no.~3, 539--600. \MR{2527325}

\bibitem[AM00]{MR1775741}
Miguel Abreu and Dusa McDuff, \emph{Topology of symplectomorphism groups of
  rational ruled surfaces}, J. Amer. Math. Soc. \textbf{13} (2000), no.~4,
  971--1009. \MR{1775741}

\bibitem[AMR88]{mardsen}
R.~Abraham, J.~E. Marsden, and T.~Ratiu, \emph{Manifolds, tensor analysis, and
  applications}, second ed., Applied Mathematical Sciences, vol.~75,
  Springer-Verlag, New York, 1988. \MR{960687}

\bibitem[AP13]{AP}
S\'{\i}lvia Anjos and Martin Pinsonnault, \emph{The homotopy {L}ie algebra of
  symplectomorphism groups of 3-fold blow-ups of the projective plane}, Math.
  Z. \textbf{275} (2013), no.~1-2, 245--292. \MR{3101807}

\bibitem[Aud04]{Audin}
Mich\`ele Audin, \emph{Torus actions on symplectic manifolds}, revised ed.,
  Progress in Mathematics, vol.~93, Birkh\"{a}user Verlag, Basel, 2004.
  \MR{2091310}

\bibitem[Aur97]{Auroux}
D.~Auroux, \emph{Asymptotically holomorphic families of symplectic
  submanifolds}, Geom. Funct. Anal. \textbf{7} (1997), no.~6, 971--995.
  \MR{1487750}

\bibitem[Ban78]{Banyaga78}
Augustin Banyaga, \emph{Sur la structure du groupe des diff\'{e}omorphismes qui
  pr\'{e}servent une forme symplectique}, Comment. Math. Helv. \textbf{53}
  (1978), no.~2, 174--227. \MR{490874}

\bibitem[Ban88]{Banyaga-Isomorphism}
\bysame, \emph{On isomorphic classical diffeomorphism groups. {II}}, J.
  Differential Geom. \textbf{28} (1988), no.~1, 23--35. \MR{950553}

\bibitem[Bre72]{Bredon}
Glen~E. Bredon, \emph{Introduction to compact transformation groups}, Pure and
  Applied Mathematics, Vol. 46, Academic Press, New York-London, 1972.
  \MR{0413144}

\bibitem[BtD85]{B-tD}
Theodor Br\"{o}cker and Tammo tom Dieck, \emph{Representations of compact {L}ie
  groups}, Graduate Texts in Mathematics, vol.~98, Springer-Verlag, New York,
  1985. \MR{781344}

\bibitem[Cau92]{error}
Robert Cauty, \emph{Sur les ouverts des {CW}-complexes et les fibr\'{e}s de
  {S}erre}, Colloq. Math. \textbf{63} (1992), no.~1, 1--7. \MR{1157891}

\bibitem[Che]{ChenUnpub}
Weimin Chen, \emph{Some rigidity results of symplectic group actions}, Personal
  communication.

\bibitem[CK19]{Liat}
River Chiang and Liat Kessler, \emph{Cyclic actions on rational ruled
  symplectic four-manifolds}, Transform. Groups \textbf{24} (2019), no.~4,
  987--1000. \MR{4038082}

\bibitem[CP]{Zn_symp}
Pranav Chakravarthy and Martin Pinsonnault, \emph{Centralizers of hamiltonian
  cyclic group actions on rational ruled surfaces}, Trans. Amer. Math. Soc. 377
  (2024), 7945--7989.

\bibitem[Del88]{Delzant}
Thomas Delzant, \emph{Hamiltoniens p\'{e}riodiques et images convexes de
  l'application moment}, Bull. Soc. Math. France \textbf{116} (1988), no.~3,
  315--339. \MR{984900}

\bibitem[Dro71]{Dror-WhiteheadTheorem}
Emmanuel Dror, \emph{A generalization of the {W}hitehead theorem}, Symposium on
  {A}lgebraic {T}opology ({B}attelle {S}eattle {R}es. {C}enter, {S}eattle,
  {W}ash., 1971), 1971, pp.~13--22. Lecture Notes in Math., Vol. 249.
  \MR{0350725}

\bibitem[Eel61]{Eells}
James Eells, Jr., \emph{Alexander-{P}ontrjagin duality in function spaces},
  Proc. {S}ympos. {P}ure {M}ath., {V}ol. {III}, American Mathematical Society,
  Providence, R.I., 1961, pp.~109--129. \MR{0125580}

\bibitem[Eva11]{Evans}
Jonathan~David Evans, \emph{Symplectic mapping class groups of some {S}tein and
  rational surfaces}, J. Symplectic Geom. \textbf{9} (2011), no.~1, 45--82.
  \MR{2787361}

\bibitem[FOT08]{dgcRationalHomotopy}
Yves F\'{e}lix, John Oprea, and Daniel Tanr\'{e}, \emph{Algebraic models in
  geometry}, Oxford Graduate Texts in Mathematics, vol.~17, Oxford University
  Press, Oxford, 2008. \MR{2403898}

\bibitem[Gro85]{Gr}
Mikhael Gromov, \emph{Pseudo holomorphic curves in symplectic manifolds},
  Inventiones mathematicae \textbf{82} (1985), no.~2, 307--347.

\bibitem[Gua16]{Guadagni}
Roberta Guadagni, \emph{Symplectic neighborhood of crossing divisors}, arXiv
  (2016).

\bibitem[Hat02]{Ha}
Allen Hatcher, \emph{Algebraic topology}, Cambridge University Press,
  Cambridge, 2002. \MR{1867354}

\bibitem[Hir94]{Hi}
Morris~W. Hirsch, \emph{Differential topology}, Graduate Texts in Mathematics,
  vol.~33, Springer-Verlag, New York, 1994, Corrected reprint of the 1976
  original. \MR{1336822}

\bibitem[Igl91]{Iglesias1991}
Patrick Igl\'{e}sias, \emph{Les {${\rm SO}(3)$}-vari\'{e}t\'{e}s symplectiques
  et leur classification en dimension {$4$}}, Bull. Soc. Math. France
  \textbf{119} (1991), no.~3, 371--396. \MR{1125672}

\bibitem[Igu88]{Igusa-PseudoIsotopies}
Kiyoshi Igusa, \emph{The stability theorem for smooth pseudoisotopies},
  $K$-Theory \textbf{2} (1988), no.~1-2, vi+355. \MR{972368}

\bibitem[Kar99]{Karshon}
Yael Karshon, \emph{Periodic {H}amiltonian flows on four-dimensional
  manifolds}, Mem. Amer. Math. Soc. \textbf{141} (1999), no.~672, viii+71.
  \MR{1612833}

\bibitem[Kar02]{finTori}
\bysame, \emph{{Maximal tori in the symplectomorphism groups of Hirzebruch
  surfaces}}, arXiv Mathematics e-prints (2002), math/0204347.

\bibitem[KKP07]{KKP}
Yael Karshon, Liat Kessler, and Martin Pinsonnault, \emph{A compact symplectic
  four-manifold admits only finitely many inequivalent toric actions}, J.
  Symplectic Geom. \textbf{5} (2007), no.~2, 139--166. \MR{2377250}

\bibitem[KKP15]{KK-Counting}
\bysame, \emph{Counting toric actions on symplectic four-manifolds}, C. R.
  Math. Acad. Sci. Soc. R. Can. \textbf{37} (2015), no.~1, 33--40. \MR{3364303}

\bibitem[Kod05]{Ko}
Kunihiko Kodaira, \emph{Complex manifolds and deformation of complex
  structures}, english ed., Classics in Mathematics, Springer-Verlag, Berlin,
  2005, Translated from the 1981 Japanese original by Kazuo Akao. \MR{2109686}

\bibitem[KT14]{KT-TallComplexityOne}
Yael Karshon and Susan Tolman, \emph{Classification of {H}amiltonian torus
  actions with two-dimensional quotients}, Geom. Topol. \textbf{18} (2014),
  no.~2, 669--716. \MR{3180483}

\bibitem[LLW22]{LiLiWu2022}
Jun Li, Tian-Jun Li, and Weiwei Wu, \emph{Symplectic {$(-2)$}-spheres and the
  symplectomorphism group of small rational 4-manifolds {II}}, Trans. Amer.
  Math. Soc. \textbf{375} (2022), no.~2, 1357--1410. \MR{4369250}

\bibitem[LM96]{MR1426534}
Fran\c{c}ois Lalonde and Dusa McDuff, \emph{The classification of ruled
  symplectic {$4$}-manifolds}, Math. Res. Lett. \textbf{3} (1996), no.~6,
  769--778. \MR{1426534}

\bibitem[LT97]{Lerman-Tolman}
Eugene Lerman and Susan Tolman, \emph{Hamiltonian torus actions on symplectic
  orbifolds and toric varieties}, Trans. Amer. Math. Soc. \textbf{349} (1997),
  no.~10, 4201--4230. \MR{1401525}

\bibitem[McC01]{McCleary}
John McCleary, \emph{A user's guide to spectral sequences}, second ed.,
  Cambridge Studies in Advanced Mathematics, vol.~58, Cambridge University
  Press, Cambridge, 2001. \MR{1793722}

\bibitem[MRS88]{Montaldi}
J.~A. Montaldi, R.~M. Roberts, and I.~N. Stewart, \emph{Periodic solutions near
  equilibria of symmetric {H}amiltonian systems}, Philos. Trans. Roy. Soc.
  London Ser. A \textbf{325} (1988), no.~1584, 237--293. \MR{946385}

\bibitem[MS12]{McD}
Dusa McDuff and Dietmar Salamon, \emph{{$J$}-holomorphic curves and symplectic
  topology}, second ed., American Mathematical Society Colloquium Publications,
  vol.~52, American Mathematical Society, Providence, RI, 2012. \MR{2954391}

\bibitem[MS17]{MS}
\bysame, \emph{Introduction to symplectic topology}, third ed., Oxford Graduate
  Texts in Mathematics, Oxford University Press, Oxford, 2017. \MR{3674984}

\bibitem[MT10]{McDuff-Tolman-MassLinear1}
Dusa McDuff and Susan Tolman, \emph{Polytopes with mass linear functions. {I}},
  Int. Math. Res. Not. IMRN (2010), no.~8, 1506--1574. \MR{2628835}

\bibitem[Pal60]{Palais1960}
Richard~S. Palais, \emph{Local triviality of the restriction map for
  embeddings.}, Commentarii mathematici Helvetici \textbf{34} (1960), 305--312.

\bibitem[Pic90]{Hans}
R.~A. Piccinini (ed.), \emph{Groups of self-equivalences and related topics},
  Lecture Notes in Mathematics, vol. 1425, Springer-Verlag, Berlin, 1990.
  \MR{1070570}

\bibitem[Pin08a]{P-MaxTori}
Martin Pinsonnault, \emph{Maximal compact tori in the {H}amiltonian group of
  4-dimensional symplectic manifolds}, J. Mod. Dyn. \textbf{2} (2008), no.~3,
  431--455. \MR{2417479}

\bibitem[Pin08b]{P-com}
\bysame, \emph{Symplectomorphism groups and embeddings of balls into rational
  ruled 4-manifolds}, Compos. Math. \textbf{144} (2008), no.~3, 787--810.
  \MR{2422351}

\bibitem[Smo92]{Smolentsev}
N.~K. Smolentsev, \emph{On the curvature of a space of associated metrics on a
  symplectic manifold}, Sibirsk. Mat. Zh. \textbf{33} (1992), no.~1, 132--139,
  220. \MR{1165685}

\bibitem[SW84]{Serrefib}
Mark Steinberger and James West, \emph{Covering homotopy properties of maps
  between {C}.{W}. complexes or {ANR}s}, Proc. Amer. Math. Soc. \textbf{92}
  (1984), no.~4, 573--577. \MR{760948}

\bibitem[Wei78]{Weinstein78}
A.~Weinstein, \emph{A universal phase space for particles in {Y}ang-{M}ills
  fields}, Lett. Math. Phys. \textbf{2} (1977/78), no.~5, 417--420. \MR{507025}

\end{thebibliography}


\end{document}